\documentclass[11pt,letterpaper]{amsart} 
\usepackage[margin=0.9in]{geometry}
\usepackage[rgb]{xcolor}
\usepackage{tikz}
\usepackage[bottom]{footmisc}
\usepackage{nicematrix}
\usepackage{verbatim}
\usepackage{mathtools}
\usepackage{amsmath}
\usepackage{amsxtra}
\usepackage{amscd}
\usepackage{amsthm}
\usepackage{amsfonts}
\usepackage{amssymb}
\usepackage{eucal}
\usepackage{float}
\usepackage{rotating}
\usepackage{enumerate}
\usepackage{mathrsfs}
\usepackage{graphicx}
\usepackage{stmaryrd}
\usepackage{tikz-cd}
\usepackage{mathdots}
\usetikzlibrary{arrows}
\usetikzlibrary{cd}
\usepackage{ytableau}
\usepackage{tikz-cd}
\usetikzlibrary{arrows}
\usetikzlibrary{cd}
\usepackage{caption}
\tikzcdset{every label/.append style = {font = \small}}
\allowdisplaybreaks

\usepackage{latexsym,todonotes}

\usepackage[numbers,sort&compress]{natbib}
\PassOptionsToPackage{hyphens}{url}
\usepackage[linktoc=all]{hyperref}
\hypersetup{colorlinks,linkcolor=blue,urlcolor=orange,citecolor=blue}
\usepackage[all, cmtip]{xy}

\usepackage{stmaryrd}
\usepackage{color}

\newtheorem{thm}{Theorem} [section]
\newtheorem{cor}[thm]{Corollary}
\newtheorem{lem}[thm]{Lemma}

\newtheorem{remark}[thm]{Remark}

\theoremstyle{definition}
\newtheorem{definition}[thm]{Definition}
\newtheorem{example}[thm]{Example}

\theoremstyle{remark}

\numberwithin{equation}{section}

\begin{document}

\newcommand{\thmref}[1]{Theorem~\ref{#1}}
\newcommand{\secref}[1]{Section~\ref{#1}}
\newcommand{\lemref}[1]{Lemma~\ref{#1}}
\newcommand{\propref}[1]{Proposition~\ref{#1}}
\newcommand{\corref}[1]{Corollary~\ref{#1}}
\newcommand{\remref}[1]{Remark~\ref{#1}}
\newcommand{\eqnref}[1]{(\ref{#1})}
\newcommand{\exref}[1]{Example~\ref{#1}}

 \newcommand{\calA}{\mathcal{A}}
  \newcommand{\calD}{\mathcal{D}}
 \newcommand{\fraksl}{\mathrm{\mathfrak{sl}}}
 \def\calUP{{\mathcal {UP}}}
 \newcommand{\GSp}{\mathrm{GSp}}
 \newcommand{\bbP}{\mathbb{P}}
 \newcommand{\PGSp}{\mathrm{PGSp}}
\newcommand{\PGSO}{\mathrm{PGSO}}
\newcommand{\PGO}{\mathrm{PGO}}
\newcommand{\SO}{\mathrm{SO}}
\newcommand{\GO}{\mathrm{GO}}
\newcommand{\GSO}{\mathrm{GSO}}
\newcommand{\Spin}{\mathrm{Spin}}
\newcommand{\GSpin}{\mathrm{GSpin}}
\newcommand{\Sp}{\mathrm{Sp}}
\newcommand{\PGL}{\mathrm{PGL}}
\newcommand{\GL}{\mathrm{GL}}
\newcommand{\SL}{\mathrm{SL}}
\newcommand{\U}{\mathrm{U}}
\newcommand{\ind}{\mathrm{ind}}
\newcommand{\Ind}{\mathrm{Ind}}
\newcommand{\im}{\mathrm{im}}
\renewcommand{\ker}{\mathrm{ker}}
 \newcommand{\triv}{\mathrm{triv}}
  \newcommand{\std}{\mathrm{std}}
 \newcommand{\Ad}{\mathrm{Ad}}
  \newcommand{\ad}{\mathrm{ad}}
  \newcommand{\Tr}{\mathrm{Tr}}
  \renewcommand{\S}{\mathscr{S}}
  \newcommand{\Y}{\mathbb{Y}}
\newcommand{\End}{\mathrm{End}}
\newcommand{\Ext}{\mathrm{Ext}}
\newcommand{\Mat}{\mathrm{Mat}}
\newcommand{\vol}{\mathrm{vol}}
\newcommand{\bigzero}{\mbox{\normalfont\Large\bfseries 0}}

\newtheorem{innercustomthm}{{\bf Theorem}}
\newenvironment{customthm}[1]
  {\renewcommand\theinnercustomthm{#1}\innercustomthm}
  {\endinnercustomthm}
  
  \newtheorem{innercustomcor}{{\bf Corollary}}
\newenvironment{customcor}[1]
  {\renewcommand\theinnercustomcor{#1}\innercustomcor}
  {\endinnercustomthm}
  
  \newtheorem{innercustomprop}{{\bf Proposition}}
\newenvironment{customprop}[1]
  {\renewcommand\theinnercustomprop{#1}\innercustomprop}
  {\endinnercustomthm}

\newcommand{\bbinom}[2]{\begin{bmatrix}#1 \\ #2\end{bmatrix}}
\newcommand{\cbinom}[2]{\set{\^!\^!\^!\begin{array}{c} #1 \\ #2\end{array}\^!\^!\^!}}
\newcommand{\abinom}[2]{\ang{\^!\^!\^!\begin{array}{c} #1 \\ #2\end{array}\^!\^!\^!}}
\newcommand{\qfact}[1]{[#1]^^!}

\newcommand{\nc}{\newcommand}
\def\C{{\mathbb C}}
\def\R{{\mathbb R}}
\def\Z{{\mathbb Z}}
\def\Q{{\mathbb Q}}
\def\A{{\mathbb A}}
\def\G{{\mathbb G}}
\def\g{{\mathfrak g}}
\def\m{{\mathfrak m}}
\def\k{{\mathfrak k}}
\def\calP{{\mathcal P}}
\def\calI{{\mathcal I}}
\def\calA{{\mathcal A}}
\def\calB{{\mathcal B}}
\def\calF{{\mathcal F}}
\def\calM{{\mathcal M}}
\def\calN{{\mathcal N}}
\def\calT{{\mathcal T}}
\def\calH{{\mathcal H}}
\def\calR{{\mathcal R}}
\def\calFJ{{\mathcal {FJ}}}
\def\calS{{\mathcal S}}
\def\calW{{\mathcal W}}
\def\p{{\mathfrak p}}
\def\sl{{\mathfrak sl}}
\def\su{{\mathfrak su}}
\def\h{{\mathfrak h}}
\def\O{{\mathbb O}}
\def\OO{{\mathcal O}}
\def\tm{{\times}}
\def\sm{{\setminus}}
\def\oomm{{\overline{\omega}}}
\def\D{{\delta}}
\def\Om{{\Omega}}
\newcommand{\e}{{\epsilon}}

\newcommand{\calZ}{\mathcal{Z}}
\newcommand{\omm}{{\omega}}
\newcommand{\RR}{\right}
\newcommand{\LL}{\left}
\newcommand{\floor}[1]{\lfloor #1 \rfloor}
\newcommand{\pair}[1]{\langle {#1} \rangle}

\newcommand{\frakg}{\mathfrak{g}}
\newcommand{\frakgl}{\mathfrak{gl}}
\newcommand{\frakso}{\mathfrak{so}}
\newcommand{\fraksp}{\mathfrak{sp}}
\newcommand{\frako}{\mathfrak{o}}
\newcommand{\fraka}{\mathfrak{a}}
\newcommand{\fraku}{\mathfrak{u}}
\newcommand{\frakp}{\mathfrak{p}}
\newcommand{\fraksu}{\mathfrak{su}}
\newcommand{\frakh}{\mathfrak{h}}
\newcommand{\V}{{\Vert}}
\newcommand{\ord}{\rm {ord}}
\newcommand{\Stab}{Stab}
\newcommand{\Sym}{Sym}
\newcommand{\tr}{\mathrm{tr}}
\newcommand{\Lie}{\mathrm{Lie}}
\newcommand{\rank}{\mathrm{rank}}

\newcommand{\innerproduct}[2]{\langle #1, #2 \rangle}
\newcommand{\HSpin}{\mathrm{HSpin}}
\newcommand{\pr}{\mathrm{pr}}
\newcommand{\GU}{GU}
\newcommand{\SU}{SU}
\newcommand{\Orth}{O}
\newcommand{\glue}{glue}
\newcommand{\diag}{\operatorname{diag}}
\newcommand{\charf}{char}

\newcommand{\ff}{B}

\nc{\etab}{\eta^{\bullet}}
\newcommand{\Iblack}{\I_{\bullet}}
\newcommand{\wb}{w_\bullet}
\newcommand{\UIblack}{\U_{\Iblack}}

\newcommand{\blue}[1]{{\color{blue}#1}}
\newcommand{\red}[1]{{\color{red}#1}}
\newcommand{\green}[1]{{\color{green}#1}}
\newcommand{\white}[1]{{\color{white}#1}}

\newcommand{\dvd}[1]{t_{\odd}^{{(#1)}}}
\newcommand{\dvp}[1]{t_{\ev}^{{(#1)}}}
\newcommand{\ev}{\mathrm{ev}}
\newcommand{\odd}{\mathrm{odd}}

\newcommand\TikCircle[1][2.5]{{\mathop{\tikz[baseline=-#1]{\draw[thick](0,0)circle[radius=#1mm];}}}}

\newcommand{\commentcustom}[1]{}

\raggedbottom
\renewcommand{\S}{\textsection}
\newcommand{\reftext}[1]{#1}
\newcommand{\xch}[2]{#1}
\newcommand{\eqncr}{\\}
\newlength{\tablewidth}
\newcommand{\tabnoteref}[2][]{\textsuperscript{#2}}
\newcommand{\tabnotetext}[3][]{\par\noindent\textsuperscript{#2}\,#3\par}
\newcommand{\gnumber}[2][]{#2}
\newcommand{\gsponsor}[2][]{#2}
\newenvironment{sgmlfig}{}{}
\providecommand{\newblock}{\ }
\title{Period integrals of distinguished polarised strongly tempered hyperspherical varieties}
\author{Colin Jia Sheng Loh}
\begin{abstract}
Recent work of \citet*{MWZ} has provided a complete list of strongly
tempered hyperspherical varieties and they proposed some new period
integrals. In this paper, I will present new period integrals of
distinguished polarised strongly tempered hyperspherical varieties and
discuss the L-functions these integrals represent, as examples of the
Relative Langlands Duality.
\end{abstract}
\maketitle
\pagestyle{plain}
\setcounter{tocdepth}{2}
\tableofcontents

\section{Introduction and main results}
\label{Introduction_and_Main_Results_VTEX1}

\subsection{Relative Langlands duality and automorphic $L$-functions}
\label{sec1.1}

The study of automorphic $L$-functions is a central aspect of the Langlands
program. Historically, one of the most powerful tools for establishing
the analytic properties of these L-functions has been the
\textit{Rankin-Selberg method} \citep*{Bump1}, which relies on representing
L-functions via global period integrals. However, constructing the appropriate
period integrals that successfully unfold to yield specific L-functions
has traditionally relied heavily on case-by-case ingenuity.

Recently, the \emph{relative Langlands program} introduced by
\citet*{BZSV} provides a systematic conceptual foundation
\citep*{MWZ} for understanding these integral representations. This arises
from the study of certain $G$-Hamiltonian spaces known as
\textit{hyperspherical varieties}. Within this theory, the authors propose
a deep involutive duality between certain $G$-hyperspherical varieties
$M$, associated with a $G$-hyperspherical quadruple $\mathcal{D}$, and
corresponding $G^{\vee}$-hyperspherical varieties $M^{\vee}$ associated
with a $G^{\vee}$-hyperspherical quadruple $\mathcal{D}^{\vee}$, where
$G^{\vee}$ is the Langlands dual group. Informally, this duality can be
described as follows.

Given a $G$-hyperspherical variety $M$, one can attach two fundamental
invariants: a period invariant $\mathcal{P}_{\mathcal{D}}$ and a spectral
invariant $\mathcal{S}_{\mathcal{D}}$. The relative Langlands duality conjectures
that these invariants satisfy \citep*[\S 1.6]{Gan}
\begin{align*}
&
\text{Period invariant of $M$ ($\mathcal{P}_{\mathcal{D}}$) = Spectral invariant
of $M^{\vee}$ ($\mathcal{S}_{\mathcal{D}^{\vee}}$)},
\\
&
\text{Spectral invariant of $M$ ($\mathcal{S}_{\mathcal{D}}$) = Period invariant
of $M^{\vee}$ ($\mathcal{P}_{\mathcal{D}^{\vee}}$)}.
\end{align*}

\subsubsection{Period invariant}
\label{sec:_Period_invariant_VTEX1}

These period integrals ${\mathcal P}_{\mathcal{D}}$ associated to a hyperspherical
quadruple $\mathcal{D}= (G,H, \rho _{H},\iota )$ (see \reftext{Definition~\ref{defn:_hyperspherical_datum_VTEX1}}) can be constructed as follows. The map
$\iota $ induces an adjoint action of $H {\times }\mathrm{SL}_{2}$ on the
Lie algebra $\mathfrak{g}$ of $G$ such that $\mathfrak{g}$ admits the following
decomposition
\begin{equation*}
\mathfrak{g}= \bigoplus _{k \in I} \rho _{k} \otimes
\operatorname{Sym}^{k}
\end{equation*}
where $\rho _{k}$ are representations of $H$ and $I$ is a subset of
${\mathbb Z}_{\geq 0}$. Let $I_{\text{odd}}$ be the subset of $I$ consisting
of all odd numbers and set
\begin{equation}
\label{eqn:_rho-H-iota_VTEX1}
\rho _{H,\iota} = \rho _{H} \oplus \left ( \bigoplus _{k \in I_{
\text{odd}}} \rho _{k} \right ),
\end{equation}
such that $\rho _{H,\iota}: H \to \mathrm{Sp}(V)$ is a symplectic representation
of $H$. Furthermore, we will assume that $\rho _{H,\iota}$ is anomaly-free
(we refer to \citep*{BZSV} for its definition). Let $k$ denote a number
field and ${\mathbb A}= {\mathbb A}_{k}$ be its ring of adeles. Let
$Y$ be a maximal isotropic subspace of $V$ and let
${\Omega }_{\psi}$ be the Weil representation of
$\widetilde{\mathrm{Sp}}(V)$ defined on the Schwartz space
${\mathcal S}(Y({\mathbb A}))$. As an anomaly-free representation, we have
$\widetilde{\mathrm{Sp}}(V)$ splitting over $\rho _{H,\iota}(H)$ and
${\Omega }_{\psi}$ restricting to a representation of
$H({\mathbb A})$ defined over ${\mathcal S}(Y({\mathbb A}))$. With these,
we define the associated series
\begin{align*}
\Theta _{\psi}^{\Phi}(h) = \sum _{X \in Y(k)} {\Omega }_{\psi}(h)
\Phi (X), && h \in H({\mathbb A}), \Phi \in {\mathcal S}(Y({\mathbb A})),
\end{align*}
and the period integral
\begin{equation}
\label{eqn:_calP-Delta_VTEX1}
{\mathcal P}_{H,\iota , \rho _{H}}(\varphi , \Phi ) = \int _{H(k) {
\setminus }H({\mathbb A})} {\mathcal P}_{\iota}(\varphi )(h) \Theta _{
\psi}^{\Phi}(h)\, dh
\end{equation}
where $\varphi $ is an automorphic form on $G({\mathbb A})$ and
${\mathcal P}_{\iota}(\varphi )$ is the degenerate Whittaker coefficient
of $\varphi $ associated to $\iota $ (see
\citep*[\S 1.2]{MWZ2} for its definition). For brevity we will
denote this period integral as
${\mathcal P}_{\mathcal{D}}(\varphi ;\Phi )$. Suppose
$\mathcal{D}^{\vee }= (\hat G, \hat H', \rho _{\hat H'}, \hat \iota ')$
is the hyperspherical quadruple dual to that of $\mathcal{D}$ under the
proposed relative Langlands duality. One can also define a corresponding
period integral ${\mathcal P}_{\mathcal{D}^{\vee}}(\varphi ;\Phi )$.

We remark that for the family of hyperspherical varieties considered in
this paper, one can consider and construct a complex-parameter dependent
global period integral variant of
${\mathcal P}_{\mathcal{D}}(\varphi ;\Phi )$, denoted by
${\mathcal P}_{\mathcal{D}}(\varphi ,s;\Phi )$. This would involve the
use of certain Eisenstein series and/or inserting determinant twist. See Section~\ref{sec:_Period_integral_variant_involving_complex_parameters_VTEX1} for more
details.
\subsection{Main result}
\label{sec:_Main_result_VTEX1}
The main result of this paper is \reftext{Theorem~\ref{thm:_main_thm_VTEX1}} regarding
the uniform construction of new\footnote{Here we consider a period integral to be new and different from those previously studied Rankin-Selberg integrals, if the integration data between such integrals are different. For instance, Eisenstein series corresponding to different induced data and the use of Bessel coefficient as compared to Fourier-Jacobi coefficient.}
Eulerian period integrals
for the remaining strongly tempered distinguished polarised hyperspherical
varieties (see \S \ref{sec:_DP_VTEX1} and \S \ref{sec:_ST_VTEX1} for their definitions) given in Appendix~\ref{sec:_Appendix_A_multiplicity_free_symplectic_representations_VTEX1}. These period integrals represent the $L$-function for the representation $\tau $ of $G^{\vee}$ given in \reftext{Table~\ref{Table:_Multiplicity-free_repn_VTEX1}}\xch{.}{ below.}

\begin{table}
\caption{Representation $(\tau ,V)$ of $G^{\vee}$.}
\label{Table:_Multiplicity-free_repn_VTEX1}
\begin{tabular}{ll}
\hline
    $G^{\vee}$ & $(\tau , V)$
    \\ \hline
    $\mathrm{GSp}_{2a} {\times }\mathrm{GL}_{b};\, (a,b) \in \left \{
\begin{matrix}
        (a\geq 3,b=2),\\ (a=2,b\geq 4),\\ (a\geq 2,b=3).
    \end{matrix}
\right \}$ & $\mathrm{std}\otimes \mathrm{std}$
    \\
    $\mathrm{GSpin}_{10}$ & $\mathrm{Spin}$
    \\
    $\mathrm{Spin}_{8}$ & $\mathrm{std}\oplus \mathrm{Spin}$
    \\
    $\mathrm{GL}_{n} {\times }\mathbb G_{m},\mathrm{GSp}_{2n}{\times }\mathbb G_{m}$ & $\mathrm{std}\oplus \mathrm{std}$
    \\
    $\underline{\mathrm{GL}}_{2} {\times }G; \, G \in \{\mathrm{GSp}_{2n}, \mathrm{GL}_{n}\}$ &
    $\underline{\mathrm{std}} \oplus (\underline{\mathrm{std}} \otimes \mathrm{std})$
    \\
    $G_{1} {\times }\underline{\mathrm{GL}}_{2} {\times }G_{2}; \, G_{1},G_{2} \in \{\mathrm{GSp}_{2m}, \mathrm{GL}_{n}\}$ & $(\mathrm{std}\otimes \underline{\mathrm{std}}) \oplus (\underline{\mathrm{std}} \otimes \mathrm{std})$
    \\ \hline
    \end{tabular}
\end{table}

\begin{thm}
\label{thm:_main_thm_VTEX1}
Let $\mathcal{D}= (G,H,\rho _{H},\iota )$ be a strongly tempered hyperspherical
quadruple dual to a distinguished polarised hyperspherical quadruple
$\mathcal{D}^{\vee }=(G^{\vee}, G^{\vee}, \tau \oplus \tau ^{\vee}, 1)$
where $(\tau ,V)$ is a representation of $G^{\vee}$ given
in \reftext{Table~\ref{Table:_Multiplicity-free_repn_VTEX1}}\xch{.}{ below.}
\begin{enumerate}
\item[(a)] (\reftext{Theorem~\ref{thm:_Unfolding_of_period_integral_and_absolute_conv_VTEX1}}). The period
integral ${\mathcal P}_{\mathcal{D}}(\varphi ,s;\Phi )$ converges absolutely
away from the poles of the Eisenstein series and in some half-plane of
${\mathbb C}$ or ${\mathbb C}^{2}$,
${\mathcal P}_{\mathcal{D}}(\varphi ,s;\Phi )$ is Eulerian and factors
through the Whittaker model.

\item[(b)] (\reftext{Corollary~\ref{cor:_Euler_factorisation_VTEX1}} and \reftext{Theorem~\ref{thm:_unramified_comp_VTEX1}}). Moreover, when $\pi $ is an irreducible cuspidal
and globally generic representation of $G({\mathbb A})$, the period integral
${\mathcal P}_{\mathcal{D}}(\varphi ,s;\Phi )$ is an integral representation
for the $(G^{\vee},\tau )$ $L$-function, i.e.
\begin{align*}
{\mathcal P}_{\mathcal{D}}(\varphi ,s;\Phi ) = L^{S}(s, \pi , \tau )
\prod _{\nu \in S} \mathcal{Z}(W_{\nu},s;\Phi _{\nu}).
\end{align*}
\end{enumerate}
\end{thm}

\subsection{Organisation of paper}
\label{sec1.3}
The outline of this paper is as follows. Firstly, in Section~\ref{sec:_hyperspherical_varieties_VTEX1} we recall some basic properties of
hyperspherical varieties and discuss two families of such varieties, namely distinguished polarised hyperspherical varieties and strongly tempered hyperspherical varieties. Also, in this section we will discuss the work of \citet*{MWZ} on strongly tempered hyperspherical varieties. Next, in Section~\ref{sec:_Preliminaries_VTEX1} we will introduce some preliminary notations and the degenerate Whittaker coefficient. Moving on, in Section~\ref{sec:_Root_Exchanges_and_Auxiliary_Integrals_VTEX1} we perform the root exchanges on the degenerate Whittaker coefficient and introduce some auxiliary integrals. Then, in Section~\ref{sec:_The_Period_integrals_and_their_unfolding_process_VTEX1} we construct the period integrals ${\mathcal P}_{\mathcal{D}}(\varphi ,s;\Phi )$ and carry out their unfolding process. In Sections~\ref{sec:_Unramified_Computation_VTEX1} and \ref{sec:_Relation_with_previously_studied_Rankin-Selberg_integrals_VTEX1} we will perform the local unramified computation and discuss some relations with the period integrals studied here and previously studied Rankin-Selberg
integrals. Finally, in Appendix~\ref{sec:_Appendix_A_multiplicity_free_symplectic_representations_VTEX1} we will discuss the classification of distinguished polarised strongly tempered hyperspherical varieties via the classification of finite dimensional multiplicity free symplectic representations.

\subsection{Acknowledgement}
\label{sec1.4}
I would first like to thank Lei Zhang for providing advice, support and
encouragement. Additionally, I would like to thank Taiwang Deng, Dihua
Jiang, Anlun Li, Weixiao Lu, Guodong Tang, Chen Wan and Guodong Xi for
helpful discussions. Furthermore, I would like to thank Joseph Hundley,
Wee Teck Gan and Bryan Peng Jun Wang for sharing insights on their work
\citep*{GH} and \citep*{GW} respectively. Lastly, I would also like to
thank the referee for careful reading and helpful suggestions for an earlier version of the paper. The work of the author is supported by
AcRF Tier 2 grant \gnumber[refid=GS4]{MOE-000936-01} of
\gsponsor[id=GS4,sponsor-id=501100001352,country=Singapore]{National University of Singapore}.

\section{Hyperspherical varieties}
\label{sec:_hyperspherical_varieties_VTEX1}
In this section, we will first recall some basic results of a hyperspherical variety and the definition of hyperspherical quadruple. Then, we will discuss two sub-families of hyperspherical varieties. Namely, they are of \emph{distinguished polarized} (DP) type and
\textit{strongly tempered} (ST) type. Finally, we will explain the relationship between these hyperspherical varieties and Rankin-Selberg integrals.

\citet*{BZSV} defined and studied a class of Hamiltonian $G$-varieties
defined over ${\mathbb C}$ (or an algebraically closed field of characteristic zero) known as \emph{hyperspherical varieties}
\citep*[\S 3.5.1]{BZSV}. A key result shown in \xch{\cite{BZSV}}{\emph{loc. cit}} is a structure theorem \citep*[Theorem 3.6.1]{BZSV} for hyperspherical varieties, stating that any hyperspherical $G$-variety arises from the following initial data:
\begin{enumerate}
\item[(1)] A map $\iota : H \times \mathrm{SL}_{2} \to G$ with
$H \subset Z_{G}(\iota (\mathrm{SL}_{2}))$ a reductive spherical subgroup,

\item[(2)] A finite dimensional symplectic representation
$(\rho _{H},S)$ of $H$.
\end{enumerate}
With these initial data, one can construct a Hamiltonian $G$-space
$M$ via the process of Whittaker induction (see \citep*[\S 3.6]{BZSV} and \citep*[\S 4]{GW}). With that, this motivates the following definition.

\begin{definition}
\label{defn:_hyperspherical_datum_VTEX1}
\citep*[Definition 1.2]{MWZ}
A quadruple $\mathcal{D} = (G,H,\rho _{H},\iota )$ is hyperspherical (or BZSV) if the associated $G$-Hamiltonian variety
$\mathcal{M}_{\mathcal{D}}$ is hyperspherical.
\end{definition}
Amongst these hyperspherical varieties, one particular family highlighted in \citep*[\S 4]{BZSV} consists of those of
\textit{distinguished polarised} type.

\subsection{Distinguished polarised hyperspherical variety}
\label{sec:_DP_VTEX1}

\begin{definition}
\label{defn:_DP_hyperspherical_variety_VTEX1}
\citep*[Definition 3.7.1]{BZSV}
Let $\mathcal D = (G, H, \rho _{H}, \iota )$ be a hyperspherical quadruple and write
\begin{align*}
\mathfrak{g}= \bigoplus _{r \in I} \rho _{r} \otimes
\operatorname{Sym}^{r}
\end{align*}
be its associated $H \times \mathrm{SL}_{2}$-decomposition. We say that
the hyperspherical Hamiltonian variety $\mathcal M_{\mathcal D}$ is of
distinguished polarised (DP) type if it satisfies the following two conditions:
\begin{enumerate}[(1)]
\item[(1)] the subset $I_{\operatorname{odd}}\subset I$ is an empty set,

\item[(2)] the symplectic $H$-representation $(\rho _{H}, S)$ admits a
$H$-stable Lagrangian decomposition $S = S^{+} \oplus S^{-}$.
\end{enumerate}
\end{definition}
We remark that the study of such DP-hyperspherical varieties is a generalisation of the earlier work on spherical varieties \citep*{SV}, which include the hyperspherical cotangent bundles $T^\ast(H\setminus G)$ associated with quadruples $(G,H,0,1)$, where $0$ is the zero vector space and $1$ is the trivial map. Moreover, in \citep*[\S 4]{BZSV}, the authors have provided an algorithm to compute the dual of DP-hyperspherical variety. However, given a hyperspherical quadruple
$\mathcal{D} = (G,H,\rho _{H},\iota )$ there is no known systematic procedure to determine its dual $\mathcal{D}^{\vee}$. Nevertheless,
\citet*{MWZ} have provided an algorithm to compute dual of a certain class of hyperspherical varieties known as those of
\textit{symplectic vector space} type. We also note the recent work by
\citet*{TWZ} in constructing an algorithm to derive the dual of an anomaly-free hyperspherical quadruple $\mathcal{D} = (G,H,\rho _{H},\iota )$ where $G$ is a simple reductive group.
\subsection{Strongly tempered hyperspherical variety}
\label{sec:_ST_VTEX1}
\begin{definition}
\label{defn:_ST,_symplectic_vector_space_hyperspherical_variety_VTEX1}
\citep*[Definition 1.8]{MWZ}
A hyperspherical variety $\mathcal{M}_{\mathcal{D}}$ is of
\textit{symplectic vector space} (sVP) type if the corresponding BZSV quadruple is of the form $\mathcal{D} = (G,G,\rho ,1)$. Also, we call a hyperspherical variety $\mathcal{N}$ of \textit{strongly tempered} (ST) type if its BZSV dual $\mathcal{N}^{\vee}$ is of symplectic vector space type.
\end{definition}
For such sVP-hyperspherical varieties $\mathcal{M}_{\mathcal{D}}$ where
$\mathcal{D} = (G,G,\rho ,1)$, the hyperspherical and anomaly-free conditions of $\mathcal{M}_{\mathcal{D}}$ are equivalent to the following three conditions
\citep*[\S 1.3]{MWZ}:
\begin{enumerate}[(1)]
\item[(1)] The symplectic representation $\rho $ is anomaly-free (see
\citep*[Definition 2.7]{MWZ} and
\citep*[Definition 5.1.2 and Proposition 5.1.5]{BZSV});

\item[(2)] The symplectic representation $\rho $ is multiplicity free (i.e. the field of invariant rational functions on the vector space is commutative with respect to the Poisson bracket);

\item[(3)] The generic stabilizer of the representation $\rho $ of
$G$ is connected.
\end{enumerate}
\citet*{MWZ} studied such sVP-hyperspherical varieties using the classification of multiplicity free symplectic representations by \citet*{Knop} and \citet*{Losev} independently. In particular, for these multiplicity free symplectic representations, \citet*{MWZ} proposed an ST hyperspherical variety that is dual to those of sVP type. In other words, \citet*{MWZ} have provided a complete list (up to isogeny) of ST anomaly-free hyperspherical varieties that are dual to those of sVP type under the proposed relative Langlands duality.

\subsection{Rankin-Selberg integrals}
\label{sec2.3}
One motivating reason to study such particular families of hyperspherical varieties is due to their deep connection with integral representations of $L$-functions, i.e. Rankin-Selberg integrals. As pointed out by \citet*[Example 4.3.12]{BZSV}, the authors propose that Rankin-Selberg integrals representing the $L$-function of $(\tau ,V)$, for which $T^{\ast }V$ is hyperspherical can be constructed in a uniform manner (see Section~\ref{sec:_Period_invariant_VTEX1} for more details).

Adding on, the study of such families of hyperspherical varieties can provide a better conceptual understanding of Rankin-Selberg integrals. As highlighted by \citet*{MWZ}, many of the previously studied Rankin-Selberg integrals that coincide with \citet*{BZSV}'s current framework arise from DP-sVP hyperspherical varieties (see Appendix~\ref{sec:_Appendix_A_multiplicity_free_symplectic_representations_VTEX1}).

We remark that there already exist Rankin-Selberg integrals for some of
the L-functions studied in this paper. For example, \citet*{ACS} have constructed a global Rankin-Selberg integral for $L(s, \pi \times \tau )$ where $\pi $ is an irreducible globally generic cuspidal automorphic representation of $\mathrm{GSpin}$ and $\tau $ is an irreducible cuspidal automorphic representation of a general linear group. Also, \citet*{Yan-Zhang} have constructed Rankin-Selberg integrals for the product of standard tensor L-functions of
$\mathrm{GL}_{l} {\times }\mathrm{GL}_{m}$ and of $\mathrm{GL}_{l} {\times }\mathrm{GL}_{n}$ for $m + n < l$, where our period
integral representation for $(G^{\vee}, \tau ) = (\mathrm{GL}_{m} {\times } \underline{\mathrm{GL}}_{2} {\times }\mathrm{GL}_{n}, (\mathrm{std}_{m}
\otimes \underline{\mathrm{std}}_{2}) \oplus (
\underline{\mathrm{std}}_{2} \otimes \mathrm{std}_{n}))$ for
$m,n\geq 2$, is not covered in \xch{\cite{Yan-Zhang}}{\emph{loc. cit}}. \citet*{GH} and \citet*{Gin95a} have constructed (multivariable) Rankin-Selberg integrals
representing (products of) $\mathrm{Spin}$ $L$-functions defined on
$\mathrm{GSO}_{2n}$ for $n=4,5$ respectively. Moreover, there are also
existing Rankin-Selberg integrals that are also associated to hyperspherical varieties such as the works of \citet*{LWX} and \citet*{LX} and the works of \citet*{N}, \citet*{BG} and \citet*{Gin95} to name a few.

\section{Preliminaries}
\label{sec:_Preliminaries_VTEX1}
In this section, we will introduce some notations and subgroups, embeddings and maps, Fourier-coefficients and auxiliary integrals that will be used throughout the entire paper.

\subsection{Notations and subgroups}
\label{sec3.1}
We will adopt the following notations. Throughout this paper, we will let $k$ denote a number field, and let ${\mathbb A}$ denote its ring of adeles and $F$ to denote a local field. Also, we will let $G$ denote a split connected reductive group. We fix a Borel subgroup $B_{G} = T_{G} N_{G}$ containing a maximal torus $T_{G}$, where $N_{G}$ is its unipotent radical. We let $\Delta $ denote a corresponding set of simple roots. For matrix groups defined below, the Borel subgroup $B_{G}$ consists of all upper triangular
matrices whereas $N_{G}$ will consist of all upper triangular unipotent
matrices and the maximal torus $T_{G}$ will be the subgroup of all diagonal matrices. Let $\mathrm{Mat}_{m,n}$ be the space consisting of
$m{\times }n$ matrices and given $X \in \mathrm{Mat}_{m,n}$ we write
${}^{t} X$ to be the matrix transpose of $X$. We also denote
$E_{i,j} \in \mathrm{Mat}_{n,n}$ to be the square $n{\times }n$ matrix
with ones on the $(i,j)$-th entry and zeros otherwise. Let
$ n(x) \in \mathrm{GL}_{2}, w_{n} \in \mathrm{GL}_{n}$ and
$j_{2n} \in \mathrm{GL}_{2n}$ be the matrix elements given by
\begin{align}
\label{eqn:_w-n_and_j-2n_VTEX1}
n(x) =
\begin{pmatrix}
1 & x
\\
& 1
\end{pmatrix}, && w_{n} = \left (
\begin{smallmatrix}
&& 1
\\
& \iddots &
\\
1 &&
\end{smallmatrix}\right )\in \mathrm{GL}_{n}, && j_{2n} =
\begin{pmatrix}
& -w_{n}
\\
w_{n} &
\end{pmatrix}\in \mathrm{GL}_{2n}.
\end{align}
Given any matrix element $g \in \mathrm{GL}_{n}$, we let ${}^{t} g$ to
denote its transpose, and $g^{\ast}$ and $g_{\ast}$ to be
\begin{align}
\label{eqn:_g-upper_ast,_g-lower_ast_VTEX1}
g^{\ast }= w_{n} {}^{t} g^{-1}w_{n}, && g_{\ast }= \frac{1}{\det g}
\cdot g.
\end{align}
In our case, $G$ will denote any of the following groups
$\mathrm{GL}_{n}, \mathrm{GSp}_{2n}, \mathrm{GSpin}_{2n+1}$ and
$\mathrm{GSO}_{2n}$ together with finite products of such groups. We will fix the following conventions of $\mathrm{GSp}_{2n}$ and
$\mathrm{GSO}_{2n}$ given as follows
\begin{gather*}
\mathrm{GSp}_{2n} = \{g \in \mathrm{GL}_{2n} \mid {}^{t} g j_{2n} g =
\lambda (g) j_{2n} \},
\\
\mathrm{GSO}_{2n} = \{g \in \mathrm{GL}_{2n}
\mid {}^{t} g w_{2n} g = \lambda (g) w_{2n}, \det (g) = \lambda (g)^{n}
\},
\end{gather*}
where $\lambda (g)$ is the similitude of $g$. As for the odd Spin similitude
group $\mathrm{GSpin}_{2n+1}$, we will define it via its based root datum.

\subsubsection{Root datum of $\mathrm{GSpin}_{2n+1}$}
\label{sec3.1.1}
According to \citet*[\S 4]{HS}, the based root datum of the connected
split reductive group $\mathrm{GSpin}_{2n+1}$ is given by
$(X,R,\Delta ,X^{\vee},R^{\vee},\Delta ^{\vee})$, where $X$ and
$X^{\vee}$ are ${\mathbb Z}$-modules generated by generators
$e_{0},e_{1},\ldots ,e_{n}$ and
$e_{0}^{*},e_{1}^{*},\ldots ,e_{n}^{*}$, respectively. The roots and coroots are given by
\begin{alignat*}{3} R & = & R_{2n+1} & = \{\pm (e_{i} \pm e_{j}): 1
\leq i < j \leq n\} \cup \{\pm e_{i} : 1 \leq i \leq n\}
\\
R^{\vee }& = & R^{\vee}_{2n+1} & = \{\pm (e_{i}^{*} - e_{j}^{*}): 1
\leq i < j \leq n\} \cup \{\pm (e_{i}^{*} + e_{j}^{*} - e_{0}^{*}): 1
\leq i < j \leq n\}
\\
& & & \quad \cup \{\pm (2e_{i}^{*}-e_{0}^{*}): 1 \leq i \leq n\}
\end{alignat*}
Moreover, we fix the following choice of simple roots and coroots:
\begin{align*}
\Delta &= \{\alpha _{1}=e_{1}-e_{2}, \alpha _{2}=e_{2}-e_{3}, \ldots ,
\alpha _{n-1}=e_{n-1}-e_{n}, \alpha _{n}=e_{n}\},
\\
\Delta ^{\vee }&= \{\alpha _{1}^{\vee}=e_{1}^{*}-e_{2}^{*}, \alpha _{2}^{
\vee}=e_{2}^{*}-e_{3}^{*}, \dots ,\alpha _{n-1}^{\vee}= e_{n-1}^{*}-e_{n}^{*},
\alpha _{n}^{\vee}= 2e_{n}^{*}-e_{0}^{*}\}.
\end{align*}
The based root datum determines, up to isomorphism, the group
$\mathrm{GSpin}_{2n+1}$ together with a Borel subgroup
$B_{\mathrm{GSpin}_{2n+1}}$ and a split maximal torus
$T_{\mathrm{GSpin}_{2n+1}} \subset B_{\mathrm{GSpin}_{2n+1}}$. Analogously, consider
$\mathrm{SO}_{2n+1} \supset B_{\mathrm{SO}_{2n+1}} \supset T_{
\mathrm{SO}_{2n+1}}$, there is a projection
\begin{equation*}
\operatorname{proj}: \mathrm{GSpin}_{2n+1} \rightarrow \mathrm{SO}_{2n+1},
\end{equation*}
which induces isomorphisms on unipotent varieties. We can further assume this projection preserves the choice of Borel subgroups and split maximal torus, i.e.
\begin{align*}
B_{\mathrm{GSpin}_{2n+1}} = \operatorname{proj}^{-1}(B_{\mathrm{SO}_{2n+1}}),&&
T_{\mathrm{GSpin}_{2n+1}} = \operatorname{proj}^{-1}(T_{\mathrm{SO}_{2n+1}}).
\end{align*}
We will represent an element of the maximal torus
$T_{\mathrm{GSpin}_{2n+1}}$ in the form
\begin{align*}
e_{0}^{*}(a_{0})e_{1}^{*}(a_{1})\cdots e_{n}^{*}(a_{n}), && a_{i}
\in \mathrm{GL}_{1}.
\end{align*}

\subsubsection{Matrix group isomorphism of low rank $\mathrm{GSpin}$ groups}
\label{sec:_Accidental_Isomorphism_of_low_rank_GSpin_groups_VTEX1}
It is well-known that there exist accidental isomorphisms for low rank
$\mathrm{GSpin}$ groups such that
\begin{align*}
\mathrm{GSpin}_{4} \cong & G(\mathrm{SL}_{2} {\times }\mathrm{SL}_{2})
= \{(g_{1},g_{2}) \in \mathrm{GL}_{2} {\times }\mathrm{GL}_{2} \mid
\det g_{1} = \det g_{2}\},
\\
\mathrm{GSpin}_{5} \cong & \mathrm{GSp}_{4},
\\
\mathrm{GSpin}_{6} \cong & \{(g,z) \in \mathrm{GL}_{4} {\times }
\mathrm{GL}_{1} \mid \det g = z^{2}\}.
\end{align*}
Similar to \citet*{AC}, we choose the following pinnings for
$\mathrm{GSpin}_{4}$ and $\mathrm{GSpin}_{6}$. In
$\mathrm{GSpin}_{4}$, we identify the generators of
$\mathrm{GSpin}_{4}$ in
$G(\mathrm{SL}_{2} {\times }\mathrm{SL}_{2})$ as
\begin{align*}
x_{e_{1}+e_{2}}(r) =& (n(r), I_{2}),&& x_{e_{1}-e_{2}}(r) =(I_{2}, n(r)),
\\
x_{-(e_{1}+e_{2})}(r) =& ({}^{t} n(r), I_{2}), && x_{-(e_{1}-e_{2})}(r)
= (I_{2}, {}^{t} n(r)),
\end{align*}
\begin{align*}
e_{0}^{\ast}(a_{0})=& (a_{0}I_{2},a_{0}I_{2}),
\qquad
e_{1}^{\ast}(a_{1})=
(\operatorname{diag}(a_{1},1), \operatorname{diag}(a_{1},1)),
\\
e_{2}^{\ast}(a_{2})=& (\operatorname{diag}(a_{2},1), \operatorname{diag}(1,a_{2})).
\end{align*}
With these, we realise
$G(\mathrm{SL}_{2} {\times }\mathrm{SL}_{2}) \cong \mathrm{GSpin}_{4}$
as the subgroup of $\mathrm{GSpin}_{2n+1}$ for $n\geq 2$ by
\begin{align}
\label{eqn:_embedding_GSpin4_VTEX1}
j_{n}: \mathrm{GSpin}_{4} \cong \left \langle x_{\pm (e_{1}+e_{2})}, x_{
\pm (e_{1}-e_{2})}, e_{0}^{\ast}, e_{1}^{\ast}, e_{2}^{\ast }\right
\rangle \hookrightarrow \mathrm{GSpin}_{2n+1}.
\end{align}
In particular by identifying $\mathrm{GSp}_{4}$ with
$\mathrm{GSpin}_{5}$, we realise
$G(\mathrm{SL}_{2} {\times }\mathrm{SL}_{2}) \subset \mathrm{GSp}_{4}$
via the embedding:
\begin{align*}
j_{2}: G(\mathrm{SL}_{2} {\times }\mathrm{SL}_{2}) \hookrightarrow
\mathrm{GSp}_{4}; && (g_{1},g_{2}) \mapsto
\begin{pmatrix}
a_{1} &&b_{1}
\\
&g_{2}&
\\
c_{1}&&d_{1}
\end{pmatrix}, && g_{1} =
\begin{pmatrix}
a_{1} & b_{1}
\\
c_{1} & d_{1}
\end{pmatrix} \in \mathrm{GL}_{2}.
\end{align*}
Similarly for $\mathrm{GSpin}_{6}$, we identify its generators in
$\{(g,z) \in \mathrm{GL}_{4} {\times }\mathrm{GL}_{1} \mid \det g = z^{2}
\}$ as
\begin{align*}
&x_{e_{2}+e_{3}}(r) = (\left (
\begin{smallmatrix}
1&r&&
\\
&1&&
\\
&&1&
\\
&&&1
\end{smallmatrix} \right ),1), && x_{e_{1}+e_{3}}(r) = (\left (
\begin{smallmatrix}
1&&r&
\\
&1&&
\\
&&1&
\\
&&&1
\end{smallmatrix} \right ),1), && x_{e_{1}+e_{2}}(r) = (\left (
\begin{smallmatrix}
1&&&r
\\
&1&&
\\
&&1&
\\
&&&1
\end{smallmatrix} \right ),1),
\\
&x_{e_{1}-e_{2}}(r) = (\left (
\begin{smallmatrix}
1&&&
\\
&1&r&
\\
&&1&
\\
&&&1
\end{smallmatrix} \right ),1), && x_{e_{1}-e_{3}}(r) = (\left (
\begin{smallmatrix}
1&&&
\\
&1&&r
\\
&&1&
\\
&&&1
\end{smallmatrix} \right ),1), && x_{e_{2}-e_{3}}(r) = (\left (
\begin{smallmatrix}
1&&&
\\
&1&&
\\
&&1&r
\\
&&&1
\end{smallmatrix} \right ),1),
\end{align*}
and
\begin{align*}
&e_{0}^{\ast}(t) = (tI_{4},t^{2}), && e_{1}^{\ast}(t) = (\left (
\begin{smallmatrix}
t&&&
\\
&t&&
\\
&&1&
\\
&&&1
\end{smallmatrix} \right ),t),
\\
&e_{2}^{\ast}(t) = (\left (
\begin{smallmatrix}
t&&&
\\
&1&&
\\
&&t&
\\
&&&1
\end{smallmatrix} \right ),t), && e_{3}^{\ast}(t) = (\left (
\begin{smallmatrix}
t&&&
\\
&1&&
\\
&&1&
\\
&&&t
\end{smallmatrix} \right ),t).
\end{align*}
With these, we define the projection of
$\operatorname{pr}: \mathrm{GSpin}_{6}(k) \to \mathrm{GL}_{4}(k)$ by
$\operatorname{pr}(g,z) = g$ and let
$\mathrm{GL}_{4}'(k):= \operatorname{pr}(\mathrm{GSpin}_{6}(k))$ denote
its image. Furthermore, we identify $\mathrm{GSpin}_{6}$ as a subgroup
of $\mathrm{GSpin}_{2n+1}$ for $n\geq 3$ by
\begin{align}
\label{eqn:_embedding_GSpin6_VTEX1}
j_{n}: \mathrm{GSpin}_{6} \cong \langle x_{\pm (e_{1}\pm e_{2})}, x_{
\pm (e_{1}\pm e_{3})}, x_{\pm (e_{2}\pm e_{3})}, e_{0}^{\ast}, e_{1}^{
\ast}, e_{2}^{\ast}, e_{3}^{\ast }\rangle \hookrightarrow
\mathrm{GSpin}_{2n+1}.
\end{align}
Finally, for the groups $\mathrm{GSO}_{2n}$ we follow the notations of
\citet*{GH}. Let $\alpha _{i} (1 \leq i \leq n)$ denote the simple roots given by $\alpha _{j} = e_{j} -e_{j+1}$ for $1 \leq j \leq n-1$ and $\alpha _{n} = e_{n-1} + e_{n}$, and let $x_{\alpha _{i}}(r)$ denote the one-dimensional unipotent subgroup corresponding to the root
$\alpha _{i}$. The roots are labelled such that
\begin{align*}
x_{\alpha _{i}}(r) = I_{2n} + r e_{i,i+1}'
\end{align*}
for $1 \leq i \leq n-1$ and
$x_{\alpha _{n}} = I_{2n} + r e_{n-1,n+1}'$ where
$e_{i,j}' = E_{i,j} - E_{2n+1-j, 2n+1-i}$.

Following the notations of \citet*[\S 2.1]{MWZ}, we will define
the following groups
\begin{align}
\label{eqn:_S(GL2_x_GL2_x_GL2)_VTEX1}
S(\mathrm{GL}_{2} {\times }\mathrm{GL}_{2} {\times }\mathrm{GL}_{2}) =&
\{(g_{1},g_{2},g_{3}) \in \mathrm{GL}_{2}^{3} \mid \det (g_{1}g_{2}g_{3})
= 1\},
\eqncr
\label{eqn:_S(GL2_x_GSO4)_VTEX1}
S(\mathrm{GL}_{2} {\times }\mathrm{GSO}_{4}) =& \{(g,h) \in
\mathrm{GL}_{2} {\times }\mathrm{GSO}_{4} \mid \det g \cdot \lambda (h)
= 1\}.
\end{align}
Additionally, we introduce the following groups
\begin{align}
\label{eqn:_S'(GSp4_x_GL4)_VTEX1}
S'(\mathrm{GSp}_{4} {\times }\mathrm{GL}_{4}) =& \{(g,h) \in
\mathrm{GSp}_{4} {\times }\mathrm{GL}_{4} \mid \lambda (g) \det h = 1
\},
\eqncr
\label{eqn:_S''_VTEX1}
S''(\mathrm{GL}_{2}^{4}) =& \{(g_{1},g_{2},g_{3},g_{4}) \in
\mathrm{GL}_{2}^{4} \mid \det g_{3} = \det g_{4},\det (g_{1} g_{2} g_{3})=1\},
\eqncr
\label{eqn:_S-ast_VTEX1}
S^{\ast}(\mathrm{GL}_{2}^{5}) =&\{(g_{1},g_{2},g_{3},g_{4}, g_{5})
\in \mathrm{GL}_{2}^{5} \mid \det g_{1} = \det g_{2}, \det (g_{2}g_{3}g_{4})
= 1, \det g_{4} = \det g_{5}\}.
\end{align}

\subsection{Representations and Eisenstein series}
\label{sec3.2}
Throughout this paper, we will assume the irreducible cuspidal automorphic
representations denoted by $\pi _{n}$ (and $\pi _{n}', \pi _{n}''$),
$\tau _{n}$ (and $\tau _{n}'$) and $\sigma _{n}$ of
$\mathrm{GL}_{n}({\mathbb A}), \mathrm{GSpin}_{2n+1}({\mathbb A})$ and
$\mathrm{GSO}_{2n}({\mathbb A})$ respectively to be globally generic. Let
$\Xi _{n}$ be any such representation, we will let
$\omega _{\Xi _{n}}$ denote its corresponding central character, and let
$\mathbf 1$ denote the trivial character. Fix a nontrivial unitary additive
character
$\psi : k {\setminus }{\mathbb A}\to {\mathbb C}^{{\times }}$. By the generic
assumption above, each of these representations has a nonzero Whittaker
model. That is, the space $\mathcal{W}(\Xi _{n},{\psi _{N_{G}}})$ consisting
of functions $W_{\varphi _{n}}^{{\psi _{N_{G}}}}$ defined by
\begin{align*}
W_{\varphi _{n}}^{{\psi _{N_{G}}}}(g) = \int _{N_{G}(k) {\setminus }N_{G}({
\mathbb A})} \varphi _{n}(ug) \psi _{N_{G}}^{-1}(u)\,du, && \varphi _{n}
\in \Xi _{n}, g \in G({\mathbb A}),
\end{align*}
is nonzero. Here
$\psi _{N_{G}} : N_{G}(k) {\setminus }N_{G}({\mathbb A}) \to {
\mathbb C}^{{\times }}$ is the non-degenerate character defined by
\begin{align*}
\psi _{N_{G}}(u) = \psi \left ( \sum _{\alpha \in \Delta} u_{\alpha }
\right ), && u \in N_{G}({\mathbb A}).
\end{align*}
In particular, in the case when $G= \mathrm{GL}_{n}$, we have
\begin{align*}
\psi _{N_{\mathrm{GL}_{n}}}(u) = \psi \left ( \sum _{i=1}^{n-1} u_{i,i+1}
\right ), && u \in N_{\mathrm{GL}_{n}}({\mathbb A}).
\end{align*}
For the groups $\mathrm{GSp}_{4}$ and $\mathrm{GL}_{3}$, we also define
the characters
$\psi '_{N_{\mathrm{GSp}_{4}}} : N_{\mathrm{GSp}_{4}}(k) {\setminus }\break N_{
\mathrm{GSp}_{4}}({\mathbb A}) \to {\mathbb C}^{{\times }}$ and
$\psi '_{N_{\mathrm{GL}_{3}}}: N_{\mathrm{GL}_{3}}(k) {\setminus }N_{
\mathrm{GL}_{3}}({\mathbb A}) \to {\mathbb C}^{{\times }}$ given by
\begin{align*}
\psi _{N_{\mathrm{GSp}_{4}}}'( \left (
\begin{smallmatrix}
1&-y_{1}&&
\\
&1&&
\\
&&1&y_{1}
\\
&&&1
\end{smallmatrix}\right )\left (
\begin{smallmatrix}
1&&y_{2}&y_{4}
\\
&1&y_{3}&y_{2}
\\
&&1&
\\
&&&1
\end{smallmatrix}\right )) =& \psi (y_{1}+y_{3}),&& \psi _{N_{
\mathrm{GL}_{3}}}'(\left (
\begin{smallmatrix}
1&y_{3}&y_{2}
\\
&1&y_{1}
\\
&&1
\end{smallmatrix}\right )) = \overline{\psi (-y_{1}+y_{3})}.
\end{align*}
We will drop the subscript $N_{G}$ when the context is clear. Aside from
cuspidal generic representations, we will also introduce the following
two types of Eisenstein series considered in this paper, namely the
\textit{mirabolic Eisenstein series} for $\mathrm{GL}_{n}$ and the
\textit{Siegel Eisenstein series} for $\mathrm{GSO}_{4}$. We will follow
Cogdell's exposition \citep*[Chapter 5]{C} of the mirabolic Eisenstein
series for $\mathrm{GL}_{n}$. Let $P_{n}$ be the parabolic subgroup of
$\mathrm{GL}_{n}$ of type $(n-1,1)$ given by
\begin{align}
\label{eqn:_GLn_parabolic_subgroup_VTEX1}
P_{n}= \left \{
\begin{pmatrix}
g & X
\\
& t
\end{pmatrix} \mid g \in \mathrm{GL}_{n-1}, t \in \mathrm{GL}_{1}, X
\in \mathrm{Mat}_{n-1,1}\right \}.
\end{align}
Let ${\delta }_{P_{n}}$ be the modular character on $P_{n}$,
$\chi : k^{{\times }}{\setminus }{\mathbb A}^{{\times }}\to {
\mathbb C}^{{\times }}$ to be a unitary Hecke character and
$s\in {\mathbb C}$ be a complex parameter. We will define
$I_{n}(s, \chi )= \mathrm{Ind}_{P_{n}({\mathbb A})}^{\mathrm{GL}_{n}({
\mathbb A})} {\delta }_{P_{n}}^{s} \otimes \chi ^{-1}$. The induced space
$I_{n}(s,\chi )$ admits a section $F_{n}(\cdot ,s,\chi ;\Phi )$ constructed
from space of Schwartz-Bruhat functions
$\mathcal S({\mathbb A}^{n})$ as follows. Given
$\Phi _{n} \in \mathcal S({\mathbb A}^{n})$ and
$g \in \mathrm{GL}_{n}({\mathbb A})$, we set
\begin{align}
\label{eqn:_Godement_section_GLn_VTEX1}
F_{n}(g,s,\chi ;\Phi _{n}) = |\det g|^{s} \int _{{\mathbb A}^{{
\times }}} \Phi _{n}(a e_{n} g) \chi (a) |a|^{ns}\,d^{{\times }}a, && e_{n}
= (0,\dots ,0,1) \in \mathrm{Mat}_{1,n}({\mathbb A}).
\end{align}
With these, we define the mirabolic Eisenstein series for
$\mathrm{GL}_{n}$ as
\begin{align}
\label{eqn:_mirabolic_Eisenstein_series_VTEX1}
E_{n}(g, s,\chi ;\Phi _{n}) = \sum _{\gamma \in P_{n}(k) {\setminus }
\mathrm{GL}_{n}(k)} F_{n}(\gamma g, s, \chi ;\Phi _{n}).
\end{align}
The mirabolic Eisenstein series has the standard analytic properties. Namely,
$E_{n}(g,s,\allowbreak \chi ;\allowbreak \Phi _{n})$ converges absolutely for
$\text{Re}(s) > 1$, admits meromorphic continuation, and satisfies the standard
functional equation. As for the Siegel Eisenstein series for
$\mathrm{GSO}_{4}$, we first define the Siegel parabolic subgroup
$P = P_{\mathrm{GSO}_{4}}$ of $\mathrm{GSO}_{4}$ given by
\begin{align*}
P = \left \{
\begin{pmatrix}
\lambda A & X
\\
& A^{\ast }
\end{pmatrix} \mid A \in \mathrm{GL}_{2}, \lambda \in \mathrm{GL}_{1}
\right \},
\end{align*}
where $X$ is chosen such that the block matrix is in
$\mathrm{GSO}_{4}$. Let ${\delta }_{P}$ be the modular character of
$P$, we define
$I(s) = \mathrm{Ind}_{P({\mathbb A})}^{\mathrm{GSO}_{4}({\mathbb A})} {
\delta }_{P}^{s}$ as the usual induced space. Given a section
$f_{s} \in I(s)$ we define the corresponding normalised Siegel Eisenstein
series for $\mathrm{GSO}_{4}$ as
\begin{align}
\label{eqn:_Siegel_GSO4_VTEX1}
E_{P}^{\ast}(g, f_{s}) := \zeta _{k}(2s) \sum _{\gamma \in P(k) {
\setminus }\mathrm{GSO}_{4}(k)} f_{s}(\gamma g), && g \in
\mathrm{GSO}_{4}({\mathbb A}),
\end{align}
where $\zeta _{k}$ is the Dedekind zeta function of $k$.
\subsection{Embeddings and maps}
\label{subsec:_Embeddings_and_Maps_VTEX1}

In this subsection, we will introduce some embeddings and maps used in
the construction of the period integrals
${\mathcal P}_{\mathcal{D}}$. To start, for $2 \leq k \leq n$ we will define
the embeddings $j_{n,k}: \mathrm{GL}_{k} \to \mathrm{GL}_{n}$,
$\iota _{n}: \mathrm{GL}_{n} \to \mathrm{Sp}_{2n}$ given by
\begin{align*}
j_{n,k} : \mathrm{GL}_{k} \to \mathrm{GL}_{n}; && g \mapsto
\operatorname{diag}(g,I_{n-k}),
\\
\iota _{n}:\mathrm{GL}_{n} \to \mathrm{Sp}_{2n}; && g \mapsto
\operatorname{diag}(g,g^{\ast}).
\end{align*}
Next, recalling the group
$S(\mathrm{GL}_{2} {\times }\mathrm{GSO}_{4})$ in \reftext{\eqref{eqn:_S(GL2_x_GSO4)_VTEX1}} we will define the embeddings
\begin{align*}
J_{D_{4}}: S(\mathrm{GL}_{2} {\times }\mathrm{GSO}_{4}) \to
\mathrm{GSO}_{8}, && J_{D_{5}}: \mathrm{GL}_{2} \to \mathrm{GSO}_{10},
\end{align*}
given by
\begin{align}
\label{eqn:_J-D4_VTEX1}
J_{D_{4}}(g,h) =
\begin{pmatrix}
A &&&B
\\
&g^{\ast}&&
\\
&&g_{\ast}&
\\
C&&&D
\end{pmatrix},
\end{align}
for $(g,h) \in S(\mathrm{GL}_{2} {\times }\mathrm{GSO}_{4})$ where the
notations $g^{\ast}$ and $g_{\ast}$ are defined in \reftext{\eqref{eqn:_g-upper_ast,_g-lower_ast_VTEX1}}, and
$h =
\begin{pmatrix}
A & B
\\
C&D
\end{pmatrix}
$ written in terms of $2{\times }2$ block matrices, and also
\begin{align}
\label{eqn:_J-D5_VTEX1}
J_{D_{5}}(
\begin{pmatrix}
a&b
\\
c&d
\end{pmatrix}) = \operatorname{diag}(ad-bc,
\begin{pmatrix}
A & B
\\
C &D
\end{pmatrix}, 1),
\end{align}
for $
\begin{pmatrix}
a&b
\\
c&d
\end{pmatrix}
\in \mathrm{GL}_{2}$, where
$A = \left (
\begin{smallmatrix}
a&&&
\\
&a&&
\\
&&a&
\\
&&&a
\end{smallmatrix}
\right )$,
$B =\left (
\begin{smallmatrix}
b&&&
\\
&-b&&
\\
&&b&
\\
&&&-b
\end{smallmatrix}
\right )$,
$C = \left (
\begin{smallmatrix}
c&&&
\\
&-c&&
\\
&&c&
\\
&&&-c
\end{smallmatrix}
\right )$ and
$D = \left (
\begin{smallmatrix}
d&&&
\\
&d&&
\\
&&d&
\\
&&&d
\end{smallmatrix}
\right )$.

Next, we will recall the Kronecker product map as
\begin{align}
\label{eqn:_Kronecker_product_map_VTEX1}
\bigotimes : \mathrm{GL}_{n} {\times }\mathrm{GL}_{m} \to \mathrm{GL}_{nm};
&& (g,h) \mapsto g \otimes h =
\begin{pmatrix}
g_{11}h&\cdots & g_{1n}h
\\
\vdots & \ddots & \vdots
\\
g_{n1}h& \cdots & g_{nn}h
\end{pmatrix}.
\end{align}
We extend the Kronecker product map to matrix subgroups
$A_{n} {\times }B_{m}\subset \mathrm{GL}_{n} {\times }\mathrm{GL}_{m}$,
and denote $A_{n} \otimes B_{m}$ to be their image in
$\mathrm{GL}_{nm}$. Using this Kronecker product map, we will construct
three more maps used in the construction of the period integrals in subsequent
sections. Firstly, we let
$\gamma _{D_{4}} = \operatorname{diag}(1,-1,1,-1,1,1,1,1)$, then we define
$\iota _{D_{4}} : S(\mathrm{GL}_{2} {\times }\mathrm{GSO}_{4}) \to
\mathrm{Sp}_{8}$ given by
\begin{align}
\label{eqn:_iota-D4_VTEX1}
\iota _{D_{4}}(g,h) = \gamma _{D_{4}}(h \otimes g) \gamma _{D_{4}}^{-1},
&&(g,h) \in S(\mathrm{GL}_{2} {\times }\mathrm{GSO}_{4}).
\end{align}
Next, we will define a map
$\mathrm{Ext}^{2}: S'(\mathrm{GSp}_{4} {\times }\mathrm{GL}_{4}) \to
\mathrm{Sp}_{24}$ as follows. To start, we define the map
$M(\cdot , 2) : \mathrm{GL}_{4} \to \mathrm{GL}_{6}$ as follows. Fix the
ordered set
\begin{align*}
\mathcal{I} = \{12,13,14,23,24,34\}.
\end{align*}
For $g = (g_{ij})_{1 \leq i,j \leq 4} \in \mathrm{GL}_{4}$ we define
$\Delta _{ij,kl}(g) = g_{ik} g_{jl} - g_{il} g_{jk}$ and define
\begin{align*}
M(g,2) = (\Delta _{ij,kl}(g))_{ij,kl \in \mathcal{I}},
\end{align*}
and define an auxiliary group homomorphism
$(\wedge ^{2})' : \mathrm{GL}_{4} \to \mathrm{GL}_{6}$ given by
\begin{align*}
(\wedge ^{2})'(h) := \eta M(h,2) \eta ^{-1},
\end{align*}
for $h \in \mathrm{GL}_{4}$, where
$\eta = \operatorname{diag}(1,1,1,1,-1,1)$ in $\mathrm{GL}_{6}$. With these
we define $\mathrm{Ext}^{2}$ given by
\begin{align}
\label{eqn:_Ext2_map_VTEX1}
\mathrm{Ext}^{2}(g,h) = (\gamma _{24}\eta _{24}) ((\wedge ^{2})'h
\otimes g) (\gamma _{24}\eta _{24})^{-1},
\end{align}
for $(g,h) \in S'(\mathrm{GSp}_{4} {\times }\mathrm{GL}_{4})$ where
$\gamma _{24}$ and $\eta _{24}$ are the matrix elements in
$\mathrm{GL}_{24}$ given by
\begin{align*}
\gamma _{24} =
\begin{pmatrix}
I_{8} &&&
\\
&&-I_{4}&
\\
&I_{4}&&
\\
&&&I_{8}
\end{pmatrix}, && \eta _{24} = \operatorname{diag}(I_{12}, -I_{2}, I_{2},
-I_{2}, I_{2}, -I_{2}, I_{2}).
\end{align*}
We will denote $g \otimes \wedge ^{2} h$ to be the image
$\mathrm{Ext}^{2}(g,h)$ in $\mathrm{Sp}_{24}$. Lastly, recalling the group
$S(\mathrm{GL}_{2} {\times }\mathrm{GL}_{2} {\times }\mathrm{GL}_{2})$
in \reftext{\eqref{eqn:_S(GL2_x_GL2_x_GL2)_VTEX1}}, we define the map
$\rho : S(\mathrm{GL}_{2} {\times }\mathrm{GL}_{2} {\times }
\mathrm{GL}_{2}) \to \mathrm{Sp}_{8}$ given by
\begin{align}
\label{eqn:_rho_map_VTEX1}
\rho (g_{1},g_{2},g_{3}) := \gamma _{\rho}(g_{3} \otimes (g_{1}
\otimes g_{2}))\gamma _{\rho}^{-1},
\end{align}
for
$(g_{1},g_{2},g_{3}) \in S(\mathrm{GL}_{2} {\times }\mathrm{GL}_{2} {
\times }\mathrm{GL}_{2})$ where
$\gamma _{\rho }= \operatorname{diag}(-1,-1,-1,-1,1,-1,-1,1)$.

\subsection{Degenerate Whittaker coefficients: Bessel and Fourier-Jacobi coefficients}
\label{sec:_Degenerate_Whittaker_coefficients_VTEX1}
In this subsection, we will recall the process of constructing degenerate
Whittaker coefficients of an automorphic form $\varphi $ of $G$, attached
to a nilpotent orbit $\mathcal{O}$ of $G$. This subsection follows that
of \citet*[\S 1.2]{MWZ2}. Let
$\iota : \mathrm{SL}_{2} \to G$ and ${\mathcal O}_{\iota}$ be the nilpotent
orbit of $\mathfrak{g}$ associated to it. It is well-known
\citep*{CM} that the collection of nilpotent orbits for classical groups
are parametrised by (certain subsets of) partitions. For instance, for
the classical group of type $A_{n-1}$, the nilpotent orbits are parametrised
by the set of all partitions of $n$, and for groups of type $B_{n}$ they
are in bijection with partitions of $2n+1$ where even parts occur with
even multiplicity. Define
\begin{align*}
U =U_{{\mathcal O}_{\iota}} = \{g \in G \mid \lim _{t\to 0}\iota (
\begin{pmatrix}
t &
\\
& t^{-1}
\end{pmatrix}) g \iota (
\begin{pmatrix}
t &
\\
& t^{-1}
\end{pmatrix}^{-1}) = 1\}.
\end{align*}
The additive character $\psi $ induces a Weil representation
${\Omega }_{\psi}$ of $U({\mathbb A})$ on the space of Schwartz functions
of $X({\mathbb A})$ (here $X({\mathbb A})$ is a Lagrangian subspace of
the weight-one space of $\mathfrak{u}$ under the adjoint action of
$\iota (\operatorname{diag}(t,t^{-1}))$). Given
$\Phi \in \mathcal{S}(X({\mathbb A}))$, we define the theta series
\begin{align}
\label{eqn:_Theta-series_VTEX1}
\Theta _{\psi}^{\Phi}(u) = \sum _{\xi \in X(k)} {\Omega }_{\psi}(u)
\Phi (\xi ),
\end{align}
and the degenerate Whittaker coefficient as
\begin{align}
\label{eqn:_degenerate_Whittaker_coefficient_VTEX1}
{\mathcal P}_{\iota}(\varphi ,\Phi ) = \int _{U(k) {\setminus }U({
\mathbb A})} \varphi (u) \Theta _{\psi}^{\Phi}(u)\,du,
\end{align}
for automorphic form $\varphi $ of $G({\mathbb A})$. For this and future
sections, we will denote $[U]= U(k) {\setminus }U({\mathbb A})$ for any
unipotent group $U$. We remark that if the nilpotent orbit
${\mathcal O}_{\iota}$ is even then the theta series
$\Theta _{\psi}^{\Phi}(u)$ is a character of
$U(k) {\setminus }U({\mathbb A})$. With these, we call the degenerate Whittaker
coefficient ${\mathcal P}_{\iota}$ a Bessel coefficient
${\mathcal B}_{\iota}$ (resp. Fourier-Jacobi coefficient
${\mathcal{FJ}}_{\iota}$) if ${\mathcal O}_{\iota}$ is even (resp. non-even).
We will give examples of such degenerate Whittaker coefficients for
$\mathrm{GL}_{n},\mathrm{GSpin}_{2n+1},\mathrm{GSO}_{8}$ and
$\mathrm{GSO}_{10}$.

\subsubsection{Bessel coefficient for general linear groups}
\label{sec:_Bessel_coefficient_for_General_Linear_Groups_VTEX1}

Let $m,k\geq 1$. Consider the even nilpotent orbit of
$\mathrm{GL}_{2m+2k+1}$ parametrised by the partition
$[2m+1, 1^{2k}]$. In this case, we can choose the unipotent subgroup
$U$ defined as
\begin{align*}
U{=} \left \{ u(A,B;v){=}
\begin{pmatrix}
I_{2k}&
\\
&v
\end{pmatrix}
\begin{pmatrix}
I_{2k}&&A
\\
&I_{m+1}&
\\
&&I_{m}
\end{pmatrix}
\begin{pmatrix}
I_{2k}&&
\\
B&I_{m}&
\\
&&I_{m+1}
\end{pmatrix} \mid
\begin{matrix}
v \in N_{\mathrm{GL}_{2m+1}},
\\
A \in \mathrm{Mat}_{2k,m},
\\
B \in \mathrm{Mat}_{m,2k}
\end{matrix} \right \}\!.
\end{align*}
We can also choose the character
$\psi _{[2m+1, 1^{2k}]} : U(k) {\setminus }U({\mathbb A}) \to {
\mathbb C}^{{\times }}$ given by
\begin{align*}
\psi _{[2m+1, 1^{2k}]}(u(A,B;v)) = \psi \left ( \sum _{i=1}^{2m} v_{i,i+1}
\right ).
\end{align*}
With these, we define the Bessel coefficient of a cusp form
$\varphi $ for $\mathrm{GL}_{2m+2k+1}({\mathbb A})$ associated to nilpotent
orbit $ [2m+1,1^{2k}]$ as
\begin{align}
\label{eqn:_Bessel_coefficient_GLn_VTEX1}
{\mathcal B}_{[2m+1, 1^{2k}]}(\varphi )= \int _{U(k) {\setminus }U({
\mathbb A})} \varphi (u) \psi _{[2m+1, 1^{2k}]}(u) \, du.
\end{align}
Then, for $g \in \mathrm{GL}_{2k}({\mathbb A})$ we write
\begin{align*}
{\mathcal B}_{[2m+1, 1^{2k}]}(\varphi )(g) = \int _{U(k) {\setminus }U({
\mathbb A})} \varphi (uj_{2m+2k+1,2k}(g)) \psi _{[2m+1,1^{2k}]}(u) \, du.
\end{align*}

\subsubsection{Fourier-Jacobi coefficient for general linear groups}
\label{sec:_Fourier-Jacobi_coefficient_for_General_Linear_Groups_VTEX1}

We will first recall some basic preliminaries on the Weil representation
following \citet*{GRS}. Let ${\mathcal H}_{n}$ be the Heisenberg group
of $2n+1$ variables. We identify an element $h \in {\mathcal H}_{n}$ with
a triple $(x,y,z)$ where $x,y \in \mathrm{Mat}_{1,n}$ are row-vectors and
$z \in \mathrm{Mat}_{1,1}$. Then, we define the following subgroups of
${\mathcal H}_{n}$:
\begin{align*}
X_{n} = \{(x,0,0) \in \mathcal H_{n}\}, && Y_{n} =\{(0,y,0) \in
\mathcal H_{n}\}, && Z_{n} = \{(0,0,z) \in \mathcal H_{n}\}.
\end{align*}
The group operation in $\mathcal H_{n}$ is given by
\begin{align*}
(x_{1},y_{1},z_{1}) (x_{2},y_{2},z_{2}) := (x_{1}+x_{2}, y_{1} + y_{2},
z_{1} + z_{2} + \frac{1}{2}(x_{1} w_{n} {}^{t} y_{2} - y_{1} w_{n} {}^{t}
x_{2})).
\end{align*}
We will denote $\omega _{\psi}$ (or $\Omega _{\psi}$) to denote the Weil
representation, which is a representation of the group
${\mathcal H}_{n}({\mathbb A}) \widetilde{\mathrm{Sp}}_{2n}({\mathbb A})$
realised on the Schwartz space $\mathcal{S}({\mathbb A}^{n})$. We have
the following explicit formulas for the Weil representation
\begin{align}
\label{eqn:_formula_Weil_-_Heisenberg_VTEX1}
\omega _{\psi}( \langle (0,y,z) (x,0,0) , {\epsilon }\rangle )\phi (
\xi ) =& {\epsilon }\psi (z + \xi w_{n} {}^{t} y) \phi (x+\xi ),
\eqncr
\label{eqn:_formula_Weil_-_Levi_Siegel_VTEX1}
\omega _{\psi}(\langle
\begin{pmatrix}
m &
\\
& m^{\ast }
\end{pmatrix}, {\epsilon }\rangle ) \phi (\xi ) =& {\epsilon }\gamma _{
\psi}(\det m) |\det m|^{1/2} \phi (\xi m),
\eqncr
\label{eqn:_formula_Weil_-_unipotent_Siegel_VTEX1}
\omega _{\psi}(\langle
\begin{pmatrix}
I_{n} & T
\\
& I_{n}
\end{pmatrix}, {\epsilon }\rangle ) \phi (\xi ) =& {\epsilon }\psi
\left ( \frac{1}{2} \xi T w_{n} {}^{t} \xi \right ) \phi (\xi ),
\end{align}
where
$\phi \in {\mathcal S}({\mathbb A}^{n}), (0,y,z)(x,0,0) \in {
\mathcal H}_{n}({\mathbb A}), {\epsilon }\in \{\pm 1\}$,
$\gamma _{\psi}$ is the Weil-index associated to $\psi $, and
$m \in \mathrm{GL}_{n}({\mathbb A}),
\begin{pmatrix}
I_{n} & T
\\
& I_{n}
\end{pmatrix}
\in \mathrm{Sp}_{2n}({\mathbb A})$. In these formulas, we view
$\xi \in {\mathbb A}^{n}$ as a row vector.

Let $m\geq 1$ and $k\geq 2$. Consider the non-even nilpotent orbit of
$\mathrm{GL}_{2m+k}$ parametrised by the partition $[2m, 1^{k}]$. In this
case, we can choose the unipotent subgroup $U$ defined as
\begin{align*}
U &= \left \{
\begin{pmatrix}
I_{k} &
\\
& v
\end{pmatrix} \left (
\begin{smallmatrix}
I_{k}&&A
\\
&I_{m+1}&
\\
&&I_{m-1}
\end{smallmatrix}\right ) \left (
\begin{smallmatrix}
I_{k}&&
\\
B&I_{m-1}&
\\
&&I_{m+1}
\end{smallmatrix}\right ) \left (
\begin{smallmatrix}
I_{k}&&Y&
\\
&I_{m}&&
\\
&&1&
\\
&&&I_{m-1}
\end{smallmatrix}\right ) \left (
\begin{smallmatrix}
I_{k}&&&
\\
&I_{m-1}&&
\\
X&&1&
\\
&&&I_{m}
\end{smallmatrix}\right )\right \},
\end{align*}
where $v \in N_{\mathrm{GL}_{2m}}, A,{}^t B \in \mathrm{Mat}_{k,m-1}, X,{}^t Y \in \mathrm{Mat}_{1,k}$, and elements of $U$ are denoted as $u(A,B,X,Y;v)$. Here, we can choose
the weight-one subspace $\mathfrak{g}_{1}$ to admit the polarisation
$X\oplus Y$ given by
\begin{align*}
X =& \langle E_{m+k,1}, E_{m+k,2},\dots , E_{m+k,k} \rangle , && Y =
\langle E_{1,m+k+1}, E_{2,m+k+1},\dots , E_{k,m+k+1} \rangle .
\end{align*}
We extend $\mathfrak{g}_{1}$ to a Heisenberg group
${\mathcal H}_{k} = \mathfrak{g}_{1} \oplus {\mathbb G}_{a}$ of
$2k+1$ variables where ${\mathbb G}_{a}$ is realised as
${\mathbb G}_{a} \cong \langle E_{k+m,k+m+1} \rangle $. Also, we define
the subgroup $U_{2,2m+k} \subset U$ consisting of elements of the form
$u(A,B,0,0;v)$ and define the character
$\psi _{[2m, 1^{k}]} : U_{2,2m+k}(k) {\setminus }U_{2,2m+k}({
\mathbb A}) \to {\mathbb C}^{{\times }}$ as
\begin{align*}
\psi _{[2m, 1^{k}]}(u(A,B,0,0;v)) = \psi \left ( \sum _{i=1}^{2m-1} v_{i,i+1}
\right ).
\end{align*}
With these, we define the Fourier-Jacobi coefficient of a cusp form
$\varphi $ for $\mathrm{GL}_{2m+k}({\mathbb A})$ associated to nilpotent
orbit $[2m,1^{k}]$ as
${\mathcal{FJ}}_{[2m, 1^{k}]}(\varphi ;\Phi _{k})$ defined by
\begin{align}
\label{eqn:_Fourier-Jacobi_coefficient_GLn_VTEX1}
{\mathcal{FJ}}_{[2m, 1^{k}]}(\varphi ;\Phi _{k})= \int _{[U_{2,2m+k}]}
\int _{[\mathrm{Mat}_{1,k}]} &\int _{[\mathrm{Mat}_{k,1}]} \varphi (u_{2}
\left (
\begin{smallmatrix}
I_{k}&&Y&
\\
&I_{m}&&
\\
&&1&
\\
&&&I_{m-1}
\end{smallmatrix}\right )\left (
\begin{smallmatrix}
I_{k}&&&
\\
&I_{m-1}&&
\\
X&&1&
\\
&&&I_{m}
\end{smallmatrix}\right ))
\\
\cdot & \Theta _{\psi}^{\Phi _{k}}((0,y,0)(x,0,0)) \psi _{[2m,1^{k}]}(u_{2})
\, du_{2}\,dX\,dY
\notag
,
\end{align}
where
$\Phi _{k} \in {\mathcal S}(X({\mathbb A})) \cong {\mathcal S}({
\mathbb A}^{k})$ and
$x = (x_{1},\dots , x_{k}) \in X({\mathbb A}), y = (y_{1},\dots , y_{k})
\in Y({\mathbb A})$ is realised in ${\mathcal H}_{k}$ as
\begin{align*}
(x_{1},\dots , x_{k}) =& \sum _{i=1}^{k} x_{i} E_{m+k,i}, && (y_{1},
\dots , y_{k}) = \sum _{i=1}^{k} y_{i} E_{k+1-i,m+k+1}.
\end{align*}
Then, for $g \in \mathrm{GL}_{k}({\mathbb A})$ we write
\begin{align*}
&{\mathcal{FJ}}_{[2m,1^{k}]}(\varphi )(g,\Phi _{k})
\\
&=\gamma _{\psi}(
\det g)^{-1}\int _{[U_{2,2m+k}]} \int _{[\mathrm{Mat}_{1,k}]} \int _{[
\mathrm{Mat}_{k,1}]} \varphi (u_{2} \left (
\begin{smallmatrix}
I_{k}&&Y&
\\
&I_{m}&&
\\
&&1&
\\
&&&I_{m-1}
\end{smallmatrix}\right ) \left (
\begin{smallmatrix}
I_{k}&&&
\\
&I_{m-1}&&
\\
X&&1&
\\
&&&I_{m}
\end{smallmatrix}\right ) j_{2m+k,k}(g))
\\
&\quad \cdot \Theta _{\psi}^{\Phi _{k}}((0,y,0)(x,0,0)\iota _{k}(g)) \psi _{[2m,
1^{k}]}(u_{2})\, du_{2}\,dX\,dY.
\end{align*}

\subsubsection{Bessel coefficients on spin similitude groups}
\label{sec:_Bessel_coefficient_GSpin_VTEX1}

We will introduce two Bessel coefficients associated to the nilpotent orbits
$[2n-3,1^{4}]$ and $[2n-5,1^{6}]$ of $\mathrm{GSpin}_{2n+1}$. Recall that
there is an isomorphism between unipotent subgroups of
$\mathrm{GSpin}_{2n+1}$ and $\mathrm{SO}_{2n+1}$, so we will express unipotent
subgroups of $\mathrm{GSpin}_{2n+1}$ as unipotent matrix subgroups of
$\mathrm{SO}_{2n+1}$. As in \citet*{G}, we define $U_{n,k}$ as the subgroup
\begin{align*}
U_{n,k} &= \left \{ u(A,B;v)=
\begin{pmatrix}
I_{k} &&
\\
& v &
\\
&&I_{k}
\end{pmatrix} \left (
\begin{smallmatrix}
I_{k} &&&A&
\\
&I_{n-k}&&&A'
\\
&&1&&
\\
&&&I_{n-k}&
\\
&&&&I_{k}
\end{smallmatrix}\right ) \left (
\begin{smallmatrix}
I_{k} &&&&
\\
B&I_{n-k}&&&
\\
&&1&&
\\
&&&I_{n-k}&
\\
&&&B'&I_{k}
\end{smallmatrix}\right )\right \},
\end{align*}
where $v \in N_{\mathrm{SO}_{2n-2k+1}}, A, {}^t B \in \mathrm{Mat}_{k,n-k}$, and  $A'$ and $B'$ are matrices such that
$u(A,B;v) \in \mathrm{SO}_{2n+1}$. Also, we define the character
$\psi _{n,k}$ by
\begin{align*}
\psi _{n,k}: U_{n,k}(k){\setminus }U_{n,k}({\mathbb A}) \to {
\mathbb C}^{{\times }}; && \psi _{n,k}(u(A,B;v)) =& \psi \left (
\sum _{i=1}^{n-k} v_{i,i+1} \right ).
\end{align*}
With these, we define the Bessel coefficient of a cusp form
$\varphi $ for $\mathrm{GSpin}_{2n+1}({\mathbb A})$ associated to nilpotent
orbit $[2n-3,1^{4}], n\geq 2$ (resp. $[2n-5,1^{6}], n\geq 3$) as
\begin{align}
\label{eqn:_Bessel_GSpin_[2n-3,1^4]_VTEX1}
{\mathcal B}_{[2n-3,1^{4}]}(\varphi ) =&\int _{U_{n,2}(k) {\setminus }U_{n,2}({
\mathbb A})} \varphi (u) \psi _{n,2}(u)\,du,
\eqncr
\label{eqn:_Bessel_GSpin_[2n-5,1^6]_VTEX1}
{\mathcal B}_{[2n-5,1^{6}]}(\varphi ) =&\int _{U_{n,3}(k) {\setminus }U_{n,3}({
\mathbb A})} \varphi (u) \psi _{n,3}(u)\,du.
\end{align}
Furthermore given $g_{1} \in \mathrm{GSpin}_{4}({\mathbb A})$ and
$g_{2} \in \mathrm{GSpin}_{6}({\mathbb A})$, we write
\begin{align*}
{\mathcal B}_{[2n-3,1^{4}]}(\varphi )(j_{n}(g_{1})) = \int _{U_{n,2}(k)
{\setminus }U_{n,2}({\mathbb A})} \varphi (uj_{n}(g_{1})) \psi _{n,2}(u)
\, du, && n\geq 2,
\\
{\mathcal B}_{[2n-5,1^{6}]}(\varphi )(j_{n}(g_{2})) =\int _{U_{n,3}(k)
{\setminus }U_{n,3}({\mathbb A})} \varphi (u j_{n}(g_{2})) \psi _{n,3}(u)
\,du, && n \geq 3,
\end{align*}
where $j_{n}$ is the embedding of
$\mathrm{GSpin}_{4}, \mathrm{GSpin}_{6} \hookrightarrow
\mathrm{GSpin}_{2n+1}$.

\subsubsection{Fourier-Jacobi coefficient of special orthogonal similitude groups}
\label{sec:_FJ_coefficient_of_GSO_VTEX1}

We will introduce the Fourier-Jacobi coefficient associated to the nilpotent
orbit $[2^{2},1^{4}]$ of $\mathrm{GSO}_{8}$ (resp. $[4^{2},1^{2}]$ of
$\mathrm{GSO}_{10}$). We choose the unipotent subgroup
$U \subset \mathrm{GSO}_{8}$ associated to the nilpotent orbit
$[2^{2},1^{4}]$ of $\mathrm{GSO}_{8}$ defined as
\begin{align*}
U &= \left \{u(X,Y;t)=
\begin{pmatrix}
I_{2} &&&
\\
&I_{2} & T &
\\
&&I_{2}&
\\
&&&I_{2}
\end{pmatrix}
\begin{pmatrix}
I_{2} &&Y&
\\
&I_{2} & & Y'
\\
&&I_{2}&
\\
&&&I_{2}
\end{pmatrix}
\begin{pmatrix}
I_{2}&&&
\\
X & I_{2} &&
\\
&&I_{2}&
\\
&&X'&I_{2}
\end{pmatrix} \right \},
\end{align*}
where $T = \operatorname{diag}(t,-t), Y = \left(\begin{smallmatrix} y_3 & y_4\\ y_1 & y_2\end{smallmatrix}\right), X = \left(\begin{smallmatrix} x_2 & x_4\\ -x_1 & -x_3\end{smallmatrix}\right)$, and $X',Y'$ are matrices such that $u(X,Y;t) \in \mathrm{GSO}_{8}$. Here,
we can choose the weight-one subspace $\mathfrak{g}_{1}$ to admit the polarisation
$X \oplus Y$ given by
\begin{align*}
&X = \langle -E_{4,1} + E_{8,5}, E_{3,1} - E_{8,6}, -E_{4,2} + E_{7,5},
E_{3,2} - E_{7,6} \rangle ,
\\
&Y = \langle E_{2,5} - E_{4,7}, E_{2,6} - E_{3,7}, E_{1,5} - E_{4,8}, E_{1,6}
- E_{3,8} \rangle ,
\end{align*}
and we extend $\mathfrak{g}_{1}$ to a Heisenberg group
${\mathcal H}_{4} = \mathfrak{g}_{1} \oplus {\mathbb G}_{a}$ of $9$ variables
where ${\mathbb G}_{a}$ is realised as
${\mathbb G}_{a} \cong \langle E_{3,5} - E_{4,6} \rangle $. Also, we define
$U_{2} \subset U$ as the subgroup consisting of $u(0,0;t)$ and define the
character
$\psi _{[2^{2},1^{4}]} : U_{2}(k){\setminus }U_{2}({\mathbb A}) \to {
\mathbb C}^{{\times }}$ as
\begin{align*}
\psi _{[2^{2},1^{4}]}(u(0,0;t)) = \psi (t).
\end{align*}
With these, we define the Fourier-Jacobi coefficient of a cusp form
$\varphi $ for $\mathrm{GSO}_{8}({\mathbb A})$ associated to nilpotent
orbit $[2^{2},1^{4}]$ as
\begin{align}
\label{eqn:_Fourier-Jacobi_GSO8_VTEX1}
{\mathcal{FJ}}_{[2^{2},1^{4}]}(\varphi ;\Phi _{4})= \int _{U(k) {
\setminus }U({\mathbb A})} \varphi (u) \Theta _{\psi}^{\Phi _{4}}((0,y,0)(x,0,0))
\psi (t)\, du,
\end{align}
where
$\Phi _{4} \in {\mathcal S}(X({\mathbb A})) \cong {\mathcal S}({
\mathbb A}^{4})$ and
$x = (x_{1},\dots , x_{4}) \in X({\mathbb A}), y = (y_{1},\dots , y_{4})
\in Y({\mathbb A})$ is realised in ${\mathcal H}_{4}$ as
\begin{align*}
(x_{1},\dots , x_{4})
&= x_{1} (-E_{4,1} + E_{8,5}) +x_{2} (E_{3,1} - E_{8,6})
+x_{3} (-E_{4,2} + E_{7,5})
\\
&\quad +x_{4} (E_{3,2} - E_{7,6}),
\\
(y_{1},\dots , y_{4})
&= y_{1} (E_{2,5} - E_{4,7}) +y_{2} (E_{2,6} - E_{3,7})
+y_{3} (E_{1,5} - E_{4,8}) +y_{4} (E_{1,6} - E_{3,8}).
\end{align*}
Then, for
$(g_{1},g_{2}) \in S(\mathrm{GL}_{2} {\times }\mathrm{GSO}_{4})$ defined
in \reftext{\eqref{eqn:_S(GL2_x_GSO4)_VTEX1}} we write
\begin{align*}
&\mathcal{FJ}_{[2^{2},1^{4}]}(\varphi )(J_{D_{4}}(g_{1},g_{2});\Phi _{4})
\\
&\quad = \int _{U(k) {\setminus }U({\mathbb A})} \varphi (uJ_{D_{4}}(g_{1},g_{2}))
\Theta _{\psi}^{\Phi _{4}}((0,y,0)(x,0,0)\iota _{D_{4}}(g_{1},g_{2}))
\psi (t)\,du,
\end{align*}
where
$J_{D_{4}}:S(\mathrm{GL}_{2} {\times }\mathrm{GSO}_{4}) \to
\mathrm{GSO}_{8}$ and
$\iota _{D_{4}} : S(\mathrm{GL}_{2} {\times }\mathrm{GSO}_{4}) \to
\mathrm{Sp}_{8}$ are as defined in \reftext{\eqref{eqn:_J-D4_VTEX1}} and \reftext{\eqref{eqn:_iota-D4_VTEX1}} respectively. Finally, for the group
$\mathrm{GSO}_{10}$ we first define some auxiliary subgroups. Let
$V,W,Z,X$ and $Y$ be unipotent subgroups of $\mathrm{GSO}_{10}$ given by
\begin{align*}
V &= \left \{ \mathbf v(v_{1},\dots , v_{7})
\right.
\\
&\quad \left.
= x_{\alpha _{2}}(v_{1}) x_{
\alpha _{3}}(v_{2}) x_{\alpha _{4}}(v_{3}) x_{\alpha _{1}+\alpha _{2}+
\alpha _{3}+\alpha _{4}}(v_{4}) x_{\alpha _{2}+\alpha _{3}}(v_{5}) x_{
\alpha _{2}+\alpha _{3}+\alpha _{4}}(v_{6}) x_{\alpha _{3}+\alpha _{4}}(v_{7})
\right \},
\\
W &= \{\mathbf w(w_{1},w_{2},w_{3}) = x_{\alpha _{1}+2\alpha _{2}+2
\alpha _{3}+\alpha _{4}+\alpha _{5}}(w_{1}) x_{\alpha _{2}+\alpha _{3}+
\alpha _{4}+\alpha _{5}}(w_{2}) x_{\alpha _{2}+2\alpha _{3}+\alpha _{4}+
\alpha _{5}}(w_{3})\},
\\
Z &= \{\mathbf z(z_{1},z_{2},z_{3},z_{4}) = x_{-\alpha _{1}}(z_{1}) x_{-
\alpha _{1}-\alpha _{2}-\alpha _{3}-\alpha _{5}}(z_{2}) x_{-\alpha _{3}-
\alpha _{5}}(z_{3}) x_{-\alpha _{5}}(z_{4})\},
\end{align*}
and
\begin{align*}
&X = \{\mathbf x(x_{1},x_{2}) = x_{-\alpha _{1}-\alpha _{2}-\alpha _{3}-
\alpha _{4}-\alpha _{5}}(x_{1}) x_{-\alpha _{1}-\alpha _{2}}(x_{2})\},
\\
&Y = \{\mathbf y(y_{1},y_{2}) = x_{\alpha _{1}+\alpha _{2}+\alpha _{3}}(y_{1})
x_{\alpha _{1}+\alpha _{2}+2\alpha _{3}+\alpha _{4}+\alpha _{5}}(y_{2})
\}.
\end{align*}
With these, we define the unipotent subgroup
$U \subset \mathrm{GSO}_{10}$ associated to the nilpotent orbit
$[4^{2},1^{2}]$ of $\mathrm{GSO}_{10}$ defined as
\begin{align*}
U = V W Z Y X.
\end{align*}
Hence, we can choose the weight-one subspace $\mathfrak{g}_{1}$ to admit
the polarisation $X \oplus Y$ given by
\begin{align*}
X = \langle -E_{7,1} + E_{10,4}, E_{3,1} - E_{10,8} \rangle , && Y =
\langle E_{1,4} - E_{7,10}, E_{1,8} - E_{3,10} \rangle ,
\end{align*}
and we extend $\mathfrak{g}_{1}$ to a Heisenberg group
${\mathcal H}_{2} = \mathfrak{g}_{1} \oplus {\mathbb G}_{a}$ of $5$ variables
where ${\mathbb G}_{a}$ is realised as
${\mathbb G}_{a} \cong \langle E_{3,4} - E_{7,8} \rangle $. Also, we define
$U_{2} = VWZ$ a subgroup of $U$ and define the character
$\psi _{[4^{2},1^{2}]} : U_{2}(k) {\setminus }U_{2}({\mathbb A}) \to {
\mathbb C}^{{\times }}$ as
\begin{align*}
\psi _{[4^{2},1^{2}]}(u_{2}) = \psi (v_{1} + v_{2} + v_{3}),
\end{align*}
where
$u_{2} = \mathbf v(v_{1},\dots , v_{7}) \mathbf w(w_{1},w_{2},w_{3})
\mathbf z(z_{1},z_{2},z_{3},z_{4})$. With these, we define the Fourier-Jacobi
coefficient of the cusp form $\varphi $ for
$\mathrm{GSO}_{10}({\mathbb A})$ associated to nilpotent orbit
$[4^{2},1^{2}]$ as
\begin{align}
\label{eqn:_Fourier_Jacobi_GSO10_VTEX1}
&{\mathcal{FJ}}_{[4^{2},1^{2}]}(\varphi ;\Phi _{2})
\nonumber
\\
&\quad =\int _{[V]} \int _{[W]}
\int _{[Z]} \int _{[Y]} \int _{[X]} \varphi (v w z y x) \Theta _{\psi}^{
\Phi _{2}}((0,y,0)(x,0,0)) \psi _{[4^{2},1^{2}]}(v w z)\,dx\, dy\,dz
\,dw\,dv,
\end{align}
where
$\Phi _{2} \in {\mathcal S}(X({\mathbb A})) \cong {\mathcal S}({
\mathbb A}^{2})$ and
$x = (x_{1},x_{2}) \in X({\mathbb A}), y = (y_{1},y_{2}) \in Y({
\mathbb A})$ is realised in ${\mathcal H}_{2}$ as
\begin{align*}
(x_{1},x_{2}) &= x_{1}(-E_{7,1} + E_{10,4}) + x_{2} (E_{3,1} - E_{10,8}),
\\
(y_{1},y_{2}) &= y_{1}(E_{1,4} - E_{7,10}) + y_{2}(E_{1,8} - E_{3,10}).
\end{align*}
Then, for $g \in \mathrm{GL}_{2}({\mathbb A})$ we write
\begin{align*}
{\mathcal{FJ}}_{[4^{2},1^{2}]}(\varphi )(g;\Phi _{2})&= \gamma _{\psi}(
\det g)^{-1} \int _{[V]} \int _{[W]} \int _{[Z]} \int _{[Y]} \int _{[X]}
\varphi (v w z y x J_{D_{5}}(g))
\\
&\quad \cdot \Theta _{\psi}^{\Phi _{2}}((0,y,0)(x,0,0)\iota _{2}((\det g) g^{
\ast})) \psi _{[4^{2},1^{2}]}(v w z) \, dx\,dy\,dz\,dw\,dv.
\end{align*}

\section{Root exchanges and auxiliary integrals}
\label{sec:_Root_Exchanges_and_Auxiliary_Integrals_VTEX1}

In this section, we will perform the root exchanges for the degenerate
Whittaker coefficients defined in the previous section and introduce some
auxiliary integrals used in the construction of the global period integrals
and the computations of the local unramified integrals. For the root exchanges
of the degenerate Whittaker coefficients, we will treat them separately
in two cases: the Bessel coefficient and the Fourier-Jacobi coefficient
similar to the previous section. The root exchanges for the Bessel coefficient
follow that of \citep*{LX,IT,MS,G}, whereas for the Fourier-Jacobi case,
we follow a similar argument as those in \citep*[\S 2]{GRS} and
\citep*[\S 3]{GJRS}.

\subsection{Root exchanges for Bessel coefficients}
\label{sec:_root_exch_for_Bessel_VTEX1}

To start, we recall the Bessel coefficient
${\mathcal B}_{[2m+1, 1^{2k}]}(\varphi )$ of a cusp form $\varphi $ for
$\mathrm{GL}_{2m+2k+1}({\mathbb A})$ given in \reftext{\eqref{eqn:_Bessel_coefficient_GLn_VTEX1}} as well as the Bessel coefficients
${\mathcal B}_{[2n+1-2k,1^{2k}]}(\varphi )$ of a cusp form
$\varphi $ for $\mathrm{GSpin}_{2n+1}({\mathbb A})$ for
$k\in \{2,3\}, n \geq k+1$ given in \reftext{\eqref{eqn:_Bessel_GSpin_[2n-3,1^4]_VTEX1}} and \reftext{\eqref{eqn:_Bessel_GSpin_[2n-5,1^6]_VTEX1}} respectively. We have the following
lemma due to \citet*[Corollary 4.3.3]{LX}, \citet*[Lemma 3]{G} and also
\citet*[\S 6.1]{ACS}.
\begin{lem}
[Corollary 4.3.3 of \citet*{LX}, Lemma 3 of \citet*{G} and Section 6.1 of \citet*{ACS}]
\label{lem:_root_exchange_Bessel_VTEX1}
The Bessel coefficient ${\mathcal B}_{[2m+1,1^{2k}]}(\varphi )$ of a cusp
form $\varphi $ of $\mathrm{GL}_{2m+2k+1}({\mathbb A})$ can be reduced
to
\begin{align*}
&{\mathcal B}_{[2m+1,1^{2k}]}(\varphi ) =\int _{\mathrm{Mat}_{m,2k}({\mathbb A})}  \int _{ [U_{(2k+1,1^{2m})}']
} \varphi (u \left (
\begin{smallmatrix}
I_{2k}&&
\\
B&I_{m}&
\\
&&I_{m+1}
\end{smallmatrix}\right )) \psi _{[2m+1,1^{2k}]}(u)\,du\,dB,
\end{align*}
where $U_{(2k+1,1^{2m})}'$ is the unipotent subgroup of
$\mathrm{GL}_{2m+2k+1}$ given by
\begin{align*}
U_{(2k+1,1^{2m})}' = \left \{ u(A',v) =
\begin{sgmlfig}
\begin{pNiceArray}{c|cc}
I_{2k} & 0 & A'
\\
\hline
& \Block{2-2}<\Large >{{v}}
\\
&
\end{pNiceArray}
\end{sgmlfig}
\mid v \in N_{\mathrm{GL}_{2m+1}}, A'\in
\mathrm{Mat}_{2k,2m}, 0 \in \mathrm{Mat}_{2k,1} \right \},
\end{align*}
and the character
$\psi _{[2m+1,1^{2k}]} : [U_{(2k+1,1^{2m})}'] \to {\mathbb C}^{{
\times }}$ is given by
\begin{align*}
\psi _{[2m+1,1^{2k}]}(u) = \psi \left ( \sum _{i=1}^{2m} v_{i,i+1}
\right ).
\end{align*}
Similarly, for $k=2,3$ we set $Z_k = \langle U_{e_i + e_j}: 1 \leq i < j \leq k\rangle$, then the Bessel coefficients
${\mathcal B}_{[2n+1-2k,1^{2k}]}(\varphi )$ of a cusp form
$\varphi $ of $\mathrm{GSpin}_{2n+1}({\mathbb A})$ can be reduced to
\begin{align*}
&\int_{[Z_k]}{\mathcal B}_{[2n+1-2k,1^{2k}]}(\varphi)(z) \, dz=
\int_{\mathrm{Mat}_{n-k,k}(\mathbb A)}
\int _{[N_{\mathrm{SO}_{2n+1-2k}}]}
\int _{[\mathrm{Mat}_{k,n-k-1}]} 
\int _{[U_{s,k}]}
\int _{[\mathrm{Mat}_{k, n-k}]}  
\\
\cdot &\psi _{[2n+1-2k,1^{2k}]}(v)
\varphi (w(v,C,u,A)\left (
\begin{smallmatrix}
I_{k} &&&&
\\
B&I_{n-k}&&&
\\
&&1&&
\\
&&&I_{n-k}&
\\
&&&B'&I_{k}
\end{smallmatrix}\right ))
 \,dA\,du\,dC\,dv\,dB,
\end{align*}
where
\begin{align*}
w(v,C,u,A) = \left (
\begin{smallmatrix}
I_{k} &&
\\
& v &
\\
&& I_{k}
\end{smallmatrix} \right ) \left (
\begin{smallmatrix}
I_{k}&&C&&&&
\\
&1&&&&&
\\
&&I_{n-k-1}&&&&
\\
&&&1&&&
\\
&&&&I_{n-k-1}&&C'
\\
&&&&&1&
\\
&&&&&&I_{k}
\end{smallmatrix}\right )u \left (
\begin{smallmatrix}
I_{k} &&&A&
\\
&I_{n-k}&&&A'
\\
&&1&&
\\
&&&I_{n-k}&
\\
&&&&I_{k}
\end{smallmatrix}\right ),
\end{align*}
and $U_{s,k}$ is the unipotent subgroup generated by the root subgroups
corresponding to the short roots $e_{1},\dots , e_{k}$ of
$\mathrm{SO}_{2n+1}$ and
\begin{align*}
\psi _{[2n+1-2k,1^{2k}]}(v) = \psi \left ( \sum _{i=1}^{n-k} v_{i,i+1}
\right ), && v \in N_{\mathrm{SO}_{2n+1-2k}}({\mathbb A}).
\end{align*}
\end{lem}

\subsection{Root exchanges for Fourier-Jacobi coefficients}
\label{sec:_root_exch_for_Fourier-Jacobi_VTEX1}

Next, we recall the Fourier-Jacobi coefficients given in Sections~\ref{sec:_Fourier-Jacobi_coefficient_for_General_Linear_Groups_VTEX1} and \ref{sec:_FJ_coefficient_of_GSO_VTEX1}. We have a similar lemma.
\begin{lem}
\label{lem:_root_exchanges_Fourier-Jacobi_VTEX1}
The term ${\mathcal{FJ}}_{[2m,1^{k}]}(\varphi )(g;\Phi _{k})$ given above
can be reduced to
\begin{align*}
 &\gamma _{\psi}^{-1}(
\det g) 
\int _{\mathrm{Mat}_{m-1,k}({\mathbb A})} \int _{\mathrm{Mat}_{1,k}({
\mathbb A})}
\int _{[N_{\mathrm{GL}_{2m}}]} \int _{[\mathrm{Mat}_{k,2m-1}]} \, dA'\, dv \, dX\,dB'\, 
{\Omega }_{\psi}(\iota _{k}(g)) \Phi _{k}(x) 
\\
& \cdot \psi _{[2m,1^{k}]}(v) \varphi (
\begin{pmatrix}
I_{k} &
\\
& v
\end{pmatrix} \left (
\begin{smallmatrix}
I_{k}&&A'
\\
& 1 &
\\
&&I_{2m-1}
\end{smallmatrix}\right ) \left (
\begin{smallmatrix}
I_{k}&&
\\
B'&I_{m-1}&
\\
&&I_{m+1}
\end{smallmatrix}\right ) \left (
\begin{smallmatrix}
I_{k}&&&
\\
&I_{m-1}&&
\\
X&&1&
\\
&&&I_{m}
\end{smallmatrix}\right ) j_{2m+k,k}(g)).
\end{align*}
Similarly, the terms
${\mathcal{FJ}}_{[2^{2},1^{4}]}(\varphi )(J_{D_{4}}(g_{1},g_{2});
\Phi _{4})$ and
${\mathcal{FJ}}_{[4^{2},1^{2}]}(\varphi )(g; \Phi _{2})$ given above can
be reduced to
\begin{align*}
&{\mathcal{FJ}}_{[2^{2},1^{4}]}(\varphi )(J_{D_{4}}(g_{1},g_{2});\Phi _{4})
=
\int _{[\mathrm{Mat}_{2,2}]} \int _{k{\setminus }{\mathbb A}} \int _{
\mathrm{Mat}_{2,2}({\mathbb A})}
\\
&\qquad  \varphi ( u(0,Y;t) u(X,0;0) J_{D_{4}}(g_{1},g_{2})) {\Omega }_{\psi}(
\iota _{D_{4}}(g_{1},g_{2})) \Phi _{4}(x) \psi (t) \,dX\,dt\,dY,
\end{align*}
\begin{align*}
&{\mathcal{FJ}}_{[4^{2},1^{2}]}(\varphi )(g; \Phi _{2})
 = \gamma _{\psi}(
\det g)^{-1} \int _{[V]} \int _{[W]} \int _{[Z]} \int _{[Y]} \int _{X({
\mathbb A})}
\\
&\varphi (v w z y x J_{D_{5}}(g)){\Omega }_{\psi}(\iota _{2}((
\det g) g^{\ast}))\Phi _{2}(x) \psi _{[4^{2},1^{2}]}(v w z) \, dx\,dy
\,dz\,dw\,dv.
\end{align*}
\end{lem}
\begin{proof}
The proof of the identities for the Fourier-Jacobi coefficient
${\mathcal{FJ}}_{[2^{2},1^{4}]}(\varphi )$ and
${\mathcal{FJ}}_{[4^{2},1^{2}]}(\varphi )$ of a cusp form $\varphi $ of
$\mathrm{GSO}_{8}({\mathbb A})$ and $\mathrm{GSO}_{10}({\mathbb A})$ respectively
are identical to that in \citep*[\S 2]{GRS} and also
\citep*[\S 3]{GJRS}. For convenience of the reader we will provide
a proof of the identity for the Fourier-Jacobi coefficient of
$\mathrm{GSO}_{8}({\mathbb A})$. From the formula \reftext{\eqref{eqn:_formula_Weil_-_Heisenberg_VTEX1}}, we can rewrite the integral defining
$\mathcal{FJ}_{[2^{2},1^{4}]}(\varphi )(J_{D_{4}}(g_{1},g_{2});\Phi _{4})$
as
\begin{align*}
&\mathcal{FJ}_{[2^{2},1^{4}]}(\varphi )(J_{D_{4}}(g_{1},g_{2});\Phi _{4})
=
\int _{k{\setminus }{\mathbb A}} \int _{[\mathrm{Mat}_{2,2}]} \int _{[
\mathrm{Mat}_{2,2}]} \varphi (u(t) u(Y) u(X) J_{D_{4}}(g_{1},g_{2}))
\\
&\qquad \cdot \sum _{\xi \in X(k)} {\Omega }_{\psi }( (0,y,t + \xi w_{4} {}^{t}
y) (x+\xi ,0,0) \iota _{D_{4}}(g_{1},g_{2}) ) \Phi _{4}(0)\,dX\,dY\,dt,
\end{align*}
where $u(t) = u(0,0;t), u(Y) = u(0,Y;0)$ and $u(X) = u(X,0;0)$. Also, by
using the left $U(k)$-invariance of the cusp form $\varphi $, we have
$\varphi (u(t) u(Y) u(X) J_{D_{4}}(g_{1},g_{2}))$ evaluating to
\begin{align*}
\varphi (u(\xi ) u(t) u(Y) u(X) J_{D_{4}}(g_{1},g_{2})) =\varphi (u(t +
\xi w_{4} {}^{t} y) u(Y) u(X + \xi )J_{D_{4}}(g_{1},g_{2})),
\end{align*}
where $u(\xi ) = u(\xi ,0;0)$ for $\xi \in \mathrm{Mat}_{2,2}(k)$. Then,
performing suitable changes of variables and collapsing with the summation
over $X(k)$ with the integration over
$\mathrm{Mat}_{2,2}(k) {\setminus }\mathrm{Mat}_{2,2}({\mathbb A})$, we
obtain the desired identity. As for the identity for the Fourier-Jacobi
coefficient ${\mathcal{FJ}}_{[2m,1^{k}]}(\varphi )(g;\Phi _{k})$ of
$\mathrm{GL}_{2m+k}({\mathbb A})$, it follows from the same argument as
above followed by that in \reftext{Lemma~\ref{lem:_root_exchange_Bessel_VTEX1}}.
\end{proof}

\subsection{Group actions and double coset decomposition}
\label{sec:_Group_actions_and_Double_coset_decomposition_VTEX1}

In this subsection, we will discuss some group actions and double coset
decomposition that are used in the unfolding of global period integrals
defined in subsequent sections. To start, we study the following double
coset decomposition
$P_{12}(k) {\setminus }\mathrm{GL}_{12}(k)/ H'_{4,3}(k)$ and
$P_{12}(k) {\setminus }\mathrm{GL}_{12}(k)/ H_{4,3}(k)$ where
$H'_{4,3}(k) = \mathrm{GL}_{4}'(k) \otimes \mathrm{GL}_{3}(k), H_{4,3}(k)
= \mathrm{GSp}_{4}(k) \otimes \mathrm{GL}_{3}(k)$ and $P_{12}(k)$ is the
standard parabolic of $\mathrm{GL}_{12}(k)$ of type $(11,1)$ given in \reftext{\eqref{eqn:_GLn_parabolic_subgroup_VTEX1}}.

These double coset decompositions can be easily derived from the double
coset decomposition
 $P_{mn}(k){\setminus }\mathrm{GL}_{mn}(k)/ (\mathrm{GL}_{m}(k)
\otimes \mathrm{GL}_{n}(k))$ studied in \citep*[\S 3]{Ha}, where
the orbits can be characterised by the $\mathrm{rank}(X)$ for
$0 \neq X \in \mathrm{Mat}_{m,n}(k)$. Moreover, by elementary linear algebra
one can characterise the orbits of the double coset decomposition
$P_{12}(k) {\setminus }\mathrm{GL}_{12}(k)/ H_{4,3}(k)$ by the set
$\mathcal{R}_{4,3}$ given by
\begin{align*}
\mathcal{R}_{4,3} = \{(\mathrm{rank}({}^{t} X j_{4} X), \mathrm{rank}(X))
\mid 0 \neq X \in \mathrm{Mat}_{4,3}(k) \}.
\end{align*}
We shall record this as a lemma.
\begin{lem}
\label{lem:_Double_coset_decomp_VTEX1}
We have the following double coset decompositions
\begin{align*}
&\mathrm{GL}_{12}(k) = \bigsqcup _{r=1}^{3}
P_{12}(k) \gamma _{r} H'_{4,3}(k),
\\
& 
\mathrm{GL}_{12}(k) = \bigsqcup _{t=1}^{4} P_{12}(k)
\omega _{t} H_{4,3}(k),
\end{align*}
where $\gamma _{r}$ and $\omega _{t}$ are elements of
$\mathrm{GL}_{12}$ given by
\begin{align*}
\gamma _{1} =&\omega _{1}= I_{12} - E_{1,1} - E_{12,12} + E_{1,12} + E_{12,1},
\\
\gamma _{2} =&\omega _{2}= I_{12} - E_{1,1} - E_{12,12} + E_{1,12} + E_{12,1}
+ E_{12,5},
\\
\gamma _{3} =&\omega _{3}=
\begin{pmatrix}
I_{5} &&
\\
&&I_{6}
\\
&1&
\end{pmatrix}(I_{12} + E_{6,8} +E_{6,10}),
\\
\omega _{4} =&I_{12} - E_{1,1} - E_{12,12} + E_{1,12} + E_{12,1} + E_{12,11}.
\end{align*}
Moreover, setting
$H'_{\gamma _{r}}(k) = \gamma _{r}^{-1} P_{12}(k) \gamma _{r} \cap H_{4,3}'(k)
, H_{\omega _{t}}(k) = \omega _{t}^{-1} P_{12}(k) \omega _{t} \cap H_{4,3}(k)$
we have
\begin{align*}
H'_{\gamma _{1}}(k) =& \left \{(
\begin{pmatrix}
a &
\\
c & D
\end{pmatrix}\otimes
\begin{pmatrix}
\alpha & 0
\\
\gamma & {\delta }
\end{pmatrix} )\mid
\begin{matrix}
a,\alpha \in \mathrm{GL}_{1}(k), {\delta }\in \mathrm{GL}_{2}(k), D
\in \mathrm{GL}_{3}(k),
\\
c \in \mathrm{Mat}_{3,1}(k), \gamma \in \mathrm{Mat}_{2,1}(k),
\\
\text{$a\det (D)$ is a square}
\end{matrix} \right \},
\\
H'_{\gamma _{2}}(k) =& \left \{ (
\begin{pmatrix}
A &
\\
C & D
\end{pmatrix}\otimes
\begin{pmatrix}
z{}^{t} A^{-1}&
\\
r & s
\end{pmatrix} ) \mid
\begin{matrix}
z,s \in \mathrm{GL}_{1}(k), A,D \in \mathrm{GL}_{2}(k),
\\
C \in \mathrm{Mat}_{2,2}(k), r \in \mathrm{Mat}_{1,2}(k),
\\
\text{$\det (AD)$ is a square}
\end{matrix} \right \},
\\
H'_{\gamma _{3}}(k) =& \left \{ (
\begin{pmatrix}
a & B
\\
& c_{0} D
\end{pmatrix}\otimes D^{\ast}) \mid
\xch{\begin{matrix}
a,c_{0} \in \mathrm{GL}_{1}(k), D \in \mathrm{GL}_{3}(k), B \in
\mathrm{Mat}_{1,3}(k),
\\
\text{$a\det (c_{0}D)$ is a square}
\end{matrix}
}{\begin{matrix}
a,c_{0} \in \mathrm{GL}_{1}(k), D \in \mathrm{GL}_{3}(k), B \in
\mathrm{Mat}_{1,3}(k), \text{$a\det (c_{0}D)$ is a square}
\end{matrix}
}\right \},
\end{align*}
and
\begin{align*}
H_{\omega _{1}}(k)=&\left \{(
\begin{pmatrix}
a &
\\
c & D
\end{pmatrix}\otimes
\begin{pmatrix}
\alpha &
\\
\gamma & {\delta }
\end{pmatrix}) \mid
\xch{\begin{matrix}
a, \alpha \in \mathrm{GL}_{1}(k),
\\
{\delta }\in \mathrm{GL}_{2}(k),
\\
c \in \mathrm{Mat}_{3,1}(k),
\\
D \in \mathrm{Mat}_{3,3}(k),
\\
\gamma \in \mathrm{Mat}_{2,1}(k),
\end{matrix}}{\begin{matrix}
a, \alpha \in \mathrm{GL}_{1}(k), {\delta }\in \mathrm{GL}_{2}(k),
\\
c \in \mathrm{Mat}_{3,1}(k), D \in \mathrm{Mat}_{3,3}(k), \gamma \in
\mathrm{Mat}_{2,1}(k),
\end{matrix}}%,
\begin{pmatrix}
a &
\\
c & D
\end{pmatrix} \in \mathrm{GSp}_{4}\right \},
\\
H_{\omega _{2}}(k)=&\left \{(
\begin{pmatrix}
A &
\\
C & D
\end{pmatrix}\otimes
\begin{pmatrix}
z {}^{t} A^{-1} &
\\
r & s
\end{pmatrix})\mid
\xch{\begin{matrix}
z,s, \in \mathrm{GL}_{1}(k),
\\
A \in \mathrm{GL}_{2}(k)
\\
C, D \in \mathrm{Mat}_{2,2}(k),
\\
r \in \mathrm{Mat}_{1,2}(k),
\end{matrix}}{\begin{matrix}
z,s, \in \mathrm{GL}_{1}(k),A \in \mathrm{GL}_{2}(k)
\\
C, D \in \mathrm{Mat}_{2,2}(k), r \in \mathrm{Mat}_{1,2}(k),
\end{matrix}}%,
\begin{pmatrix}
A &
\\
C & D
\end{pmatrix} \in \mathrm{GSp}_{4}(k) \right \},
\\
H_{\omega _{3}}(k) =&\left \{ (
\begin{pmatrix}
a & u
\\
0 & c_{0}A
\end{pmatrix}\otimes A^{\ast})\mid a,c_{0} \in \mathrm{GL}_{1}(k), A
\in \mathrm{GL}_{3}(k),
\begin{pmatrix}
a & u
\\
0 & c_{0}A
\end{pmatrix}\in \mathrm{GSp}_{4}(k) \right \},
\\
H_{\omega _{4}}(k)=&\left \{ (
\begin{pmatrix}
a_{11}&&a_{12}
\\
&D &
\\
a_{21} && a_{22}
\end{pmatrix}\otimes
\begin{pmatrix}
z{}^{t} A^{-1} &
\\
r & s
\end{pmatrix}) \mid A =
\begin{pmatrix}
a_{11} & a_{12}
\\
a_{21} & a_{22}
\end{pmatrix}, D \in \mathrm{GL}_{2}(k),
\right.
\\
&\left.
\begin{matrix}
z,s \in \mathrm{GL}_{1}(k), r \in \mathrm{Mat}_{1,2}(k)
\\
\det A = \det D
\end{matrix} \right \}.
\end{align*}
\end{lem}

Next, we will study two group actions defined on
$\mathrm{Mat}_{1,n}(k)$. Firstly, we consider the following subgroup of
$S(\mathrm{GL}_{2} {\times }\mathrm{GL}_{2} {\times }\mathrm{GL}_{2})(k)$
given by
\begin{align}
\label{eqn:_S(GL2_x_GL2_x_GL1)_VTEX1}
\left \{(g_{1},g_{2}, \operatorname{diag}(\det (g_{1}g_{2})^{-1},1))
\mid g_{1},g_{2} \in \mathrm{GL}_{2}(k)\right \},
\end{align}
and observe that
\begin{align*}
\rho (g_{1},g_{2},\operatorname{diag}(t_{1},t_{2})) =
\begin{pmatrix}
t_{1}(g_{1}\otimes g_{2}) &
\\
& t_{2} \kappa (g_{1} \otimes g_{2}) \kappa ^{-1}
\end{pmatrix},
\end{align*}
where $\kappa = \operatorname{diag}(1, -1,-1,1)$ in
$\mathrm{GL}_{4}$. By the explicit formula \reftext{\eqref{eqn:_formula_Weil_-_Levi_Siegel_VTEX1}} of the Weil representation, we
will investigate the right group action of
$\mathrm{GL}_{2}(k) {\times }\mathrm{GL}_{2}(k) \cong \{(g_{1},g_{2},
\det (g_{1}g_{2})^{-1})\mid g_{i} \in \mathrm{GL}_{2}(k)\}$ on
$\mathrm{Mat}_{1,4}(k)$ given by
\begin{align}
\label{eqn:_group_action_GL2_x_GL2_x_GL1_VTEX1}
\mathrm{Mat}_{1,4}(k) {\times }\mathrm{GL}_{2}(k) {\times }
\mathrm{GL}_{2}(k) \to \mathrm{Mat}_{1,4}(k); && (X, g_{1},g_{2})
\mapsto X \cdot \frac{1}{\det (g_{1} g_{2})}(g_{1}\otimes g_{2}).
\end{align}
By elementary linear algebra, the orbits of this group action can be characterised
by $\mathrm{rank}(X)$ for $X \in \mathrm{Mat}_{2,2}(k)$. Similarly, consider
the following subgroup
$S'(\mathrm{GSp}_{4} {\times }\mathrm{GL}_{3})(k)$ of
$S'(\mathrm{GSp}_{4} {\times }\mathrm{GL}_{4})(k)$ given by
\begin{align}
\label{eqn:_S'(GSp4_x_GL3)_VTEX1}
&S'(\mathrm{GSp}_{4} {\times }\mathrm{GL}_{3})
\nonumber
\\
&\quad = \left \{ (g_{1},
\begin{pmatrix}
g_{2}' &
\\
& 1
\end{pmatrix}) \in \mathrm{GSp}_{4}(k) {\times }\mathrm{GL}_{4}(k)
\mid g_{2}' \in \mathrm{GL}_{3}(k), \lambda (g_{1}) \det (g_{2}') = 1
\right \},
\end{align}
and observe that
\begin{align*}
\mathrm{Ext}^{2}(g_{1},
\begin{pmatrix}
g_{2}' &
\\
& 1
\end{pmatrix}) =
\begin{pmatrix}
\Phi _{1}(g_{1}, g_{2}') &
\\
& \Phi _{2}(g_{1},g_{2}')
\end{pmatrix} \in \mathrm{Sp}_{24}
\end{align*}
for
$(g_{1},
\begin{pmatrix}
g_{2}' &
\\
& 1
\end{pmatrix}
) \in S'(\mathrm{GSp}_{4} {\times }\mathrm{GL}_{3})$ where
$\Phi _{1},\Phi _{2}: S'(\mathrm{GSp}_{4} {\times }\mathrm{GL}_{3})
\to \mathrm{GL}_{12}$ are group homomorphisms given by
\begin{align*}
\Phi _{1}(g_{1},g_{2}') =
\begin{pmatrix}
I_{10} &
\\
& -I_{2}
\end{pmatrix} ( M(g_{2}', 2) \otimes g_{1} )
\begin{pmatrix}
I_{10} &
\\
& -I_{2}
\end{pmatrix},
\\
\Phi _{2}(g_{1},g_{2}') =
\begin{pmatrix}
I_{6} &&
\\
&-I_{4} &
\\
&&I_{2}
\end{pmatrix} (g_{2}' \otimes g_{1})
\begin{pmatrix}
I_{6} &&
\\
&-I_{4} &
\\
&&I_{2}
\end{pmatrix},
\end{align*}
where $M(g_{2}', 2) = (\det g_{2}') S {}^{t}g_{2}^{\prime \,-1} S$ for
$S = w_{3}\operatorname{diag}(1,-1,1)$. Again, by the explicit formula \reftext{\eqref{eqn:_formula_Weil_-_Levi_Siegel_VTEX1}} of the Weil representation, we
will investigate the right group action of
$S'(\mathrm{GSp}_{4} {\times }\mathrm{GL}_{3})(k)$ on
$\mathrm{Mat}_{1,12}(k)$ given by
\begin{align}
\label{eqn:_group_action_GSp4_x_GL3_VTEX1}
\mathrm{Mat}_{1,12}(k) {\times }S'(\mathrm{GSp}_{4} {\times }
\mathrm{GL}_{3})(k) \to \mathrm{Mat}_{1,12}(k); && (X,g_{1},g_{2})
\mapsto X \cdot \Phi _{1}(g_{1},g_{2}).
\end{align}
Also, similar to the double coset decomposition of
$P_{12}(k){\setminus }\mathrm{GL}_{12}(k)/ H_{4,3}(k)$ the orbits of this
group action can be characterised by
$(\mathrm{rank}({}^{t} X j_{4} X), \mathrm{rank}(X))$ for
$X \in \mathrm{Mat}_{4,3}(k)$. We shall record these as a lemma as well.
\begin{lem}
\label{lem:_group_action_VTEX1}
The group action of $\mathrm{GL}_{2}(k) {\times }\mathrm{GL}_{2}(k)$ on
$\mathrm{Mat}_{1,4}(k)$ given in \reftext{\eqref{eqn:_group_action_GL2_x_GL2_x_GL1_VTEX1}} has three orbits given by the
representatives $\eta _{0} = (0,1,1,0)$, $\eta _{1} = (1,0,0,0)$ and
$\eta _{2} = (0,0,0,0)$ and their stabilisers are given by
\begin{align*}
Stab(\eta _{0}) =& \{(g,g^{\ast}, 1) \mid g \in \mathrm{GL}_{2}(k)\},
\\
Stab(\eta _{1}) =& \{(
\begin{pmatrix}
a &
\\
c & d
\end{pmatrix},
\begin{pmatrix}
e &
\\
g & d^{-1}
\end{pmatrix}, a^{-1}e^{-1}) \mid a,d,e \in \mathrm{GL}_{1}(k), c,g
\in k\},
\\
Stab(\eta _{2}) =& \mathrm{GL}_{2}(k) {\times }\mathrm{GL}_{2}(k).
\end{align*}
Similarly, the group action of
$S'(\mathrm{GSp}_{4} {\times }\mathrm{GL}_{3})(k)$ on
$\mathrm{Mat}_{1,12}(k)$ given in \reftext{\eqref{eqn:_group_action_GSp4_x_GL3_VTEX1}} has five orbits given by the representatives
$\xi _{0},\dots , \xi _{4}$ where
\begin{align*}
\xi _{0} = -e_{4} + e_{7} + e_{10}, && \xi _{1} = e_{1}, && \xi _{2} =
e_{1} + e_{6},&& \xi _{3} = e_{1} + e_{8},&& \xi _{4} =0,
\end{align*}
where $e_{i}$'s are the row-vectors in $\mathrm{Mat}_{1,12}(k)$ with one
in the $i$-entry and zero otherwise. Furthermore, their stabilisers are
given by
\begin{align*}
Stab(\xi _{0}) =&\left \{ (
\begin{pmatrix}
\lambda & u
\\
& A
\end{pmatrix}^{\ast}, A) \in S'(\mathrm{GSp}_{4} {\times }\mathrm{GL}_{3})(k) \right\},
\end{align*}
where $\lambda = \det(A), A = \left (
\begin{smallmatrix}
a_{11} & a_{12} & a_{13}
\\
a_{21} & a_{22} & a_{23}
\\
&&1
\end{smallmatrix}\right )$ and $\begin{pmatrix}
\lambda & u
\\
& A
\end{pmatrix} \in \mathrm{GSp}_{4}(k)$,
\begin{align*}
Stab(\xi _{1}) =&\left \{ (
\begin{sgmlfig}
\begin{pNiceArray}{c|cc|c}
a & & &
\\
\hline x_{1}& \Block{2-2}<\Large >{B} &
\\
x_{2}& & &
\\
\hline x_{3} & x_{2}' & x_{1}' & a^{-1}\det B
\end{pNiceArray}
\end{sgmlfig}
,
\begin{pmatrix}
C &
\\
v & \det ^{-1}(BC)
\end{pmatrix}) \in \mathrm{GSp}_{4}(k) {\times }\mathrm{GL}_{3}(k) \right\},
\end{align*}
where $a\det(C) =1$ and $v \in \mathrm{Mat}_{1,2}(k)$,
\begin{align*}
Stab(\xi _{2}) =&\left \{ (
\begin{pmatrix}
A &
\\
C & \lambda A^{\ast }
\end{pmatrix},
\begin{pmatrix}
\lambda \det ^{-1}(A) &
\\
v& \lambda ^{-1} (\det A) {}^{t} A^{-1}
\end{pmatrix}) \in \mathrm{GSp}_{4}(k) {\times }\mathrm{GL}_{3}(k) \right\},
\end{align*}
where $\lambda \in \mathrm{GL}_{1}(k), A \in \mathrm{GL}_{2}(k)$ and $v \in \mathrm{Mat}_{2,1}(k)$,
\begin{align*}
Stab(\xi _{3}) =&\left \{ (
\begin{pmatrix}
a_{11}&&a_{12}
\\
&B&
\\
a_{21} && a_{22}
\end{pmatrix},
\begin{pmatrix}
1 &
\\
v& {}^{t} A^{-1}
\end{pmatrix}) \in \mathrm{GSp}_{4}(k) {\times }\mathrm{GL}_{3}(k)\right\},
\end{align*}
where $A = (a_{ij}), B \in \mathrm{GL}_{2}(k), v \in \mathrm{Mat}_{2,1}(k)$ and $\det (A) = \det (B)$, and
\begin{align*}
Stab(\xi _{4}) =& S'(\mathrm{GSp}_{4} {\times }\mathrm{GL}_{3})(k).
\end{align*}
\end{lem}

\section{The global period integrals and their unfolding process}
\label{sec:_The_Period_integrals_and_their_unfolding_process_VTEX1}

In this section, we will pass from the period invariant
${\mathcal P}_{\mathcal{D}}(\varphi ;\Phi )$ given in \reftext{\eqref{eqn:_calP-Delta_VTEX1}} to the explicit families of global period integrals
${\mathcal P}_{\mathcal{D}}(\varphi ,s;\Phi )$ studied in this paper. The
section is organised as follows.

Firstly, we will explain two standard ways of introducing complex parameters
into the period invariant
${\mathcal P}_{\mathcal{D}}(\varphi ;\Phi )$ to obtain
${\mathcal P}_{\mathcal{D}}(\varphi ,s;\Phi )$. This includes replacing
a suitable theta factor by a mirabolic Eisenstein series and also inserting
$|\det |^{s}$-twists. Both constructions are modelled on the Rankin-Selberg
integrals \citep*{Jac,JS81,JPSS} for the tensor product $L$-functions on
the general linear groups.

Next, we will introduce some auxiliary integrals that arise after the root
exchange in Section~\ref{sec:_Root_Exchanges_and_Auxiliary_Integrals_VTEX1}.
These auxiliary integrals will be used in the construction and unfolding
of ${\mathcal P}_{\mathcal{D}}(\varphi ,s;\Phi )$.

Finally, we will construct the period integrals
${\mathcal P}_{\mathcal{D}}(\varphi ,s;\Phi )$ which represent the
$L$-function of the representations given in \reftext{Table~\ref{Table:_Multiplicity-free_repn_VTEX1}}. Then, we will perform the unfolding
process for these integrals to show each of them unfolds to the Whittaker
model. These period integrals
${\mathcal P}_{\mathcal{D}}(\varphi ,s;\Phi )$ will be addressed uniformly
in families.

Throughout this section and the remainder of the paper, if
$\mathcal{D}= (G,H,\rho _{H},\iota )$ is a hyperspherical quadruple, we
write
\begin{align*}
[H] = Z_{G,H}({\mathbb A}) H(k) {\setminus }H({\mathbb A}), && Z_{G,H}=
(Z_{G} \cap H)^{\circ }
\end{align*}
where $Z_{G}$ is the center of $G$. We will always regard $H$ as a subgroup
of $G$ via the maps and embeddings introduced in Section~\ref{subsec:_Embeddings_and_Maps_VTEX1}. For brevity, we will also use
${\mathcal P}_{\mathcal{D}}(\varphi ,s;\Phi )$ to uniformly denote all
the period integrals introduced in Sections~\ref{sec:_Period_Integrals_on_General_Linear_Groups_and_Spin_Similitude_Groups_VTEX1}-\ref{sec:_Glued_Period_Integrals_on_General_Linear_Groups_and_Spin_Similitude_Groups_VTEX1}.

\subsection{Period integral variant involving complex parameters}
\label{sec:_Period_integral_variant_involving_complex_parameters_VTEX1}

We will now explain how the period invariant
${\mathcal P}_{\mathcal{D}}(\varphi ;\Phi )$ defined in \reftext{\eqref{eqn:_calP-Delta_VTEX1}} gives rise to the parameter-dependent period integrals
${\mathcal P}_{\mathcal{D}}(\varphi ,s;\Phi )$ used below. For a more detailed
discussion of the relationship between the Rankin-Selberg integrals
${\mathcal P}_{\mathcal{D}}(\varphi ,s;\Phi )$ and the special values of
the corresponding period integral
${\mathcal P}_{\mathcal{D}}(\varphi ;\Phi )$, we refer to
\citep*[\S 1.4]{MWZ}.

For the constructions considered here, there are two mechanisms for introducing
the complex parameter. The first applies when the symplectic representation
contains a polarised summand of the form $\tau \oplus \tau ^{\vee}$; its
corresponding theta series can then be replaced by its Mellin transform,
namely a mirabolic Eisenstein series. The second applies when the image
of $H$ under projection to a factor of $G$ has trivial intersection with
the center of that factor; in that case one may insert a determinant twist.
These two methods are the familiar ones appearing in the classical Rankin--Selberg
integrals for the tensor product $L$-function of
$\mathrm{GL}_{n}\times \mathrm{GL}_{m}$; see, for instance,
\citep*{Jac} for $\mathrm{GL}_{2}\times \mathrm{GL}_{2}$,
\citep*{JS81} for $\mathrm{GL}_{n}\times \mathrm{GL}_{n}$ with
$n\geq 3$, and \citep*{JPSS} for
$\mathrm{GL}_{n}\times \mathrm{GL}_{m}$ with $m\neq n$.

\subsubsection{The mirabolic Eisenstein series construction}
\label{sec:_The_mirabolic_Eisenstein_series_construction_VTEX1}

Let $D=(G,H,\rho _{H},\iota )$ be a hyperspherical quadruple, and suppose
that $\rho _{H,\iota}:H\to \mathrm{Sp}(V)$ is anomaly-free. Assume that
the symplectic $H$-representation $\rho _{H}$ decomposes as
\begin{align*}
\rho _{H}=\rho '_{H}\oplus (\tau \oplus \tau ^{\vee}),
\end{align*}
where $\rho '_{H}$ is again symplectic and $(\tau ,S)$ is a finite-dimensional
$H$-module. Put $m=\dim S$, and let
\begin{align*}
j:H\longrightarrow \mathrm{GL}(S)
\end{align*}
be the homomorphism describing the action of $H$ on $S$. The summand
$S\oplus S^{\vee}$ carries its standard symplectic form, and the action
of $H$ on this summand factors through the Siegel Levi of
$\mathrm{Sp}(S\oplus S^{\vee})$:
\begin{align*}
h\longmapsto
\begin{pmatrix}
j(h) &
\\
& j(h)^{*}
\end{pmatrix}.
\end{align*}
Next, we define
\begin{align*}
\rho '_{H,\iota} =\rho '_{H}\oplus \left (\bigoplus _{k\in I_{
\mathrm{odd}}}\rho _{k}\right ).
\end{align*}
Then $\rho '_{H,\iota}:H\to \mathrm{Sp}(V')$ is still an anomaly-free symplectic
representation. We choose a maximal isotropic subspace of $V$ of the form
$Y'\oplus S$, where $Y'$ is a maximal isotropic subspace of $V'$. With
this choice, one has the following factorisation
\begin{align*}
{\mathcal S}\bigl((Y'\oplus S)({\mathbb A})\bigr) \simeq {\mathcal S}
\bigl(Y'({\mathbb A})\bigr)\otimes {\mathcal S}\bigl(S({\mathbb A})
\bigr).
\end{align*}
Let $\Omega _{\psi}$ and $\Omega '_{\psi}$ denote the Weil representations
attached to $V$ and $V'$, respectively, and let $\Omega ^{S}_{\psi}$ denote
the Weil representation attached to $S\oplus S^{\vee}$. For
$\Phi '\in {\mathcal S}(Y'({\mathbb A}))$ and
$\Phi _{m}\in {\mathcal S}(S({\mathbb A}))$, set
\begin{align}
\Theta ^{\Phi '}_{\psi}(h) =\sum _{X\in Y'(k)}\Omega '_{\psi}(h)\Phi '(X),
\qquad \Theta ^{\Phi _{m}}_{\psi}(h) =\sum _{\xi \in S(k)}\Omega ^{S}_{
\psi}(h)\Phi _{m}(\xi ).
\label{eq5.1}
\end{align}
Thus, for a pure tensor $\Phi =\Phi '\otimes \Phi _{m}$, the theta series
attached to $V$ factors as
\begin{align*}
\Theta ^{\Phi}_{\psi}(h) =\Theta ^{\Phi '}_{\psi}(h)\Theta ^{\Phi _{m}}_{
\psi}(h).
\end{align*}
The parameter $s$ is then introduced by replacing the second theta factor
with its Mellin transform. After choosing an identification
$S(k)\simeq k^{m}$, this Mellin transform is the mirabolic Eisenstein series
$E_{m}(j(h),s,\chi ;\Phi _{m})$ defined in \reftext{\eqref{eqn:_mirabolic_Eisenstein_series_VTEX1}}. Accordingly, we consider the
period ${\mathcal P}_{\mathcal{D}}(\varphi ,s;\Phi )$ given by
\begin{align}
\label{eqn:_period_integral_mirabolic_variant_VTEX1}
{\mathcal P}_{\mathcal{D}}(\varphi ,s;\Phi ) =& \int _{[H]} {
\mathcal P}_{\iota}(\varphi )(h) \Theta _{\psi}^{\Phi '}(h) E_{m}(j(h),s,
\chi ;\Phi _{m})\, dh.
\end{align}
In its domain of absolute convergence, $E_{m}(g,s,\chi ;\Phi _{m})$ can
be written as
\begin{align*}
E_{m}(g,s,\chi ;\Phi _{m}) =|\det g|^{s} \int _{k^{\times}\backslash {
\mathbb A}^{\times}} \sum _{\xi \in k^{m}-\{0\}} \Phi _{m}(a\xi g)|a|^{ms}
\chi (a)\,d^{\times }a, && g\in \mathrm{GL}_{m}({\mathbb A}).
\end{align*}
The following is an example of
${\mathcal P}_{\mathcal{D}}(\varphi ,s;\Phi )$ involving the mirabolic
Eisenstein series.
\begin{example}
\label{exmp5.1}
Consider the hyperspherical quadruple
\begin{align*}
\mathcal{D}= (G,H,\rho _{H},\iota ) = (\mathrm{GSp}_{4} {\times }
\mathrm{GL}_{4}, S'(\mathrm{GSp}_{4} {\times }\mathrm{GL}_{4}), (
\mathrm{std}_{\mathrm{GSp}_{4}} \otimes \wedge ^{2} \mathrm{std}_{
\mathrm{GL}_{4}}) \oplus T(\mathrm{std}_{\mathrm{GL}_{4}}), 1).
\end{align*}
Here $H=S'(\mathrm{GSp}_{4}\times \mathrm{GL}_{4})$ is the subgroup defined
in \reftext{\eqref{eqn:_S'(GSp4_x_GL4)_VTEX1}}, and
\begin{align*}
\rho _{H,\iota}=\rho _{H} = (\mathrm{std}_{\mathrm{GSp}_{4}}\otimes
\wedge ^{2}\mathrm{std}_{\mathrm{GL}_{4}}) \oplus (\mathrm{std}_{
\mathrm{GL}_{4}}\oplus \mathrm{std}_{\mathrm{GL}_{4}}^{\vee}).
\end{align*}
Both summands are symplectic anomaly-free $H$-representations. Let
$Y'\subset \mathrm{std}_{\mathrm{GSp}_{4}}\otimes \wedge ^{2}
\mathrm{std}_{\mathrm{GL}_{4}}$ be a Lagrangian, which we identify with
$k^{12}$. For a decomposable cusp form
$\varphi =\varphi _{\tau _{2}}\otimes \varphi _{\pi _{4}}$ of
$\mathrm{GSp}_{4}({\mathbb A})\times \mathrm{GL}_{4}({\mathbb A})$, the
corresponding period is
\begin{align*}
{\mathcal P}_{\mathcal{D}}(\varphi ,s;\Phi )= \int _{[S'(\mathrm{GSp}_{4}
{\times }\mathrm{GL}_{4})]} \varphi _{\tau _{2}}(g_{1}) \varphi _{
\pi _{4}}(g_{2}) E_{4}(g_{2},s,\chi ;\Phi _{4}) \Theta _{\psi}^{\Phi _{12}}(g_{1}
\otimes \wedge ^{2} g_{2})\, dg_{1}\, dg_{2}.
\end{align*}
Here $g_{1}\otimes \wedge ^{2} g_{2}$ denotes the image of
$(g_{1},g_{2})\in H({\mathbb A})$ under the map $\mathrm{Ext}^{2}$ defined
in \reftext{\eqref{eqn:_Ext2_map_VTEX1}}.
\begin{remark}
\label{rem5.2}
Using the well-known isomorphism of
$\mathrm{GSO}_{4} \cong (\mathrm{GL}_{2} {\times }\mathrm{GL}_{2})/
\Delta \mathbb G_{m}$ (see for instance
\citep*[\S 4.2]{YZ2024}), the Siegel Eisenstein series of
$\mathrm{GSO}_{4}$ defined in \reftext{\eqref{eqn:_Siegel_GSO4_VTEX1}} can be regarded
as a mirabolic Eisenstein series with respect to a copy of
$\mathrm{GL}_{2}$.
\end{remark}
\end{example}

\subsubsection{The determinant-shift construction}
\label{sec5.1.2}

A second way to introduce complex parameters is to twist the integrand
by a suitable determinant character. Let
$\mathcal{D}=(G,H,\rho _{H},\iota )$ be a hyperspherical quadruple, and
suppose that
\begin{align*}
G=G_{1}\times G_{2}\times \cdots \times G_{r}.
\end{align*}
We regard $H$ as a subgroup of $G$ through the embeddings fixed earlier.
If for some factor $G_{i}$ the projection
\begin{align*}
\operatorname{pr}_{i}:H\subset G\longrightarrow G_{i}
\end{align*}
satisfies
\begin{align*}
\operatorname{pr}_{i}(H)\cap Z_{G_{i}}=\{e\},
\end{align*}
where $Z_{G_{i}}$ is the center of $G_{i}$, we may insert determinant twists
to define a parameter-dependent integral
${\mathcal P}_{\mathcal{D}}(\varphi ,s;\Phi )$. In its simplest form, we
have
\begin{align*}
{\mathcal P}_{\mathcal{D}}(\varphi ,s_{i};\Phi ) = \int _{[H]} {
\mathcal P}_{\iota}(\varphi )(h) \Theta _{\psi}^{\Phi}(h) |\det
\operatorname{pr}_{i}(h)|^{s_{i}}\, dh.
\end{align*}
This is the determinant-shift analogue of the construction used in
\citep*{JPSS}.
\begin{example}
\label{example:_det-shift_VTEX1}
Consider
\begin{align*}
\mathcal{D}= (G,H,\rho _{H},\iota ) = (\mathrm{GSp}_{4} {\times }
\mathrm{GL}_{5}, S'(\mathrm{GSp}_{4} {\times }\mathrm{GL}_{4}),
\mathrm{std}_{\mathrm{GSp}_{4}} \otimes \wedge ^{2}\mathrm{std}_{
\mathrm{GL}_{4}},1).
\end{align*}
As in the previous example,
$H = S'(\mathrm{GSp}_{4} {\times }\mathrm{GL}_{4})$ but now
\begin{align*}
\rho _{H,\iota} = \rho _{H} = \mathrm{std}_{\mathrm{GSp}_{4}}
\otimes \wedge ^{2}\mathrm{std}_{\mathrm{GL}_{4}}.
\end{align*}
The two relevant projections are
\begin{align*}
&\operatorname{pr}_{\mathrm{GSp}_{4}}: H \subset \mathrm{GSp}_{4} {
\times }\mathrm{GL}_{5} \to \mathrm{GSp}_{4}, && \operatorname{pr}_{
\mathrm{GL}_{5}}: H \subset \mathrm{GSp}_{4} {\times }\mathrm{GL}_{5}
\to \mathrm{GL}_{5}.
\end{align*}
In this situation
\begin{align*}
\operatorname{pr}_{\mathrm{GSp}_{4}}(H)\cap Z_{\mathrm{GSp}_{4}}
\neq \{e\}, && \operatorname{pr}_{\mathrm{GL}_{5}}(H) \cap Z_{
\mathrm{GL}_{5}} = \{e\}.
\end{align*}
Therefore, the determinant twist is introduced through the general linear
factor. With the same notation as the previous example, the resulting period
is
\begin{align*}
{\mathcal P}_{\mathcal{D}}(\varphi ,s;\Phi ) = \int _{[S'(
\mathrm{GSp}_{4} {\times }\mathrm{GL}_{4})]} \varphi _{\tau _{2}}(g_{1})
\varphi _{\pi _{5}}(
\begin{pmatrix}
g_{2} &
\\
& 1
\end{pmatrix}) \Theta _{\psi}^{\Phi _{12}}(g_{1} \otimes \wedge ^{2} g_{2})
|\det g_{2}|^{s}\, dg_{1}\, dg_{2},
\end{align*}
\end{example}

\subsection{Auxiliary integrals}
\label{sec5.2}

\subsubsection{Auxiliary integrals involving cusp forms}
\label{sec:_Aux_integrals_involving_cusp_forms_VTEX1}

In this subsection, we will introduce some auxiliary integrals involving
cusp forms that will be used in the subsequent sections. Let
$\chi _{1},\chi _{2}: k^{{\times }}{\setminus }{\mathbb A}^{{\times }}
\to {\mathbb C}^{{\times }}$ be Hecke characters and
$s \in {\mathbb C}$. Recalling the same notations in Sections~\ref{sec:_Bessel_coefficient_for_General_Linear_Groups_VTEX1},
\ref{sec:_Fourier-Jacobi_coefficient_for_General_Linear_Groups_VTEX1} and \ref{sec:_Bessel_coefficient_GSpin_VTEX1}, then for $n\geq 2$ we write
\begin{align*}
&{\mathcal P}_{[2n-3,1^{4}]}(\varphi _{\tau _{n}})(j_{n}(g_{1},g_{2}),s,
\chi _{1},\chi _{2};\Phi _{2})
\\
&\quad = \chi _{1}(\det g_{1}) {\mathcal B}_{[2n-3,1^{4}]}(
\varphi _{\tau _{n}})(j_{n}(g_{1},g_{2})) E_{2}(g_{1},s,\chi _{2};
\Phi _{2}).
\end{align*}
Also for $k\in \{2,4\}, n \geq k$ and
$g \in \mathrm{GL}_{k}({\mathbb A})$, we write
\begin{align*}
&{\mathcal P}_{[n-k,1^{k}]}(\varphi _{\pi _{n}})(g,s,\chi _{1},\chi _{2};
\Phi _{k})
\\
&\quad =
\everymath{\displaystyle}
\begin{cases}
\chi _{1}(\det g) \varphi _{\pi _{n}}(g) E_{n}(g,s,\chi _{2};\Phi _{n})
& \text{if }k=n,%
\\
\noalign{\vspace{3pt}}
\chi _{1}(\det g) |\det g|^{s-\frac{1}{2}} {\mathcal{FJ}}_{[n-k,1^{k}]}(
\varphi _{\pi _{n}})(g;\Phi _{k})
&\text{if }k\neq n, n-k\text{ is even},%
\\
\noalign{\vspace{3pt}}
\chi _{1}(\det g) |\det g|^{s-\frac{1}{2}} {\mathcal B}_{[n-k,1^{k}]}(
\varphi _{\pi _{n}})(g) &\text{if }n-k\text{ is odd}.
\end{cases}
\end{align*}

\subsubsection{Auxiliary integrals involving Whittaker functions}
\label{sec:_Aux_integrals_involving_Whittaker_functions_VTEX1}

In this subsection, we will introduce some auxiliary integrals involving
Whittaker functions that will be used in the subsequent sections. Let
$W_{\pi _{n}}W_{\tau _{n}}$ be Whittaker functions associated to cuspidal
generic representation $\pi _{n},\tau _{n}$ of
$\mathrm{GL}_{n}({\mathbb A})$ and
$\mathrm{GSpin}_{2n+1}({\mathbb A})$ respectively. For $n\geq 2$ and
$(g_{1},g_{2}) \in G(\mathrm{SL}_{2}{\times }\mathrm{SL}_{2}) \cong
\mathrm{GSpin}_{4}$, we write
\begin{align*}
&\mathcal{U}_{[2n-3,1^{4}]}(W_{\tau _{n}})(j_{n}(g_{1},g_{2}),s;\Phi _{2})
\\
&\quad = |\det g_{1}|^{s} \Phi _{2}(e_{2} g_{1}) \int _{\mathrm{Mat}_{n-2,2}({
\mathbb A})} W_{\tau _{n}}(\left (
\begin{smallmatrix}
I_{2} &&&&
\\
B&I_{n-2}&&&
\\
&&1&&
\\
&&&I_{n-2}&
\\
&&&B'&I_{2}
\end{smallmatrix}\right ) j_{n}(g_{1},g_{2}))\,dB.
\end{align*}
Also for $n\geq 3$ and $g \in \mathrm{GSpin}_{6}({\mathbb A})$, we write
\begin{align*}
\mathcal{U}_{[2n-5,1^{6}]}(W_{\tau _{n}}) (j_{n}(g)) = \int _{
\mathrm{Mat}_{n-3,3}({\mathbb A})} W_{\tau _{n}}(\left (
\begin{smallmatrix}
I_{3} &&&&
\\
B&I_{n-3}&&&
\\
&&1&&
\\
&&&I_{n-3}&
\\
&&&B'&I_{3}
\end{smallmatrix}\right ) j_{n}(g))\,dB.
\end{align*}
Whereas for $k\in \{2,4\}, n \geq k$ and
$g \in \mathrm{GL}_{k}({\mathbb A})$ we write
$\mathcal U_{[n-k,1^{k}]}(W_{\pi _{n}})(g,s;\Phi _{k})$ as
\begin{itemize}
\item If $k=n$,
$\mathcal U_{[n-k,1^{k}]}(W_{\pi _{n}})(g,s;\Phi _{k}) = |\det g|^{s}
\Phi _{k}(e_{k} g) W_{\pi _{k}}(g)$.
\item If $k\neq n$ and $n-k$ is even, we write $\mathcal U_{[n-k,1^{k}]}(W_{\pi _{n}})(g,s;\Phi _{k})$ as
\begin{align*}
&\gamma _{
\psi}^{-1}(\det g) |\det g|^{s-\frac{1}{2}} \int _{\mathrm{Mat}_{
\frac{n-k-2}{2},k}({\mathbb A})} \int _{\mathrm{Mat}_{1,k}({\mathbb A})}
\\
&\quad \cdot W_{\pi _{n}}(\left (
\begin{smallmatrix}
I_{k}&&
\\
B'&I_{\frac{n-k-2}{2}}&
\\
&&I_{\frac{n-k+2}{2}}
\end{smallmatrix}\right )\left (
\begin{smallmatrix}
I_{k}&&&
\\
&I_{\frac{n-k-2}{2}}&&
\\
X&&1&
\\
&&&I_{\frac{n-k}{2}}
\end{smallmatrix}\right ) j_{n,k}(g)) {\Omega }_{\psi}(\iota _{k}(g))
\Phi _{k}(X)\,dX\,dB'.
\end{align*}
\item If $n-k$ is odd, we write $\mathcal U_{[n-k,1^{k}]}(W_{\pi _{n}})(g,s;\Phi _{k})$ as
\begin{align*}
&|\det g|^{s-
\frac{1}{2}} \int _{\mathrm{Mat}_{\frac{n-k-1}{2},k}({\mathbb A})} W_{
\pi _{n}}(\left (
\begin{smallmatrix}
I_{k}&&
\\
B&I_{\frac{n-k-1}{2}}&
\\
&&I_{\frac{n-k+1}{2}}
\end{smallmatrix}\right )j_{n,k}(g))\,dB.
\end{align*}
\end{itemize}
With these, we can proceed with the construction of the period integrals
${\mathcal P}_{\mathcal{D}}(\varphi ,s;\Phi )$.

\subsection{Period integrals on general linear groups $\times$ spin similitude groups}
\label{sec:_Period_Integrals_on_General_Linear_Groups_and_Spin_Similitude_Groups_VTEX1}

In this subsection, we will construct the period integrals for
$\mathrm{GSpin}_{2m+1} {\times }\mathrm{GL}_{2}$ for $m\geq 3$,
$\mathrm{GSp}_{4} {\times }\mathrm{GL}_{n}$ for $n\geq 4$ and
$\mathrm{GSpin}_{2m+1} {\times }\mathrm{GL}_{3}$ for $m\geq 2$. In these
cases, the corresponding hyperspherical quadruples are
\begin{align*}
\mathcal{D}_{m,2} =&
\everymath{\displaystyle}\begin{cases}
(\mathrm{GSpin}_{2m+1}{\times }\mathrm{GL}_{2}, \mathrm{GSpin}_{4}, T(
\mathrm{std}_{\mathrm{GL}_{2}}), [2m-3,1^{4}]),
& \text{if }m\geq 3,
\end{cases}
\\
\mathcal{D}_{2,n} =&
\everymath{\displaystyle}\begin{cases}
(\mathrm{GSp}_{4} {\times }\mathrm{GL}_{4}, S'(\mathrm{GSp}_{4} {
\times }\mathrm{GL}_{4}), \mathrm{std}_{\mathrm{GSp}_{4}} \otimes
\wedge ^{2} \mathrm{std}_{\mathrm{GL}_{4}} \oplus T(\mathrm{std}_{
\mathrm{GL}_{4}}),1),& \text{if }n=4,%
\\
\noalign{\vspace{3pt}}
(\mathrm{GSp}_{4} {\times }\mathrm{GL}_{n}, S'(\mathrm{GSp}_{4} {
\times }\mathrm{GL}_{4}), \mathrm{std}_{\mathrm{GSp}_{4}} \otimes
\wedge ^{2}\mathrm{std}_{\mathrm{GL}_{4}}, (1,[n-4,1^{4}])),
& \text{if }n\geq 5,
\end{cases}
\\
\mathcal{D}_{m,3} =&
\everymath{\displaystyle}\begin{cases}
(\mathrm{GSp}_{4} {\times }\mathrm{GL}_{3}, \mathrm{GSp}_{4} {\times }
\mathrm{GL}_{3}, T(\mathrm{std}_{\mathrm{GSp}_{4}} \otimes
\mathrm{std}_{\mathrm{GL}_{3}}),1),& \text{if }m=2,%
\\
\noalign{\vspace{3pt}}
(\mathrm{GSpin}_{2m+1} {\times }\mathrm{GL}_{3}, \mathrm{GSpin}_{6} {
\times }\mathrm{GL}_{3}, T(\mathrm{HSpin}_{6} \otimes \mathrm{std}_{
\mathrm{GL}_{3}}),
([2m-5,1^{6}],1)),& \text{if }m\geq 3.
\end{cases}
\end{align*}
With these, we define the
$\mathrm{GSpin}_{2a+1}{\times }\mathrm{GL}_{b}$-period integral
${\mathcal P}_{\mathcal{D}_{a,b}}^{\mathrm{GSpin}{\times }\mathrm{GL}}(
\varphi ,s;\Phi )$ as
\begin{align*}
&{\mathcal P}_{\mathcal{D}_{m,2}}^{\mathrm{GSpin}{\times }\mathrm{GL}}(
\varphi ,s;\Phi ) = \int _{[G(\mathrm{SL}_{2} {\times }\mathrm{SL}_{2})]}
{\mathcal P}_{[2m-3,1^{4}]}(\varphi _{\tau _{m}})(j_{m}(g_{1},g_{2}),s,
\mathbf{1}, \omega _{\tau _{m}}\omega _{\pi _{2}};\Phi _{2}) \varphi _{
\pi _{2}}(g_{2})\,dg_{1}\,dg_{2},
\\
&{\mathcal P}_{\mathcal{D}_{2,n}}^{\mathrm{GSpin}{\times }\mathrm{GL}}(
\varphi ,s;\Phi )
= \int _{[S'(\mathrm{GSp}_{4} {\times }\mathrm{GL}_{4})]}
\varphi _{\tau _{2}}(g_{1}^{\ast}) {\mathcal P}_{[n-4,1^{4}]}(
\varphi _{\pi _{n}})(g_{2},s,\mathbf{1},\omega _{\tau _{2}}^{2}
\omega _{\pi _{n}};\Phi _{4}) \Theta _{\psi}^{\Phi _{12}}(g_{1}
\otimes \wedge ^{2} g_{2})\,dg_{1}\,dg_{2},
\\
&{\mathcal P}_{\mathcal{D}_{2,3}}^{\mathrm{GSpin}{\times }\mathrm{GL}}(
\varphi ,s;\Phi )= \int _{[\mathrm{GSp}_{4} {\times }\mathrm{GL}_{3}]}
\varphi _{\tau _{2}}(g_{1}) \varphi _{\pi _{3}}(g_{2}^{\ast}) E_{12}
\left ( g_{1} \otimes g_{2}, \frac{2s+1}{4}, \omega _{\tau _{2}};
\Phi _{12}\right )\,dg_{1}\,dg_{2},
\\
&{\mathcal P}_{\mathcal{D}_{m,3}}^{\mathrm{GSpin}{\times }\mathrm{GL}}(
\varphi ,s;\Phi )
=\int _{[\mathrm{GSpin}_{6} {\times }\mathrm{GL}_{3}]}
{\mathcal B}_{[2m-5,1^{6}]}(\varphi _{\tau _{m}})(j_{m}(g_{1}))
\varphi _{\pi _{3}}(g_{2}^{\ast}) E_{12}\left ( \operatorname{pr}(g_{1})
\otimes g_{2}, \frac{2s+1}{4}, \omega _{\tau _{m}};\Phi _{12}\right )dg_{i},
\end{align*}
for cusp forms
$\varphi _{\tau _{a}} \otimes \varphi _{\pi _{b}} \in \tau _{a}
\otimes \pi _{b}$, where in the case $(a,b) = (2,n)$,
$\Theta _{\psi}^{\Phi _{12}}$ is the $\widetilde{\mathrm{Sp}}_{24}$-theta
series pulled back to
$S'(\mathrm{GSp}_{4} {\times }\mathrm{GL}_{4})({\mathbb A})$ via the
$\mathrm{Ext}^{2}$ map \reftext{\eqref{eqn:_Ext2_map_VTEX1}}. Also for
$(a,b) =(m,3)$ where $m\geq 2$, we are assuming
$\omega _{\tau _{m}} \omega _{\pi _{3}} = \mathbf1$.

\subsection{Multi-variable period integrals on general linear groups and spin similitude groups}
\label{sec:_Multi-Variable_Period_Integrals_VTEX1}

In this subsection, we will construct the multi-variable period integrals
${\mathcal P}_{\mathcal{D},n}^{G}(\varphi ,s,w,\allowbreak \chi ,\allowbreak \mu ;\Phi )$. Let
$\chi ,\mu : k^{{\times }}{\setminus }{\mathbb A}^{{\times }}\to {
\mathbb C}^{{\times }}$ be Hecke characters. Here, the corresponding hyperspherical
quadruples are
\begin{align*}
\mathcal{D}^{\mathrm{GL}}_{n} =&
\everymath{\displaystyle}\begin{cases}
(\mathrm{GL}_{2}, \mathrm{GL}_{2}, T(\mathrm{std}_{\mathrm{GL}_{2}})
\oplus T(\mathrm{std}_{\mathrm{GL}_{2}}),1)&\text{if }n=2,%
\\
\noalign{\vspace{3pt}}
(\mathrm{GL}_{n}, \mathrm{GL}_{2}, T(\mathrm{std}_{\mathrm{GL}_{2}}), [n-2,1^{2}]),
&\text{if }n\geq 3,
\end{cases}
\\
\mathcal{D}^{\mathrm{GSpin}}_{n} =&
\everymath{\displaystyle}\begin{cases}
(\mathrm{GSpin}_{2n+1}, \mathrm{GSpin}_{4}, T(\mathrm{std}_{
\mathrm{GL}_{2}}) \oplus T(\mathrm{std}_{\mathrm{GL}_{2}}), [2n-3,1^{4}]),
&\text{if }n\geq 2.
\end{cases}
\end{align*}
With these, we define the multivariable period integral
${\mathcal P}_{\mathcal{D},n}^{G}={\mathcal P}_{\mathcal{D},n}^{G}(
\varphi ,s,w,\chi ,\mu ;\Phi )$ as
\begin{align*}
{\mathcal P}_{\mathcal{D},n}^{\mathrm{GL}}=& \int _{[\mathrm{GL}_{2}]}
{\mathcal P}_{[n-2,1^{2}]}\left (\varphi _{\pi _{n}},g,\frac{s+w}{2},
\chi , \omega _{\pi _{n}}\mu \chi ;\Phi _{2}'\right ) E_{2}\left (g,
\frac{w-s+1}{2}, \chi \mu ^{-1};\Phi _{2}\right )\,dg,
\\
{\mathcal P}_{\mathcal{D},n}^{\mathrm{GSpin}}=& \int _{[G(\mathrm{SL}_{2}
{\times }\mathrm{SL}_{2})]} {\mathcal P}_{[2n-3,1^{4}]}(\varphi _{
\tau _{n}})\left (j_{n}(g_{1},g_{2}), \frac{s+w}{2}, \chi , \omega _{
\tau _{n}}\mu \chi ;\Phi _{2}'\right ) E_{2}\left (g_{2},
\frac{w-s+1}{2}, \chi \mu ^{-1};\Phi _{2}\right )\,dg_{i}\xch{,}{.}
\end{align*}
for cusp forms $\varphi _{\pi _{n}} \in \pi _{n}$ and
$\varphi _{\tau _{n}} \in \tau _{n}$.

\subsection{Period integrals on special orthogonal similitude groups}
\label{sec:_Period_Integrals_on_Special_Orthogonal_Similitude_Groups_VTEX1}

In this subsection, we will construct the period integrals on
$\mathrm{GSO}_{8}$ and $\mathrm{GSO}_{10}$. Here, the hyperspherical quadruples
are
\begin{align*}
\mathcal{D}_{D_{4}} =& (\mathrm{GSO}_{8}, S(\mathrm{GL}_{2} {\times }
\mathrm{GSO}_{4}), T(\mathrm{std}_{\mathrm{GL}_{2}}) \oplus T(
\mathrm{std}_{\mathrm{GL}_{2}}), [2^{2},1^{4}]),
\\
\mathcal{D}_{D_{5}}
= &(\mathrm{GSO}_{10}, \mathrm{GL}_{2}, 0, [4^{2},1^{2}]).
\end{align*}
With these, we define the period integrals
${\mathcal P}_{\mathcal{D},D_{4}}={\mathcal P}_{\mathcal{D},D_{4}}(\varphi ,s,w;\Phi )$ and
${\mathcal P}_{\mathcal{D},D_{5}}={\mathcal P}_{\mathcal{D},D_{5}}(\varphi ,s;\Phi )$ as
\begin{align*}
&{\mathcal P}_{\mathcal{D},D_{4}}= \int _{[S(
\mathrm{GL}_{2} {\times }\mathrm{GSO}_{4})]} \mathcal{FJ}_{[2^{2},1^{4}]}(
\varphi _{\sigma _{4}})(J_{D_{4}}(g_{1},g_{2});\Phi _{4}) E_{2}(g_{1},s,
\mathbf{1};\Phi _{2}) E_{P}^{\ast}(g_{2},f_{w}) \,dg_{1}\,dg_{2},
\\
&{\mathcal P}_{{\mathcal D},D_{5}} = \int _{[
\mathrm{GL}_{2}]} |\det g|^{s-1/2} \mathcal{FJ}_{[4^{2},1^{2}]}(
\varphi _{\sigma _{5}})(g;\Phi _{2})\,dg,
\end{align*}
for cusp forms $\varphi _{\sigma _{4}} \in \sigma _{4}$,
$\varphi _{\sigma _{5}} \in \sigma _{5}$ and $f_{w} \in I(w)$, where we
are assuming $\omega _{\sigma _{4}} = \mathbf{1}$.

\subsection{Glued period integrals on general linear groups and spin similitude groups}
\label{sec:_Glued_Period_Integrals_on_General_Linear_Groups_and_Spin_Similitude_Groups_VTEX1}

In this sub-section, we will construct the glued multi-variable period
integrals for $\underline{\mathrm{GL}}_{2} {\times }G$ and
$G_{1} {\times }\underline{\mathrm{GL}}_{2} {\times }G_{2}$ where
$G, G_{1}$ and $G_{2}$ are either $\mathrm{GL}_{n}$ or
$\mathrm{GSpin}_{2n+1}$ for $n\geq 2$. Here, the hyperspherical quadruples
are
\begin{align*}
&\mathcal{D}^{\mathbf 1,\mathrm{GL}}_{n} =
\everymath{\displaystyle}\begin{cases}
(\underline{\mathrm{GL}}_{2} {\times }\mathrm{GL}_{2}, \mathrm{GL}_{2}
{\times }\mathrm{GL}_{2}, T(\mathrm{std}_{\mathrm{GL}_{2}}) \oplus T(
\mathrm{std}_{\mathrm{GL}_{2}} \otimes \mathrm{std}_{\mathrm{GL}_{2}}),
1), & \text{if }n=2,%
\\
\noalign{\vspace{3pt}}
(\underline{\mathrm{GL}}_{2} {\times }\mathrm{GL}_{n}, \mathrm{GL}_{2}
{\times }\mathrm{GL}_{2}, T(\mathrm{std}_{\mathrm{GL}_{2}}\otimes
\mathrm{std}_{\mathrm{GL}_{2}}), (1, [n-2,1^{2}])),
& \text{if }n\geq 3,
\end{cases}
\\
&\begin{aligned}
\mathcal{D}^{\mathbf 1,\mathrm{GSpin}}_{n}
&= (
\underline{\mathrm{GL}}_{2} {\times }\mathrm{GSpin}_{2n+1},
\mathrm{GL}_{2} {\times }G(\mathrm{SL}_{2} {\times }\mathrm{SL}_{2}), T(
\mathrm{std}_{\mathrm{GL}_{2}} \otimes \mathrm{std}_{\mathrm{GL}_{2}})
\oplus T(\mathrm{std}_{\mathrm{GL}_{2}}), (1,[2n-3,1^{4}])),
\end{aligned}
\end{align*}
and
\begin{align*}
&\mathcal{D}_{m,n}^{\mathrm{GL},\mathrm{GL}}
=
\everymath{\displaystyle}\begin{cases}
(\mathrm{GL}_{2} {\times }\underline{\mathrm{GL}}_{2} {\times }
\mathrm{GL}_{2}, S(\mathrm{GL}_{2}^{3}), T(\mathrm{std}_{\mathrm{GL}_{2}})
\oplus T(\mathrm{std}_{\mathrm{GL}_{2}}) \oplus \mathrm{std}_{
\mathrm{GL}_{2}}^{\otimes 3}, 1),
&\text{if }m=n=2,%
\\
\noalign{\vspace{3pt}}
(\mathrm{GL}_{2} {\times }\underline{\mathrm{GL}}_{2} {\times }
\mathrm{GL}_{n}, S(\mathrm{GL}_{2}^{3}), T(\mathrm{std}_{\mathrm{GL}_{2}})
\oplus \mathrm{std}_{\mathrm{GL}_{2}}^{\otimes 3}, (1,1,[n-2,1^{2}])
&\text{if }m=2,n\geq 3,%
\\
\noalign{\vspace{3pt}}
(\mathrm{GL}_{m} {\times }\underline{\mathrm{GL}}_{2} {\times }
\mathrm{GL}_{n}, S(\mathrm{GL}_{2}^{3}), \mathrm{std}_{\mathrm{GL}_{2}}^{
\otimes 3}, ([m-2,1^{2}],1,[n-2,1^{2}]))) &
\text{if }m,n\geq 3,
\end{cases}
\end{align*}
\begin{align*}
\mathcal{D}_{m,n}^{\mathrm{GL},\mathrm{GSpin}}=&%
\everymath{\displaystyle}\begin{cases}
(\mathrm{GL}_{2} {\times }\underline{\mathrm{GL}}_{2} {\times }
\mathrm{GSpin}_{2n+1}, S''(\mathrm{GL}_{2}^{4}),
T(\mathrm{std}_{
\mathrm{GL}_{2}}) \oplus T(\mathrm{std}_{\mathrm{GL}_{2}}) \oplus
\mathrm{std}_{\mathrm{GL}_{2}}^{\otimes 3}, (1,1,[2n-3,1^{4}]))
&\text{if } m=2,%
\\
\noalign{\vspace{3pt}}
(\mathrm{GL}_{m} {\times }\underline{\mathrm{GL}}_{2} {\times }
\mathrm{GSpin}_{2n+1}, S''(\mathrm{GL}_{2}^{4}),
T(\mathrm{std}_{
\mathrm{GL}_{2}}) \oplus \mathrm{std}_{\mathrm{GL}_{2}}^{\otimes 3}, ([m-2,1^{2}],1,[2n-3,1^{4}]))
&\text{if }m\geq 3,
\end{cases}
\end{align*}
and $\mathcal{D}_{m,n}^{\mathrm{GSpin},\mathrm{GSpin}}$ is
\begin{align*}
&(\mathrm{GSpin}_{2m+1} {\times }\underline{\mathrm{GL}}_{2} {\times }
\mathrm{GSpin}_{2n+1}, S^{\ast}(\mathrm{GL}_{2}^{5}), T(\mathrm{std}_{
\mathrm{GL}_{2}}) \oplus T(\mathrm{std}_{\mathrm{GL}_{2}}) \oplus
\mathrm{std}_{\mathrm{GL}_{2}}^{\otimes 3},
[2m-3,1^{4}],1,[2n-3,1^{4}])).
\end{align*}
With these, we define the glued
$\underline{\mathrm{GL}}_{2} {\times }G$-period integral
${\mathcal P}_{n}^{\mathbf1,G}(\varphi ,s,w;\Phi )$ and the
$G_{1} {\times }\underline{\mathrm{GL}}_{2} {\times }G_{2}$-period integrals
${\mathcal P}_{m,n}^{G_{1},G_{2}}(\varphi ,s,w;\Phi )$ as
\begin{align*}
{\mathcal P}_{\mathcal{D},n}^{\mathbf1,\mathrm{GL}}(\varphi ,s,w;
\Phi ) = \int _{[\mathrm{GL}_{2} {\times }\mathrm{GL}_{2}]} & {
\mathcal P}_{[n-2,1^{2}]}(\varphi _{\pi _{n}'}) (g_{1},s,\omega _{
\pi _{2}},\omega _{\pi _{2}}\omega _{\pi _{n}'};\Phi _{2}) \varphi _{
\pi _{2}}(g_{2}^{\ast})
\\
\cdot & \omega _{\pi _{2}}(\det g_{2}) E_{4}\left ( g_{1} \otimes g_{2},
\frac{2w+1}{4}, \omega _{\pi _{2}};\Phi _{4}\right ) \, dg_{i},
\end{align*}
\begin{align*}
{\mathcal P}_{\mathcal{D},n}^{\mathbf1,\mathrm{GSpin}}(\varphi ,s,w;
\Phi )
= \int _{[\mathrm{GL}_{2} {\times }G(\mathrm{SL}_{2} {\times }
\mathrm{SL}_{2})]} &{\mathcal P}_{[2n-3,1^{4}]}(\varphi _{\tau _{n}}) (j_{n}(g_{2},g_{3}),
s, \mathbf 1, \omega _{\pi _{2}}\omega _{\tau _{n}};\Phi _{2})
\varphi _{\pi _{2}}(g_{1}^{\ast})
\\
&\qquad \cdot \omega _{\pi _{2}}(\det g_{1} g_{3}) E_{4}\left ( g_{1}
\otimes g_{3}, \frac{2w+1}{4}, \omega _{\pi _{2}};\Phi _{4} \right )
\, dg_{i},
\end{align*}
and
\begin{align*}
&{\mathcal P}_{\mathcal{D},m,n}^{\mathrm{GL},\mathrm{GL}}(\varphi ,s,w;
\Phi )
=\int _{[S(\mathrm{GL}_{2}^{{\times }3})]} {\mathcal P}_{[m-2,1^{2}]}(
\varphi _{\pi _{m}})(g_{1},w,\mathbf{1}, \omega _{\pi _{m}}\omega _{
\pi _{2}'};\Phi _{2})
\\
&\qquad \cdot  \varphi _{\pi _{2}'}(g_{2}^{\ast}) {\mathcal P}_{[n-2,1^{2}]}(
\varphi _{\pi _{n}''})(g_{3},s,\mathbf{1},\omega _{\pi _{2}'}\omega _{
\pi _{n}''};\Phi _{2}') \Theta _{\psi}^{\Phi _{4}}(\rho (g_{1},g_{2},g_{3}))
\,dg_{i},
\\
&{\mathcal P}_{\mathcal{D},m,n}^{\mathrm{GL},\mathrm{GSpin}}(\varphi ,s,w;
\Phi )
= \int _{[S''(\mathrm{GL}_{2}^{{\times }4})]} {\mathcal P}_{[m-2,1^{2}]}(
\varphi _{\pi _{m}})(g_{1},w,\mathbf{1},\omega _{\pi _{m}}\omega _{
\pi _{2}'};\Phi _{2}) \varphi _{\pi _{2}'}(g_{2}^{\ast})
\\
&\qquad \cdot {\mathcal P}_{[2n-3,1^{4}]}(\varphi _{\tau _{n}})(j_{n}(g_{4},g_{3}),s,
\mathbf{1},\omega _{\pi _{2}'}\omega _{\tau _{n}};\Phi _{2}') \Theta _{
\psi}^{\Phi _{4}}(\rho (g_{2},g_{3},g_{1}))\,dg_{i},
\\
&{\mathcal P}_{\mathcal{D},m,n}^{\mathrm{GSpin},\mathrm{GSpin}}(
\varphi ,s,w;\Phi )
= \int _{[S^{\ast}(\mathrm{GL}_{2}^{{\times }5})]}
{\mathcal P}_{[2m-3,1^{4}]}(\varphi _{\tau _{m}})(j_{m}(g_{1},g_{2}),w,
\mathbf{1},\omega _{\tau _{m}}\omega _{\pi _{2}'};\Phi _{2}) \varphi _{
\pi _{2}'}(g_{3}^{\ast})
\\
&\qquad \cdot {\mathcal P}_{[2n-3,1^{4}]}(\varphi _{\tau _{n}'})(j_{n}(g_{4},g_{5}),s,
\mathbf{1},\omega _{\pi _{2}'}\omega _{\tau _{n}'};\Phi _{2}')
\Theta _{\psi}^{\Phi _{4}}(\rho (g_{2},g_{3},g_{5}))\,dg_{i},
\end{align*}
for cusp forms $\varphi _{\pi _{n}} \in \pi _{n}$ (resp.
$\varphi _{\pi _{n}'}\in \pi _{n}', \varphi _{\pi _{n}''} \in \pi _{n}''$)
and $\varphi _{\tau _{n}} \in \tau _{n}$ (resp.
$\varphi _{\tau _{n}'} \in \tau _{n}'$). Here
$\Theta _{\psi}^{\Phi _{4}}$ is the $\widetilde{\mathrm{Sp}}_{8}$-theta
series pulled back to $S(\mathrm{GL}_{2}^{{\times }3})({\mathbb A})$ via
the map $\rho $ \reftext{\eqref{eqn:_rho_map_VTEX1}}.

\subsection{Unfolding process}
\label{sec5.7}

In this sub-section, we will perform the unfolding process of the period
integrals ${\mathcal P}_{\mathcal{D}}(\varphi ,s;\Phi )$ defined in Sections~\ref{sec:_Period_Integrals_on_General_Linear_Groups_and_Spin_Similitude_Groups_VTEX1}-\ref{sec:_Glued_Period_Integrals_on_General_Linear_Groups_and_Spin_Similitude_Groups_VTEX1}
and show that they unfold to the Whittaker model. More precisely, we have
the following theorem.
\begin{thm}
\label{thm:_Unfolding_of_period_integral_and_absolute_conv_VTEX1}
The period integrals ${\mathcal P}_{\mathcal{D}}(\varphi ,s;\Phi )$ defined
in Sections~\ref{sec:_Period_Integrals_on_General_Linear_Groups_and_Spin_Similitude_Groups_VTEX1}-\ref{sec:_Glued_Period_Integrals_on_General_Linear_Groups_and_Spin_Similitude_Groups_VTEX1}
converge absolutely for all $s,w \in {\mathbb C}$, away from the poles
of the Eisenstein series. Moreover, for sufficiently large
$\operatorname{Re}(s), \operatorname{Re}(w) \gg 0$ (resp.
$\operatorname{Re}(w)\gg \operatorname{Re}(s) \gg 0$) the period integrals
defined in \ref{sec:_Period_Integrals_on_General_Linear_Groups_and_Spin_Similitude_Groups_VTEX1},
\ref{sec:_Period_Integrals_on_Special_Orthogonal_Similitude_Groups_VTEX1} and \ref{sec:_Glued_Period_Integrals_on_General_Linear_Groups_and_Spin_Similitude_Groups_VTEX1}
(resp. the period integrals defined in Section~\ref{sec:_Multi-Variable_Period_Integrals_VTEX1}) unfold to the Whittaker model.

The $\mathrm{GSpin}_{2a+1} {\times }\mathrm{GL}_{b}$-period integrals
${\mathcal P}_{a,b}^{\mathrm{GSpin}{\times }\mathrm{GL}}={\mathcal P}_{a,b}^{
\mathrm{GSpin}{\times }\mathrm{GL}}(\varphi ,s;\Phi )$ in Section~\ref{sec:_Period_Integrals_on_General_Linear_Groups_and_Spin_Similitude_Groups_VTEX1}
unfold to
\begin{align*}
{\mathcal P}_{{\mathcal D},m,2}^{\mathrm{GSpin}{\times }\mathrm{GL}}=&
\int _{U'_{G(\mathrm{SL}_{2}{\times }\mathrm{SL}_{2})}({\mathbb A}){
\setminus }G(\mathrm{SL}_{2}{\times }\mathrm{SL}_{2})({\mathbb A})}
\mathcal{U}_{[2m-3,1^{4}]}(W_{\tau _{m}}^{\bar \psi})(j_{m}(g_{1},g_{2}),s;
\Phi _{2}) W_{\pi _{2}}^{\psi}(g_{2})dg_{i},
\\
{\mathcal P}_{{\mathcal D},2,n}^{\mathrm{GSpin}{\times }\mathrm{GL}}=&
\int _{U_{\mathrm{GSp}_{4}{\times }\mathrm{GL}_{4}}'({\mathbb A}) {
\setminus }S'(\mathrm{GSp}_{4} {\times }\mathrm{GL}_{4})({\mathbb A})}
W_{\tau _{2}}^{\psi '}(g_{1}^{\ast}) \mathcal{U}_{[n-4,1^{4}]}(W_{
\pi _{n}}^{\bar \psi})(g_{2},s;\Phi _{4})
\omega _{\psi}(g_{1}
\otimes \wedge ^{2} g_{2}) \Phi _{12}(\xi _{0})dg_{i},
\\
{\mathcal P}_{{\mathcal D},2,3}^{\mathrm{GSpin}{\times }\mathrm{GL}}=&
\int _{Z'({\mathbb A})U_{\mathrm{GSp}_{4}{\times }\mathrm{GL}_{3}}'({
\mathbb A}){\setminus }(\mathrm{GSp}_{4} {\times }\mathrm{GL}_{3})({
\mathbb A})} W_{\tau _{2}}^{\psi '}(g_{1}) W_{\pi _{3}}^{\bar \psi}(g_{2}^{
\ast}) \Phi _{12}(e_{12}\gamma (g_{1}\otimes g_{2}))
|\det (g_{1}
\otimes g_{2})|^{(2s+1)/4}\,dg_{i},
\\
{\mathcal P}_{{\mathcal D},m,3}^{\mathrm{GSpin}{\times }\mathrm{GL}}=&
\int \mathcal{U}_{[2m-5,1^{6}]}(W_{\tau _{m}}^{\psi})(j_{m}(g_{1})) W_{
\pi _{3}}^{\bar\psi}(g_{2}^{\ast}) \Phi _{12}(e_{12} \gamma (
\operatorname{pr}(g_{1}) \otimes g_{2})) |\det (\operatorname{pr}(g_{1})
\otimes g_{2})|^{(2s+1)/4} dg_{i},
\end{align*}
where the outer integral of
${\mathcal P}_{{\mathcal D},m,3}^{\mathrm{GSpin}{\times }\mathrm{GL}},
m \geq 3$ is defined over
$Z'({\mathbb A}) U'_{\mathrm{GL}_{4} {\times }\mathrm{GL}_{3}}({
\mathbb A}) {\setminus }\break(\mathrm{GSpin}_{6} {\times }\mathrm{GL}_{3})({
\mathbb A})$. Here, $\xi _{0}$ is the representative in
$\mathrm{Mat}_{1,12}$ given in \reftext{Lemma~\ref{lem:_group_action_VTEX1}} and
$\gamma = \gamma _{3} = \omega _{3}$ is the matrix representative in \reftext{Lemma~\ref{lem:_Double_coset_decomp_VTEX1}}. Also, $U_{G}'$ is the unipotent subgroup
of $G$ given by
\begin{enumerate}[(i)]
\item[(i)] $U_{G(\mathrm{SL}_{2} {\times }\mathrm{SL}_{2})}'$ is the maximal
unipotent subgroup of $G(\mathrm{SL}_{2} {\times }\mathrm{SL}_{2})$ consisting
of upper triangular matrices.

\item[(ii)] $U'_{\mathrm{GSp}_{4} {\times }\mathrm{GL}_{4}}$ is the unipotent
subgroup of $\mathrm{GSp}_{4} {\times }\mathrm{GL}_{4}$ with the parametrisation
\begin{align*}
U'_{\mathrm{GSp}_{4} {\times }\mathrm{GL}_{4}} = \left \{(\left (
\begin{smallmatrix}
1&-x_{1} & x_{2}' & x_{4}
\\
&1&x_{3} & x_{2}
\\
&&1&x_{1}
\\
&&&1
\end{smallmatrix}\right )^{\ast}, \left (
\begin{smallmatrix}
1 & x_{3} & x_{2}& z_{1}
\\
&1&x_{1}&z_{2}
\\
&&1&z_{3}
\\
&&&1
\end{smallmatrix}\right )) \in \mathrm{GSp}_{4} {\times }\mathrm{GL}_{4}
\right \}.
\end{align*}
\item[(iii)] $U'_{\mathrm{GSp}_{4} {\times }\mathrm{GL}_{3}}$ is the unipotent
subgroup of $\mathrm{GSp}_{4} {\times }\mathrm{GL}_{3}$ with the parametrisation
\begin{align*}
U'_{\mathrm{GSp}_{4} {\times }\mathrm{GL}_{3}}=\left \{ ( \left (
\begin{smallmatrix}
1&-x_{1}&x_{2}'&x_{4}
\\
&1&x_{3}&x_{2}
\\
&&1&x_{1}
\\
&&&1
\end{smallmatrix}\right ), \left (
\begin{smallmatrix}
1&x_{3}&x_{2}
\\
&1&x_{1}
\\
&&1
\end{smallmatrix}\right )^{\ast }) \in \mathrm{GSp}_{4}{\times }
\mathrm{GL}_{3}\right \}.
\end{align*}
\item[(iv)] $U'_{\mathrm{GL}_{4} {\times }\mathrm{GL}_{3}}$ is the unipotent
subgroup of $\mathrm{GL}_{4} {\times }\mathrm{GL}_{3}$ with the parametrisation
\begin{align*}
U'_{\mathrm{GL}_{4} {\times }\mathrm{GL}_{3}}=\left \{(\left (
\begin{smallmatrix}
1&y_{1} & y_{2} & y_{4}
\\
&1&x_{3} & x_{2}
\\
&&1&x_{1}
\\
&&&1
\end{smallmatrix}\right ), \left (
\begin{smallmatrix}
1 & x_{3} & x_{2}
\\
&1&x_{1}
\\
&&1
\end{smallmatrix}\right )^{\ast}) \in \mathrm{GL}_{4} {\times }
\mathrm{GL}_{3} \right \}.
\end{align*}
\end{enumerate}
Furthermore, $Z'$ is the following subgroup of the center of
$\mathrm{GSp}_{4} {\times }\mathrm{GL}_{3}$ and
$\mathrm{GSpin}_{6} {\times }\mathrm{GL}_{3}$ with the parametrisation
\begin{align*}
Z' =
\everymath{\displaystyle}\begin{cases}
\{(tI_{4}, t^{-1}I_{3})\}
&\subset \mathrm{GSp}_{4} {\times } \mathrm{GL}_{3},%
\\
\noalign{\vspace{3pt}}
\{(e^{\ast}_{0}(t), t^{-1}I_{3})\} &\subset \mathrm{GSpin}_{6} {
\times }\mathrm{GL}_{3}.
\end{cases}
\end{align*}
Next, the multivariable period integrals
${\mathcal P}_{\mathcal{D},n}^{G} = {\mathcal P}_{\mathcal{D},n}^{G}(
\varphi ,s,w,\chi ,\mu ;\Phi )$ defined in Section~\ref{sec:_Multi-Variable_Period_Integrals_VTEX1} unfold to
\begin{align*}
{\mathcal P}_{\mathcal{D},n}^{\mathrm{GL}}&= \int _{\mathrm{GL}_{2}({
\mathbb A})} \chi (\det g) \mathcal{U}_{[n-2,1^{2}]}(W_{\pi _{n}}^{
\bar\psi}) \left (g, \frac{s+w}{2};\Phi _{2}'\right )
\cdot F_{2}\left (w_{2}g,
\frac{w-s+1}{2},\chi \mu ^{-1};\Phi _{2}\right )\,dg,
\\
{\mathcal P}_{\mathcal{D},n}^{\mathrm{GSpin}} &=\int \chi (\det g_{1})
\mathcal U_{[2n-3,1^{4}]}(W_{\tau _{n}}^{\bar\psi})\left (j_{n}(g_{1},g_{2}),
\frac{s+w}{2};\Phi _{2}' \right )
\cdot F_{2}\left (w_{2} g_{2},
\frac{w-s+1}{2}, \chi \mu ^{-1};\Phi _{2}\right )dg_{i},
\end{align*}
where in ${\mathcal P}_{\mathcal{D},n}^{\mathrm{GSpin}}$ the integral is
defined over
$U_{1}({\mathbb A}) {\setminus }G(\mathrm{SL}_{2} {\times }
\mathrm{SL}_{2})({\mathbb A})$. Here, $U_{1}$ is the unipotent subgroup
of $G(\mathrm{SL}_{2} {\times }\mathrm{SL}_{2})$ with the parametrisation
$U_{1} = \{(n(x), I_{2})\}$, and $w_{2}$ is the long Weyl element in
$\mathrm{GL}_{2}$. Next, for the period integrals
${\mathcal P}_{\mathcal{D},D_{4}}={\mathcal P}_{\mathcal{D},D_{4}}(
\varphi ,s,w;\Phi )$ and
${\mathcal P}_{\mathcal{D},D_{5}}={\mathcal P}_{\mathcal{D},D_{5}}(
\varphi ,s;\Phi )$ defined in Section~\ref{sec:_Period_Integrals_on_Special_Orthogonal_Similitude_Groups_VTEX1}, they
unfold to
\begin{align*}
{\mathcal P}_{\mathcal{D},D_{4}} =& \int _{N_{S(\mathrm{GL}_{2} {
\times }\mathrm{GSO}_{4})}({\mathbb A}) {\setminus }S(\mathrm{GL}_{2} {
\times }\mathrm{GSO}_{4})({\mathbb A})} \int _{{\mathbb A}} \int _{
\mathrm{Mat}_{2,2}({\mathbb A})} |\det g_{1}|^{s}
\\
\cdot &
\Phi _{2}(e_{2} g_{1})
W_{\sigma _{4}}^{\bar\psi}(w[42] x_{\alpha _{2} + \alpha _{4}}(r) u(X)
J_{D_{4}}(g_{1},g_{2}))
f_{w}^{\ast}(g_{2}) {\Omega }_{\psi}(\iota _{D_{4}}(g_{1},g_{2}))
\Phi _{4}(x) \,dr\,dx\,dg_{1}\,dg_{2},
\end{align*}

\begin{align*}
{\mathcal P}_{\mathcal{D},D_{5}}
=&\int _{N_{\mathrm{GL}_{2}}({
\mathbb A}) {\setminus }\mathrm{GL}_{2}({\mathbb A})} \int _{Z({
\mathbb A})} \int _{X({\mathbb A})} |\det g|^{s-1/2} \gamma _{\psi}^{-1}(
\det g) W_{\sigma _{5}}^{\bar\psi}(zx J_{D_{5}}(g))
{\Omega }_{\psi}(
\iota _{2}((\det g) g^{\ast})) \Phi _{2}(x) \,dx\,dz\,dg,
\end{align*}
where $N_{S(\mathrm{GL}_{2} {\times }\mathrm{GSO}_{4})}$ is the maximal
unipotent subgroup of $S(\mathrm{GL}_{2} {\times }\mathrm{GSO}_{4})$ consisting
of upper triangular unipotent matrices. Also,
$f_{w}^{\ast}(g_{2}) = \zeta _{k}(2w) f_{w}(g_{2})$,
$x_{\alpha _{2} + \alpha _{4}}(r) = I_{8} + r(E_{25} - E_{47}) \in
\mathrm{GSO}_{8}$ and $w[42]$ is the Weyl element in
$\mathrm{GSO}_{8}$ given by
\begin{align*}
w[42] = \left (
\begin{smallmatrix}
I_{2} &&&
\\
&&I_{2} &
\\
&I_{2}&&
\\
&&&I_{2}
\end{smallmatrix}\right ) \left (
\begin{smallmatrix}
1 &&&&
\\
&w_{2}&&&
\\
&&I_{2}&&
\\
&&&w_{2}&
\\
&&&&1
\end{smallmatrix} \right ).
\end{align*}
Likewise, $N_{\mathrm{GL}_{2}}$ is the unipotent subgroup of
$\mathrm{GL}_{2}$ consisting of upper triangular unipotent matrices, and
$Z$ and $X$ are subgroups of $\mathrm{GSO}_{10}$ defined in Section~\ref{sec:_FJ_coefficient_of_GSO_VTEX1}. Lastly for the glued period integrals
${\mathcal P}_{\mathcal{D},n}^{\mathbf1,G} = {\mathcal P}_{
\mathcal{D},n}^{\mathbf1,G}(\varphi ,s,w;\Phi )$ and
${\mathcal P}_{\mathcal{D},m,n}^{G_{1},G_{2}} = {\mathcal P}_{
\mathcal{D},m,n}^{G_{1},G_{2}}(\varphi ,s,w;\Phi )$ defined in Section~\ref{sec:_Glued_Period_Integrals_on_General_Linear_Groups_and_Spin_Similitude_Groups_VTEX1},
they unfold to
\begin{align*}
{\mathcal P}_{\mathcal{D},n}^{\mathbf1,\mathrm{GL}} =& \int _{U_{1,2}^{
\Delta}({\mathbb A}) {\setminus }\mathrm{GL}_{2}^{{\times }2}({
\mathbb A})} \mathcal U_{[n-2,1^{2}]}(W_{\pi _{n}'})(g_{1},s;\Phi _{2})
W_{\pi _{2}}^{\bar \psi}(g_{2}^{\ast})
\\
\cdot & \omega _{\pi _{2}}(\det g_{1} g_{2}) \Phi _{4}(e_{4} \gamma _{2,2}(g_{1}
\otimes g_{2})) |\det (g_{1} \otimes g_{2})|^{(2w+1)/4}\, dg_{i},
\end{align*}
\begin{align*}
{\mathcal P}_{\mathcal{D},n}^{\mathbf1,\mathrm{GSpin}} =& \int _{U_{1,3}^{
\Delta}({\mathbb A}) U_{2}({\mathbb A}) {\setminus }(\mathrm{GL}_{2}({
\mathbb A}) {\times }G(\mathrm{SL}_{2}({\mathbb A}) {\times }
\mathrm{SL}_{2}({\mathbb A})))} \mathcal U_{[2n-3,1^{4}]}(W_{\tau _{n}}^{
\psi})(j_{n}(g_{2},g_{3}),s;\Phi _{2})
\\
\cdot &W_{\pi _{2}}^{\bar\psi}(g_{1}^{
\ast})
 \omega _{\pi _{2}}(\det g_{1} g_{3}) \Phi _{4}(e_{4} \gamma _{2,2}
(g_{1} \otimes g_{3})) |\det (g_{1} \otimes g_{3})|^{(2w+1)/4}\, dg_{i},
\end{align*}
and
\begin{align*}
{\mathcal P}_{\mathcal{D},m,n}^{\mathrm{GL},\mathrm{GL}}=& \int _{U_{1,2}^{
\Delta}({\mathbb A}) U_{3}({\mathbb A}) {\setminus }S(\mathrm{GL}_{2}^{{
\times }3})({\mathbb A})} \mathcal{U}_{[m-2,1^{2}]}(W_{\pi _{m}}^{
\psi})(g_{1},w;\Phi _{2}) W_{\pi _{2}'}^{\bar \psi}(g_{2}^{\ast})
\\
\cdot & \mathcal{U}_{[n-2,1^{2}]}(W_{\pi _{n}''}^{\bar \psi})(g_{3},s;
\Phi _{2}') \omega _{\psi}(\rho (g_{1},g_{2},g_{3}))\Phi _{4}(\eta _{0})
\,dg_{i},
\\
{\mathcal P}_{\mathcal{D},m,n}^{\mathrm{GL},\mathrm{GSpin}}=& \int _{U_{1}({
\mathbb A}) U_{2,3}^{\Delta}({\mathbb A}) U_{4}({\mathbb A}) {
\setminus }S''(\mathrm{GL}_{2}^{{\times }4})({\mathbb A})}
\mathcal{U}_{[m-2,1^{2}]}(W_{\pi _{m}}^{\bar \psi})(g_{1},w;\Phi _{2})
W_{\pi _{2}'}^{\bar \psi}(g_{2}^{\ast})
\\
\cdot & \mathcal{U}_{[2n-3,1^{4}]}(W_{\tau _{n}}^{\psi})(j_{n}(g_{4},g_{3}),
s;\Phi _{2}') \omega _{\psi}(\rho (g_{2},g_{3},g_{1}))\Phi _{4}(\eta _{0})
\,dg_{i},
\\
{\mathcal P}_{\mathcal{D},m,n}^{\mathrm{GSpin},\mathrm{GSpin}}=& \int _{U_{1}({
\mathbb A}) U_{2,3}^{\Delta}({\mathbb A}) U_{4}({\mathbb A}) U_{5}({
\mathbb A}) {\setminus }S^{\ast}(\mathrm{GL}_{2}^{{\times }5})({
\mathbb A})} \mathcal{U}_{[2m-3,1^{4}]}(W_{\tau _{m}}^{\psi})(j_{m}(g_{1},g_{2}),w;
\Phi _{2})
\\
\cdot &W_{\pi _{2}'}^{\bar \psi}(g_{3}^{\ast})
\mathcal{U}_{[2n-3,1^{4}]}(W_{\tau _{n}'}^{\bar \psi})(j_{n}(g_{4},g_{5}),s;
\Phi _{2}') \omega _{\psi}(\rho (g_{2},g_{3},g_{5}))\Phi _{4}(\eta _{0})
\,dg_{i}.
\end{align*}
Here, $\gamma _{2,2}$ is the element in $\mathrm{GL}_{4}$ given by
\begin{align*}
\gamma _{2,2}= \left (
\begin{smallmatrix}
1 & &
\\
&&I_{2}
\\
&1&
\end{smallmatrix}\right ) \left (
\begin{smallmatrix}
1&&&
\\
&1&1&
\\
&&1&
\\
&&&1
\end{smallmatrix}\right ),
\end{align*}
and $\eta _{0}$ is the representative in $\mathrm{Mat}_{1,4}$ given in
\reftext{Lemma~\ref{lem:_group_action_VTEX1}}. Also, $U_{i}$ is the subgroup of
$\mathrm{GL}_{2}$ consisting of upper triangular unipotent matrices, realised
in the $i$-th component of $\mathrm{GL}_{2}^{{\times }n}$, and
$U_{i,j}^{\Delta}$ is the subgroup of
$\mathrm{GL}_{2} {\times }\mathrm{GL}_{2}$ given by
$\{(n(x)^{\ast}, n(x))\}$, realised as subgroups of
$\mathrm{GL}_{2}^{{\times }n}$ with the parametrisation
\begin{align*}
U_{i} =& \{(I_{2},\dots , I_{2}, n(x), I_{2},\dots , I_{2})\}
\subset \mathrm{GL}_{2}^{{\times }n},
\\
U_{i,j}^{\Delta }=& \{(I_{2},\dots , I_{2}, n(x)^{\ast}, I_{2},\dots ,
I_{2}, \dots , I_{2}, n(x), I_{2},\dots , I_{2})\} \subset
\mathrm{GL}_{2}^{{\times }n}.
\end{align*}
\end{thm}
\begin{proof}
For the first part, the absolute convergence follows from the fact that
the Eisenstein series are of moderate growth and the degenerate Whittaker
coefficient of the cusp forms $\varphi _{\pi _{n}}$,
$\varphi _{\tau _{n}}$ and $\varphi _{\sigma _{n}}$ are rapidly decreasing
\citep*[Lemmas 2.1 and 6.1]{BA-So} on their respective integration domain.
Now, we will assume for sufficiently large
$\operatorname{Re}(s), \operatorname{Re}(w) \gg 0$ (resp.
$\operatorname{Re}(w) \gg \operatorname{Re}(s) \gg 0$) for the period integrals
defined in \ref{sec:_Period_Integrals_on_General_Linear_Groups_and_Spin_Similitude_Groups_VTEX1},
\ref{sec:_Period_Integrals_on_Special_Orthogonal_Similitude_Groups_VTEX1} and \ref{sec:_Glued_Period_Integrals_on_General_Linear_Groups_and_Spin_Similitude_Groups_VTEX1}
(resp. for the period integrals defined in Section~\ref{sec:_Multi-Variable_Period_Integrals_VTEX1}).

The identity for the period integrals
${\mathcal P}_{\mathcal{D},m,2}^{\mathrm{GSpin}{\times }\mathrm{GL}}$ and
${\mathcal P}_{\mathcal{D},n}^{G}$ for
$G\in \{\mathrm{GL},\mathrm{GSpin}\}$ follows from a direct application
of the root exchange \reftext{Lemma~\ref{lem:_root_exchange_Bessel_VTEX1}} and \reftext{Lemma~\ref{lem:_root_exchanges_Fourier-Jacobi_VTEX1}} for the degenerate Whittaker coefficients,
and the Bruhat decomposition of $\mathrm{GL}_{2}$. Next for the period
integrals
${\mathcal P}_{\mathcal{D},2,n}^{\mathrm{GSpin}{\times }\mathrm{GL}}$,
${\mathcal P}_{\mathcal{D},m,3}^{\mathrm{GSpin}{\times }\mathrm{GL}}$ and
the glued period integrals
${\mathcal P}^{\mathbf1, G}_{\mathcal{D},n}$,
${\mathcal P}_{\mathcal{D},m,n}^{G_{1},G_{2}}$, their unfolding processes
follow as a direct consequence of the lemmas discussed in Section~\ref{sec:_Group_actions_and_Double_coset_decomposition_VTEX1}. Hence, we will
detail relevant aspects of the unfolding process for the
${\mathcal P}_{\mathcal{D},2,n}^{\mathrm{GSpin}{\times }\mathrm{GL}}$-period
integral, and sketch the process for the remaining ones.

In the
${\mathcal P}_{\mathcal{D},2,n}^{\mathrm{GSpin}{\times }\mathrm{GL}}$-period
integral, we proceed with the use of \reftext{Lemma~\ref{lem:_root_exchange_Bessel_VTEX1}} and \reftext{Lemma~\ref{lem:_root_exchanges_Fourier-Jacobi_VTEX1}} to perform the root exchange for
the degenerate Whittaker coefficient
${\mathcal P}_{[n-4,1^{4}]}(\varphi _{\pi _{n}})$. Then, by Fourier expansion
along the unipotent subgroup
\begin{align*}
V_{4,n}= \left \{v(r) =v(r_{1},r_{2},r_{3},r_{4}) = I_{n} + \sum _{i=1}^{4}
r_{i} E_{i,5} \right \},
\end{align*}
and let $\mathrm{GL}_{4}$ in the second component of
$S'(\mathrm{GSp}_{4} {\times }\mathrm{GL}_{4})$ act on it with two orbits.
By cuspidality of $\varphi _{\pi _{n}}$, the trivial orbit vanishes. Moreover,
we choose the non-trivial character representative $\psi _{V_{4,n}}$ given
by
\begin{align*}
\psi _{V_{4,n}}(v(r)) = \psi (r_{4}).
\end{align*}
From this, we see that its stabiliser is
$S'(\mathrm{GSp}_{4} {\times }\mathrm{GL}_{3}) U_{3,1}$, where
$U_{3,1}$ is the unipotent radical of the parabolic $P_{4}$ of
$\mathrm{GL}_{4}$ given in \reftext{\eqref{eqn:_GLn_parabolic_subgroup_VTEX1}}. Next, we
unfold the theta series $\Theta _{\psi}^{\Phi}$. From \reftext{Lemma~\ref{lem:_group_action_VTEX1}}, we see the orbits of group action of
$S'(\mathrm{GSp}_{4} {\times }\mathrm{GL}_{3})$ on
$\mathrm{Mat}_{1,12}$ are negligible in the sense of Gelbart-Piateski-Shapiro
\cite{GePS} except for the generic orbit with representative
$\xi _{0}$. Moreover, we have
\begin{align*}
(I_{4} \otimes \wedge ^{2} \left (
\begin{smallmatrix}
1 & & &x_{1}
\\
&1&&x_{2}
\\
&&1&x_{3}
\\
&&&1
\end{smallmatrix}\right )) =
\begin{pmatrix}
I_{12} & X(x_{1},x_{2},x_{3})
\\
&I_{12}
\end{pmatrix} \in \mathrm{Sp}_{24},
\end{align*}
where $X(x_{1},x_{2},x_{3})$ is the matrix in $\mathrm{Mat}_{12,12}$ given
by
\begin{align*}
X(x_{1},x_{2},x_{3})&= x_{1}\left ( - \sum _{p=1}^{2} E_{p,p+4} -
\sum _{q=7}^{8} E_{q,q+4} + \sum _{r=3}^{6} E_{r,r+4} \right )
+ x_{2}
\left ( \sum _{r=1}^{4} E_{r,r} + \sum _{p=9}^{12} E_{p,p}\right ) + x_{3}
\sum _{r=5}^{12} E_{r,r-4}.
\end{align*}
Using the explicit Weil formula \reftext{\eqref{eqn:_formula_Weil_-_unipotent_Siegel_VTEX1}} we see that
\begin{align*}
\omega _{\psi}(I_{4} \otimes \wedge ^{2} \left (
\begin{smallmatrix}
1 & & &x_{1}
\\
&1&&x_{2}
\\
&&1&x_{3}
\\
&&&1
\end{smallmatrix}\right ))\Phi _{12}(\xi _{0}) = \psi (x_{3}) \Phi _{12}(
\xi _{0}).
\end{align*}
Thus, proceeding with a standard Fourier expansion argument along the unipotent
subgroups $\{I_{4} + z_{1} E_{1,3} + z_{2} E_{2,4}\}$ and
$\{I_{4} + z_{3} E_{1,2}\}$ of $\mathrm{GL}_{4}$ we obtain the desired
identity. Likewise, for the glued period integrals
${\mathcal P}_{\mathcal{D},n}^{\mathbf1,G}$ and
${\mathcal P}_{\mathcal{D},m,n}^{G_{1},G_{2}}$ we perform the same unfolding
argument as above. In particular, for the case of
${\mathcal P}_{\mathcal{D},m,n}^{G_{1},G_{2}}$, we unfold the Theta series
$\Theta _{\psi}^{\Phi _{4}}$ by considering the group action of
$S(\mathrm{GL}_{2} {\times }\mathrm{GL}_{2} {\times }\mathrm{GL}_{1})$
on $\mathrm{Mat}_{1,4}$ given in \reftext{\eqref{eqn:_group_action_GL2_x_GL2_x_GL1_VTEX1}}. By \reftext{Lemma~\ref{lem:_group_action_VTEX1}}, it is clear that all the orbits are negligible
in the sense of Gelbart-Piateski-Shapiro \cite{GePS}, except for
the generic orbit represented by $\eta _{0} = (0,1,1,0)$. In fact, we see
that
\begin{align*}
\rho (I_{2}, I_{2},n(x_{3})) =
\begin{pmatrix}
I_{4} & X(x_{3})
\\
& I_{4}
\end{pmatrix} \in \mathrm{Sp}_{8}, && X(x_{3}) = \operatorname{diag}(-x_{3},
x_{3}, x_{3}, -x_{3}),
\end{align*}
and by the formula \reftext{\eqref{eqn:_formula_Weil_-_unipotent_Siegel_VTEX1}} of the
Weil representation we have
\begin{align*}
\omega _{\psi}(\rho (I_{2}, I_{2},n(x_{3})))\Phi _{4}(\eta _{0}) =
\psi (x_{3}) \Phi _{4}(\eta _{0}).
\end{align*}
Analogously, for the
${\mathcal P}_{\mathcal{D},m,3}^{\mathrm{GSpin}{\times }\mathrm{GL}}$-period
integral we use \reftext{Lemma~\ref{lem:_root_exchange_Bessel_VTEX1}} to perform the root
exchange for the Bessel coefficient
${\mathcal B}_{[2m-5,1^{6}]}(\varphi _{\tau _{m}})$ and unfold the mirabolic
Eisenstein series $E_{12}(\cdot )$. Using the double coset decompositions
in \reftext{Lemma~\ref{lem:_Double_coset_decomp_VTEX1}}, we see that all the orbits are
negligible in the sense of Gelbart-Piateski-Shapiro
\cite{GePS} except for the generic orbit represented by
$\gamma _{3} = \omega _{3}$. Proceeding with a standard Fourier expansion
argument along the unipotent subgroups of $\mathrm{GL}_{3}$ we obtain the
desired identity. Finally, for the periods
${\mathcal P}_{\mathcal{D},D_{4}}$ and
${\mathcal P}_{\mathcal{D},D_{5}}$, their unfolding process follows that
of \citet*{GH} and \citep*{Gin95a} respectively.

For the period integral ${\mathcal P}_{\mathcal{D},D_{5}}$ we use \reftext{Lemma~\ref{lem:_root_exchanges_Fourier-Jacobi_VTEX1}} to perform the root exchange for
the Fourier-Jacobi coefficient
${\mathcal{FJ}}_{[4^{2},1^{2}]}(\varphi _{\sigma _{5}})$ to obtain
\begin{align*}
\int \gamma _{\psi}^{-1}(\det g)& |\det g|^{s-1/2} \int _{[VWZY]}
\int _{X({\mathbb A})} \varphi _{\sigma _{5}}(vwzyx J_{D_{5}}(g)) {
\Omega }_{\psi}(\iota _{2}( (\det g) g^{\ast}))
\\
\cdot &
\Phi _{2}(x) \psi _{[4^{2},1^{2}]}(vwz)
\,dx\cdots \,dg,
\end{align*}
where the outer integral is defined over
$\mathrm{GL}_{2}(k){\setminus }\mathrm{GL}_{2}({\mathbb A})$ and
$V,W,Z,Y,X$ are unipotent subgroups defined in Section~\ref{sec:_FJ_coefficient_of_GSO_VTEX1}. Proceeding with the exact unfolding process
detailed in \citep*[Proposition 2.5]{Gin95a}, we obtain the desired identity.
Lastly for the period integral ${\mathcal P}_{\mathcal{D},D_{4}}$, we write
$H = S(\mathrm{GL}_{2} {\times }\mathrm{GSO}_{4})$ and $G = \mathrm{GSO}_8$. We proceed with unfolding both Eisenstein series and using \reftext{Lemma~\ref{lem:_root_exchanges_Fourier-Jacobi_VTEX1}} to perform the root exchange for
the Fourier-Jacobi coefficient
${\mathcal{FJ}}_{[2^{2},1^{4}]}(\varphi _{\sigma _{4}})$ to obtain
\begin{align*}
&\int _{Z_{G,H}({\mathbb A}) M_{Q}(k) U_{Q}(k) {\setminus }H({\mathbb A})}
\int _{[\mathrm{Mat}_{2,2}]} \int _{k{\setminus }{\mathbb A}} \int _{
\mathrm{Mat}_{2,2}({\mathbb A})} F_{2}(g_{1},s,\mathbf 1; \Phi _{2}) f_{w}^{
\ast}(g_{2})
\\
&\cdot  \varphi _{\sigma _{4}}(u(0,Y;t) u(X,0;0) J_{D_{4}}(g_{1},g_{2}))
{\Omega }_{\psi}(\iota _{D_{4}}(g_{1},g_{2})) \Phi _{4}(x) \psi (t)
\,dX\,dt \,dY\,dg_{1}\,dg_{2},
\end{align*}
where $M_{Q}$ and $U_{Q}$ are subgroups of $H$ given by
\begin{align*}
M_{Q} &= \left \{(\operatorname{diag}(t_{1},t_{2}),
\operatorname{diag}(A,\lambda A^{\ast})) \mid
\begin{matrix}
t_{1},t_{2} \in \mathrm{GL}_{1}, A \in \mathrm{GL}_{2},
\\
\lambda t_{1} t_{2} = 1
\end{matrix} \right \},
\\
& U_{Q} = \left \{ (n(x), \left (
\begin{smallmatrix}
I_{2} & Z
\\
& I_{2}
\end{smallmatrix} \right )) \mid Z = \operatorname{diag}(z,-z)
\right \}.
\end{align*}
Then, we proceed with Fourier expansion along the abelian unipotent subgroup
$V_{2}$ generated by the root subgroups
$x_{\alpha _{2}+\alpha _{3}}(r_{1})$ and
$x_{\alpha _{1} + \alpha _{2} + \alpha _{3}}(r_{2})$ and let $M_{Q}$ act
it with two orbits. By cuspidality, it is clear that the trivial orbit
vanishes. Indeed, there exists a constant term along $U_{P_{3}}$ the unipotent
radical of the standard maximal parabolic subgroup
$P_{3}\subset \mathrm{GSO}_{8}$ defined by deleting the simple root
$\alpha _{3}$. Consequently, the above integral expression becomes
\begin{align*}
&\int _{Z_{G,H}({\mathbb A}) R_{S}(k) U_{S}(k) U_{Q}({\mathbb A}) {
\setminus }H({\mathbb A})} \int _{[V_{2}]} \int _{[U_{Q}]} \int _{[
\mathrm{Mat}_{2,2}]} \int _{k{\setminus }{\mathbb A}} \int _{
\mathrm{Mat}_{2,2}({\mathbb A})} F_{2}(g_{1},s,\mathbf 1; \Phi _{2}) f_{w}^{
\ast}(g_{2})
\\
\cdot & \varphi _{\sigma _{4}}(v(r_{1},r_{2}) u' u(0,Y;t) u(X,0;0)J_{D_{4}}(g_{1},g_{2}))
{\Omega }_{\psi}(\iota _{D_{4}}(g_{1},g_{2})) \Phi _{4}(x)
\psi (t+r_{1})
\,dX\,dt \,dY\,du'\,dr_{i}\,dg_{j},
\end{align*}
where
$v(r_{1},r_{2}) = x_{\alpha _{2}+\alpha _{3}}(r_{1}) x_{\alpha _{1} +
\alpha _{2} + \alpha _{3}}(r_{2})$ and $R_{S}$ and $U_{S}$ are subgroups
of $M_{Q}$ given by
\begin{align*}
R_{S} = \left \{(\operatorname{diag}(t_{1},t_{2}),
\operatorname{diag}(a,t_{1}^{-1}, t_{2}^{-1}, a^{-1}t_{1}^{-1}t_{2}^{-1}))
\right \}, && U_{S} = \left \{ (I_{2},
\begin{pmatrix}
n(z) &
\\
& n(-z)
\end{pmatrix}) \right \}.
\end{align*}
Next, consider Fourier expansion along the unipotent subgroup generated
by $x_{-\alpha _{2}}(r_{3})$ and using the same arguments as in
\citep*[Equation 2.13]{GH} the above integral expression becomes
\begin{align*}
\int \int _{[V_{2}]} &\int _{[U_{Q}]} \int _{(k {\setminus }{
\mathbb A})^{3}} \int _{{\mathbb A}} \int _{\mathrm{Mat}_{2,2}({
\mathbb A})} F_{2}(g_{1},s,\mathbf 1; \Phi _{2}) f_{w}^{\ast}(g_{2}){
\Omega }_{\psi}(\iota _{D_{4}}(g_{1},g_{2}))\Phi _{4}(x) \psi (t + r_{1})
\\
\cdot &\varphi _{\sigma _{4}}(v(r_{1},r_{2}) u' x_{-\alpha _{2}}(r_{3})
u(0,Y';t) x_{\alpha _{2}+\alpha _{4}}(y_{5}) u(X,0;0)
J_{D_{4}}(g_{1},g_{2}))
\, dX\,dy_{5}\,dt\,dY'\,du'\,dr_{i}\,dg_{i},
\end{align*}
where $Y' = Y'(y_{6},y_{7},y_{8})$ is realised in $\mathrm{GSO}_{8}$ via
\begin{align*}
\begin{pmatrix}
I_{2} &&Y'&
\\
&I_{2} & & - w_{2}{}^{t}Y'w_{2}
\\
&&I_{2} &
\\
&&&I_{2}
\end{pmatrix}, && Y' =
\begin{pmatrix}
y_{7} & y_{8}
\\
0 & y_{6}
\end{pmatrix}.
\end{align*}
Finally using the left-invariance of the cusp form
$\varphi _{\sigma _{4}}$ we proceed by conjugating the unipotent elements
in the argument of $\varphi _{\sigma _{4}}$ with the Weyl element
$w[42]$. Then, proceeding with the Fourier expansion along the unipotent
subgroup generated by simple roots $x_{\alpha _{1}}(a_{1})$ and
$x_{\alpha _{4}}(a_{4})$, we obtain our desired identity.
\end{proof}
From the above, we have part $(a)$ of \reftext{Theorem~\ref{thm:_main_thm_VTEX1}}. Moreover,
by the uniqueness of the Whittaker model we have the following corollary.
\begin{cor}
\label{cor:_Euler_factorisation_VTEX1}
Following the notations introduced in \reftext{Theorem~\ref{thm:_Unfolding_of_period_integral_and_absolute_conv_VTEX1}}, for factorisable
integration data and for sufficiently large
$\operatorname{Re}(s), \operatorname{Re}(w) \gg 0$ (resp.
$\operatorname{Re}(w) \gg \operatorname{Re}(s) \gg 0$) the period integrals
defined in \ref{sec:_Period_Integrals_on_General_Linear_Groups_and_Spin_Similitude_Groups_VTEX1},
\ref{sec:_Period_Integrals_on_Special_Orthogonal_Similitude_Groups_VTEX1} and \ref{sec:_Glued_Period_Integrals_on_General_Linear_Groups_and_Spin_Similitude_Groups_VTEX1}
(resp. the period integrals defined in Section~\ref{sec:_Multi-Variable_Period_Integrals_VTEX1}) admit an Euler factorisation.
Namely, for $(a,b) \in \{(m,2), (2,n), (m,3)\}$, the period integrals
${\mathcal P}_{{\mathcal D},a,b}^{\mathrm{GSpin}{\times }\mathrm{GL}}(
\varphi ,s;\Phi )$ defined in Section~\ref{sec:_Period_Integrals_on_General_Linear_Groups_and_Spin_Similitude_Groups_VTEX1}
have the Euler factorisation
\begin{align*}
{\mathcal P}_{{\mathcal D},a,b}^{\mathrm{GSpin}{\times }\mathrm{GL}}(
\varphi ,s;\Phi ) =& \prod _{\nu }\mathcal{Z}_{a,b}^{\mathrm{GSpin}{
\times }\mathrm{GL}}(W_{\nu},s;\Phi _{\nu}),
\end{align*}
where their local integrals
$\mathcal{Z}_{a,b}^{\mathrm{GSpin}{\times }\mathrm{GL}}=\mathcal{Z}_{a,b}^{
\mathrm{GSpin}{\times }\mathrm{GL}}(W_{\nu},s;\Phi _{\nu})$ are defined
by
\begin{align*}
\mathcal{Z}_{m,2}^{\mathrm{GSpin}{\times }\mathrm{GL}}=& \int _{U'_{G(
\mathrm{SL}_{2}{\times }\mathrm{SL}_{2})}(k_{\nu}){\setminus }G(
\mathrm{SL}_{2}{\times }\mathrm{SL}_{2})(k_{\nu})}\mathcal{U}_{[2m-3,1^{4}]}(W_{
\tau _{m},\nu}^{\bar \psi _{\nu}})(j_{m}(g_{1},g_{2}),s;\Phi _{2,\nu})
W_{\pi _{2},\nu}^{\psi _{\nu}}(g_{2})dg_{i},
\\
\mathcal{Z}_{2,n}^{\mathrm{GSpin}{\times }\mathrm{GL}}=& \int _{ U_{
\mathrm{GSp}_{4}{\times }\mathrm{GL}_{4}}'(k_{\nu}) {\setminus }S'(
\mathrm{GSp}_{4} {\times }\mathrm{GL}_{4})(k_{\nu})} W_{\tau _{2},\nu}^{
\psi '_{\nu}}(g_{1}^{\ast}) \mathcal{U}_{[n-4,1^{4}]}(W_{\pi _{n},\nu}^{
\bar \psi _{\nu}})(g_{2},s;\Phi _{4,\nu})
\omega _{\psi _{\nu}}(g_{1}
\otimes \wedge ^{2} g_{2}) \Phi _{12,\nu}(\xi _{0})dg_{i},
\\
\mathcal{Z}_{2,3}^{\mathrm{GSpin}{\times }\mathrm{GL}}=& \int _{Z'(k_{
\nu})U_{\mathrm{GSp}_{4}{\times }\mathrm{GL}_{3}}'(k_{\nu}){
\setminus }(\mathrm{GSp}_{4} {\times }\mathrm{GL}_{3})(k_{\nu})} W_{
\tau _{2},\nu}^{\psi '_{\nu}}(g_{1}) W_{\pi _{3},\nu}^{\bar \psi _{
\nu}}(g_{2}^{\ast}) \Phi _{12,\nu}(e_{12}\gamma (g_{1}\otimes g_{2}))
|
\det (g_{1} \otimes g_{2})|^{(2s+1)/4} \,dg_{i},
\\
\mathcal{Z}_{m,3}^{\mathrm{GSpin}{\times }\mathrm{GL}}=& \int
\mathcal{U}_{[2m-5,1^{6}]}(W_{{\tau _{m},\nu}}^{\psi _{\nu}})(j_{m}(g_{1}))
W_{{\pi _{3},\nu}}^{\bar\psi _{\nu}}(g_{2}^{\ast}) \Phi _{12,\nu}(e_{12}
\gamma (\operatorname{pr}(g_{1}) \otimes g_{2}))
|\det (
\operatorname{pr}(g_{1}) \otimes g_{2})|^{(2s+1)/4} dg_{i},
\end{align*}
where the integral in
$\mathcal{Z}_{m,3}^{\mathrm{GSpin}{\times }\mathrm{GL}}$ is defined over
$Z'(k_{\nu}) U'_{\mathrm{GL}_{4} {\times }\mathrm{GL}_{3}}(k_{\nu}) {
\setminus }(\mathrm{GSpin}_{6} {\times }\mathrm{GL}_{3})(k_{\nu})$. Next,
the multivariable period integrals defined in Section~\ref{sec:_Multi-Variable_Period_Integrals_VTEX1},
${\mathcal P}_{\mathcal{D},n}^{G} = {\mathcal P}_{\mathcal{D},n}^{G}(
\varphi ,s,w,\chi ,\allowbreak \mu ;\allowbreak \Phi )$ for
$G \in \{\mathrm{GL},\mathrm{GSpin}\}$ have the Euler factorisation
\begin{align*}
{\mathcal P}_{\mathcal{D},n}^{G}(\varphi ,s,w,\chi ,\mu ;\Phi ) =&
\prod _{\nu }\mathcal{Z}_{n}^{G}(W_{\nu},s,w,\chi _{\nu},\mu _{\nu};
\Phi _{\nu}),
\end{align*}
where their local integrals
$\mathcal{Z}_{n}^{G}=\mathcal{Z}_{n}^{G}(W_{\nu},s,w,\chi _{\nu},\mu _{
\nu};\Phi _{\nu})$ are defined by
\begin{align*}
\mathcal{Z}_{n}^{\mathrm{GL}}=& \int _{\mathrm{GL}_{2}(k_{\nu})}
\chi _{\nu}(\det g) \mathcal{U}_{[n-2,1^{2}]}(W_{\pi _{n},\nu}^{
\bar\psi _{\nu}}) \left (g, \frac{s+w}{2};\Phi _{2,\nu}'\right )
F_{2}
\left (w_{2}g, \frac{w-s+1}{2},\chi _{\nu}\mu _{\nu}^{-1};\Phi _{2,
\nu}\right )\,dg,
\\
\mathcal{Z}_{n}^{\mathrm{GSpin}} =&\int \chi _{\nu}(\det g_{1})
\mathcal U_{[2n-3,1^{4}]}(W_{\tau _{n},\nu}^{\bar\psi _{\nu}})\left (j_{n}(g_{1},g_{2}),
\frac{s+w}{2};\Phi _{2,\nu}' \right )
F_{2}\left (w_{2} g_{2},
\frac{w-s+1}{2}, \chi _{\nu }\mu _{\nu}^{-1};\Phi _{2,\nu}\right )dg_{i},
\end{align*}
where the integral in $\mathcal{Z}_{n}^{\mathrm{GSpin}}$ is defined over
$U_{1}(k_{\nu}) {\setminus }G(\mathrm{SL}_{2} {\times }\mathrm{SL}_{2})(k_{
\nu})$. Also, the period integrals
 ${\mathcal P}_{\mathcal{D}, D_{4}}(\varphi ,s,w;\Phi )$ and
${\mathcal P}_{\mathcal{D}, D_{5}}(\varphi ,s;\Phi )$ defined in Section~\ref{sec:_Period_Integrals_on_Special_Orthogonal_Similitude_Groups_VTEX1} have
the Euler factorisation
\begin{align*}
{\mathcal P}_{\mathcal{D}, D_{4}}(\varphi ,s,w;\Phi ) =& \prod _{\nu }
\mathcal{Z}_{D_{4}}(W_{\nu},s,w;\Phi _{\nu}), && {\mathcal P}_{
\mathcal{D}, D_{5}}(\varphi ,s;\Phi ) = \prod _{\nu }\mathcal{Z}_{D_{5}}(W_{
\nu},s;\Phi _{\nu}),
\end{align*}
where their local integrals
$\mathcal{Z}_{D_{4}} = \mathcal{Z}_{D_{4}}(W_{\nu},s,w;\Phi _{\nu})$ and
$\mathcal{Z}_{D_{5}} = \mathcal{Z}_{D_{5}}(W_{\nu},s;\Phi _{\nu})$ are
defined by
\begin{align*}
\mathcal{Z}_{D_{4}}=&\int _{N_{S(\mathrm{GL}_{2} {\times }\mathrm{GSO}_{4})}(k_{
\nu}) {\setminus }S(\mathrm{GL}_{2} {\times }\mathrm{GSO}_{4})(k_{\nu})}
\int _{k_{\nu}} \int _{\mathrm{Mat}_{2,2}(k_{\nu})} |\det g_{1}|^{s}
\Phi _{2,\nu}(e_{2} g_{1})
\\
\cdot &
 W_{\sigma _{4},\nu}^{\bar \psi _{\nu}}(w[42]
x_{\alpha _{2} + \alpha _{4}}(r) u(X) J_{D_{4}}(g_{1},g_{2}))
f_{w,\nu}^{\ast}(g_{2}) {\Omega }_{\psi _{\nu}}(\iota _{D_{4}}(g_{1},g_{2}))
\Phi _{4,\nu}(x) \,dr\,dx\,dg_{i},
\end{align*}
\begin{align*}
\mathcal{Z}_{D_{5}}=& \int _{N_{\mathrm{GL}_{2}}(k_{\nu}) {\setminus }
\mathrm{GL}_{2}(k_{\nu})} \int _{Z(k_{\nu})} \int _{X(k_{\nu})} |
\det g|^{s-1/2} \gamma _{\psi _{\nu}}^{-1}(\det g) W_{\sigma _{5},
\nu}^{\bar \psi _{\nu}}(zx J_{D_{5}}(g))
{\Omega }_{\psi _{\nu}}(
\iota _{2}((\det g) g^{\ast})) \Phi _{2,\nu}(x) \,dx\,dz\,dg.
\end{align*}
Lastly, the glued period integrals
${\mathcal P}_{\mathcal{D},n}^{\mathbf1,G}(\varphi ,s,w;\Phi )$ and
${\mathcal P}_{\mathcal{D},m,n}^{G_{1},G_{2}}(\varphi ,s,w;\Phi )$ defined
in Section~\ref{sec:_Glued_Period_Integrals_on_General_Linear_Groups_and_Spin_Similitude_Groups_VTEX1}
have the Euler factorisation
\begin{align*}
&{\mathcal P}_{\mathcal{D},n}^{\mathbf1,G}(\varphi ,s,w;\Phi )=
\prod _{\nu }\mathcal{Z}_{n}^{\mathbf1,G}(W_{\nu},s,w;\Phi _{\nu}),
\\
& {
\mathcal P}_{\mathcal{D},m,n}^{G_{1},G_{2}}(\varphi ,s,w;\Phi ) =
\prod _{\nu }\mathcal{Z}_{m,n}^{G_{1},G_{2}}(W_{\nu},s,w;\Phi _{\nu}),
\end{align*}
where their local integrals
$\mathcal{Z}_{n}^{\mathbf1,G} = \mathcal{Z}_{n}^{\mathbf1,G}(W_{\nu},s,w;
\Phi _{\nu})$ and
$\mathcal{Z}_{m,n}^{G_{1},G_{2}}=\mathcal{Z}_{m,n}^{G_{1},G_{2}}(W_{
\nu},s,w;\Phi _{\nu})$ are defined by
\begin{align*}
\mathcal{Z}_{n}^{\mathbf1,\mathrm{GL}}=&
\int _{U_{1,2}^{\Delta}(k_{\nu}) {\setminus }\mathrm{GL}_{2}^{{
\times }2}(k_{\nu})} \mathcal U_{[n-2,1^{2}]}(W_{\pi _{n}',\nu})(g_{1},s;
\Phi _{2,\nu}) W_{\pi _{2},\nu}^{\bar \psi _{\nu}}(g_{2}^{\ast})
\\
\cdot & \omega _{\pi _{2},\nu}(\det g_{1} g_{2}) \Phi _{4,\nu}(e_{4}
\gamma _{2,2}(g_{1} \otimes g_{2})) |\det (g_{1} \otimes g_{2})|^{(2w+1)/4}
\, dg_{i},
\end{align*}
\begin{align*}
\mathcal{Z}_{n}^{\mathbf1, \mathrm{GSpin}}
&=
\int _{U_{1,3}^{\Delta}(k_{\nu}) U_{2}(k_{\nu}) {\setminus }(
\mathrm{GL}_{2}(k_{\nu}) {\times }G(\mathrm{SL}_{2}(k_{\nu}) {\times }
\mathrm{SL}_{2}(k_{\nu})))} \mathcal U_{[2n-3,1^{4}]}(W_{\tau _{n},
\nu}^{\psi _{\nu}})(j_{n}(g_{2},g_{3}),s;\Phi _{2,\nu})
\\
&\cdot  W_{\pi _{2},\nu}^{\bar\psi _{\nu}}(g_{1}^{\ast}) \omega _{
\pi _{2},\nu}(\det g_{1} g_{3}) \Phi _{4,\nu}(e_{4} \gamma _{2,2} (g_{1}
\otimes g_{3})) |\det (g_{1} \otimes g_{3})|^{(2w+1)/4}\, dg_{i},
\end{align*}
and
\begin{align*}
\mathcal{Z}_{m,n}^{\mathrm{GL},\mathrm{GL}}=& \int _{U_{1,2}^{\Delta}(k_{
\nu}) U_{3}(k_{\nu}) {\setminus }S(\mathrm{GL}_{2}^{{\times }3})(k_{
\nu})} \mathcal{U}_{[m-2,1^{2}]}(W_{\pi _{m},\nu}^{\psi _{\nu}})(g_{1},w;
\Phi _{2,\nu}) W_{\pi _{2}',\nu}^{\bar \psi _{\nu}}(g_{2}^{\ast})
\\
\cdot & \mathcal{U}_{[n-2,1^{2}]}(W_{\pi _{n}'',\nu}^{\bar \psi _{\nu}})(g_{3},s;
\Phi _{2,\nu}') \omega _{\psi _{\nu}}(\rho (g_{1},g_{2},g_{3}))\Phi _{4,
\nu}(\eta _{0})\,dg_{i},
\\
\mathcal{Z}_{m,n}^{\mathrm{GL},\mathrm{GSpin}}= &\int _{U_{1}(k_{\nu})
U_{2,3}^{\Delta}(k_{\nu}) U_{4}(k_{\nu}) {\setminus }S''(\mathrm{GL}_{2}^{{
\times }4})(k_{\nu})} \mathcal{U}_{[m-2,1^{2}]}(W_{\pi _{m},\nu}^{
\bar \psi _{\nu}})(g_{1},w;\Phi _{2,\nu}) W_{\pi _{2}',\nu}^{
\bar \psi _{\nu}}(g_{2}^{\ast})
\\
\cdot & \mathcal{U}_{[2n-3,1^{4}]}(W_{\tau _{n},\nu}^{\psi _{\nu}})(j_{n}(g_{4},g_{3}),
s;\Phi _{2,\nu}') \omega _{\psi _{\nu}}(\rho (g_{2},g_{3},g_{1}))
\Phi _{4,\nu}(\eta _{0})\,dg_{i},
\end{align*}
\begin{align*}
\mathcal{Z}_{m,n}^{\mathrm{GSpin},\mathrm{GSpin}}
&= \int _{U_{1}(k_{
\nu}) U_{2,3}^{\Delta}(k_{\nu}) U_{4}(k_{\nu}) U_{5}(k_{\nu}) {
\setminus }S^{\ast}(\mathrm{GL}_{2}^{{\times }5})(k_{\nu})}
\mathcal{U}_{[2m-3,1^{4}]}(W_{\tau _{m},\nu}^{\psi _{\nu}})(j_{m}(g_{1},g_{2}),w;
\Phi _{2,\nu}) 
\\
&\cdot W_{\pi _{2}',\nu}^{\bar \psi _{\nu}}(g_{3}^{\ast}) \mathcal{U}_{[2n-3,1^{4}]}(W_{\tau _{n}',\nu}^{\bar \psi _{\nu}})(j_{n}(g_{4},g_{5}),s;
\Phi _{2,\nu}') \omega _{\psi _{\nu}}(\rho (g_{2},g_{3},g_{5}))\Phi _{4,
\nu}(\eta _{0})\,dg_{i}.
\end{align*}
\end{cor}

\section{Unramified computation}
\label{sec:_Unramified_Computation_VTEX1}

In this section, we will first present some preliminary information used
for the unramified computation and subsequently we will evaluate the local
integrals defined in the previous section with unramified integration data.

\subsection{Preliminaries for unramified computation}
\label{sec:_Preliminaries_for_unramified_computation_VTEX1}

In this subsection, we will define some unramified integration data and
recall the Casselman-Shalika \cite{CS} and Shintani
\cite{Sh} formula for the unramified Whittaker functions. Moreover,
we will state some lemmas used in the computation of the unramified local
integrals given in subsequent sections. This section follows largely from
\citep*[\S 8]{ACS}. Throughout this section, we let $F$ denote
a non-archimedean local field of characteristic zero, with ring of integers
${\mathcal O}_{F}$ and $\varpi \in {\mathcal O}_{F}$ a uniformizer such
that $|\varpi | = q^{-1}$ for $|\cdot | = |\cdot |_{F}$ being the usual
absolute value on $F$, and $q$ is the cardinality of the residue field.
We will fix $\psi $ an additive character of $F$ which is unramified, i.e.
it is trivial on ${\mathcal O}_{F}$ but non-trivial on
$\varpi ^{-1} {\mathcal O}_{F}$. We extend $\psi $ to $\psi _{G}$ an unramified
character on $N_{G}(F)$ for
$G \in \{\mathrm{GL}_{n}(F), \mathrm{GSpin}_{2n+1}(F), \mathrm{GSO}_{2n}(F)
\}$. Similarly, we will let
$\chi ,\mu : F^{{\times }}\to {\mathbb C}^{{\times }}$ be unramified characters
such that
$\chi ({\mathcal O}_{F}^{{\times }}) = \mu ({\mathcal O}_{F}^{{
\times }}) =1$. We will also denote $\Phi _{k}^{\circ}$ to be the
characteristic function of $\mathrm{Mat}_{1,k}({\mathcal O}_{F})$. Lastly,
we denote
$f_{w}^{\circ }\in \mathrm{Ind}_{P(F)}^{\mathrm{GSO}_{4}(F)} {\delta }_{P}^{w}$
to be the normalised spherical section such that
$f_{w}^{\circ}(k) = f_{w}^{\circ}(I_{4}) =1$ for
$k \in \mathrm{GSO}_{4}({\mathcal O}_{F})$.

Let $(\pi _{n}, V_{\pi _{n}})$, $(\tau _{n}, V_{\tau _{n}})$ and
$(\sigma _{n}, V_{\sigma _{n}})$ denote an irreducible admissible unramified
$\psi $-generic representation of $\mathrm{GL}_{n}(F)$ (resp.
$\mathrm{GSpin}_{2n+1}(F)$ and $\mathrm{GSO}_{2n}(F)$). By a standard result
of unramified representations, we can assume that
\begin{align*}
\pi _{n} =& \mathrm{Ind}_{B_{\mathrm{GL}_{n}}(F)}^{\mathrm{GL}_{n}(F)}
{\delta }_{B_{\mathrm{GL}_{n}}(F)}^{1/2} \alpha ,&& \tau _{n} =
\mathrm{Ind}_{B_{\mathrm{GSpin}_{2n+1}}(F)}^{\mathrm{GSpin}_{2n+1}(F)}
{\delta }_{B_{\mathrm{GSpin}_{2n+1}(F)}}^{1/2} \beta ,
&&
\sigma _{n} =
\mathrm{Ind}_{B_{\mathrm{GSO}_{2n}}(F)}^{\mathrm{GSO}_{2n}(F)} {
\delta }_{B_{\mathrm{GSO}_{2n}}(F)}^{1/2} \gamma ,
\end{align*}
for unramified characters $\alpha ,\beta ,\gamma $ of the maximal split
torus $T_{G}(F)$ for $G = \mathrm{GL}_{n}$, $\mathrm{GSpin}_{2n+1}$ and
$\mathrm{GSO}_{2n}$ respectively. Moreover, by Satake isomorphism the representations
$\pi _{n}, \tau _{n}$ and $\sigma _{n}$ are determined by its Satake parameters
$t_{\pi _{n}}, t_{\tau _{n}}$ and $t_{\sigma _{n}}$, a semi-simple conjugacy
class in the dual group
$\mathrm{GL}_{n}^{\vee }\cong \mathrm{GL}_{n}({\mathbb C}),
\mathrm{GSpin}_{2n+1}^{\vee }\cong \mathrm{GSp}_{2n}({\mathbb C})$ and
$\mathrm{GSO}_{2n}^{\vee }\cong \mathrm{GSpin}_{2n}({\mathbb C})$ respectively.
Explicitly, the Satake parameters $t_{\pi _{n}}$ and $t_{\tau _{n}}$ have
the matrix representation
\begin{align*}
&t_{\pi _{n}} = \operatorname{diag}(\alpha _{1}(\varpi ),\dots ,
\alpha _{n}(\varpi )),
\\
&t_{\tau _{n}} = \operatorname{diag}(\beta _{1}(\varpi ), \dots ,
\beta _{n}(\varpi ), \beta _{n}^{-1}\beta _{0}(\varpi ), \dots ,
\beta _{1}^{-1}\beta _{0}(\varpi )),
\end{align*}
for $\alpha = \alpha _{1} \otimes \cdots \otimes \alpha _{n}$ and
$\beta = \beta _{0} \otimes \beta _{1} \otimes \cdots \otimes \beta _{n}$. Let
$W_{\pi _{n}}^{\circ }\in {\mathcal W}(\pi _{n}, \psi _{\mathrm{GL}_{n}})$,
$W_{\tau _{n}}^{\circ }\in {\mathcal W}(\tau _{n}, \psi _{
\mathrm{GSpin}_{2n+1}})$ and
$W_{\sigma _{n}}^{\circ }\in {\mathcal W}(\sigma _{n}, \psi _{
\mathrm{GSO}_{2n}})$ denote the normalised unramified Whittaker function.
By the Iwasawa decomposition, these Whittaker functions are completely
determined by its values on dominant elements of the torus
$T_{G}(F)$. Here, we say $t \in T_{G}(F)$ is dominant if
$|\alpha (t)|\leq 1$ for all simple roots $\alpha $. The Casselman-Shalika
\cite{CS} and Shintani \cite{Sh} formula evaluates these
normalised unramified Whittaker functions at dominant elements of
$T_{G}(F)$ and expresses them in terms of characters of finite-dimensional
representations $G^{\vee}$.

We recall that irreducible finite-dimensional algebraic representations of
$\mathrm{GL}_{n}({\mathbb C})$ and $\mathrm{Sp}_{2n}({\mathbb C})$ are
classified by their highest weights, parametrised by
$\underline{k} = (k_{1},\dots ,k_{n}) \in {\mathbb Z}^{n}$ for
$k_{1} \geq k_{2} \geq \cdots \geq k_{n}$, where for $\mathrm{Sp}_{2n}({\mathbb C})$ we require $k_n \geq0$. We will let
$\rho ^{G}(k_{1},\dots , k_{n})$ denote the irreducible finite-dimensional
representation of highest weight $(k_{1},\dots , k_{n})$ for
$G \in \{\mathrm{GL}_{n}({\mathbb C}), \mathrm{Sp}_{2n}({\mathbb C})
\}$. Furthermore, irreducible finite-dimensional representation of
$\mathrm{GSp}_{2n}({\mathbb C})$ can be written as
$\rho ^{\mathrm{GSp}_{2n}({\mathbb C})}(\underline{k}; k_{0})$ for
$k_{0} \in {\mathbb Z}$, where
$\rho ^{\mathrm{GSp}_{2n}({\mathbb C})}(\underline{k}; k_{0})\vert _{
\mathrm{Sp}_{2n}({\mathbb C})} \cong \rho ^{\mathrm{Sp}_{2n}({
\mathbb C})}(\underline k)$. We will also denote
$\chi ^{\mathrm{GL}_{n}({\mathbb C})}(\underline{k})$ (resp.
$\chi ^{\mathrm{GSp}_{2n}({\mathbb C})}(\underline{k};k_{0})$) to be the
character of the irreducible finite-dimensional representation
$\rho ^{\mathrm{GL}_{n}({\mathbb C})}(\underline{k})$ (resp.
$\rho ^{\mathrm{GSp}_{2n}({\mathbb C})}(\underline{k};k_{0})$) evaluated
at $g \in \mathrm{GL}_{n}({\mathbb C})$ (resp.
$g \in \mathrm{GSp}_{2n}({\mathbb C})$). With these, the Casselman-Shalika
\cite{CS} and Shintani \cite{Sh} formula for these unramified
Whittaker functions are given
\begin{align*}
W_{\pi _{n}}^{\circ}(t_{\underline{k}; \mathrm{GL}_{n}}(\varpi )) =& {
\delta }_{B_{\mathrm{GL}_{n}}(F)}^{1/2}(t_{\underline{k}, \mathrm{GL}_{n}}(
\varpi )) \chi ^{\mathrm{GL}_{n}({\mathbb C})}(\underline{k})(t_{\pi _{n}}),
\\
W_{\tau _{n}}^{\circ}(t_{\underline{k};\mathrm{GSpin}_{2n+1}}(\varpi ))
=& {\delta }_{B_{\mathrm{GSpin}_{2n+1}}(F)}^{1/2}(t_{\underline{k};
\mathrm{GSpin}_{2n+1}}(\varpi )) \chi ^{\mathrm{GSp}_{2n}({\mathbb C})}(
\underline{k}; \sum _{i=1}^{n} k_{i})(t_{\tau _{n}}),
\end{align*}
for $t_{\underline{k}; \mathrm{GL}_{n}}(\varpi )$ (resp.
$t_{\underline{k};\mathrm{GSpin}_{2n+1}}(\varpi )$) is the dominant element
in $T_{\mathrm{GL}_{n}}(F)$ (resp. $T_{\mathrm{GSpin}_{2n+1}}(F)$)
\begin{align*}
t_{\underline{k}; \mathrm{GL}_{n}}(\varpi ) =& \operatorname{diag}(
\varpi ^{k_{1}}, \varpi ^{k_{2}},\dots , \varpi ^{k_{n}}),&& t_{
\underline{k};\mathrm{GSpin}_{2n+1}}(\varpi ) = e_{1}^{\ast}(\varpi ^{k_{1}})
e_{2}^{\ast}(\varpi ^{k_{2}})\cdots e_{n}^{\ast}(\varpi ^{k_{n}}).
\end{align*}
For brevity, we will also drop the notation $(t_{\pi _{n}})$ (resp.
$(t_{\tau _{n}})$) when the context is clear. On the other hand, let
$\omega _{i}$ denote the $i$-th fundamental weight of
$\mathrm{GSpin}_{2n}({\mathbb C})$. For $k_{1},\dots , k_{n} \geq 0$ we
denote $(k_{1},\dots , k_{n})$ to be the character of the irreducible representation
with the highest weight $\sum _{i} k_{i} \omega _{i}$, evaluated at
$t_{\sigma _{n}}$, and by the Casselman-Shalika \cite{CS} formula,
we have a similar identity. Specifically, for representations
$\sigma _{4}$ and $\sigma _{5}$ of $\mathrm{GSO}_{8}(F)$ and
$\mathrm{GSO}_{10}(F)$ respectively, where we assume $\sigma _{4}$ to have
trivial central character (i.e. we regard $\sigma _{4}$ as a representation
of $\mathrm{PGSO}_{8}(F)$), we have
\begin{align*}
W_{\sigma _{4}}^{\circ}(t_{(k_{1},k_{2},k_{3},k_{4});\mathrm{GSO}_{8}}(
\varpi )) =& {\delta }_{B_{\mathrm{GSO}_{8}}(F)}^{1/2}(t_{(k_{1},k_{2},k_{3},k_{4});
\mathrm{GSO}_{8}}(\varpi )) (k_{1},k_{2},k_{3},k_{4}),
\\
W_{\sigma _{5}}^{\circ}(t_{(k_{1},k_{2});\mathrm{GSO}_{10}}(\varpi )) =
& {\delta }_{B_{\mathrm{GSO}_{10}}(F)}^{1/2}(t_{(k_{1},k_{2});
\mathrm{GSO}_{10}}(\varpi )) \omega_{\sigma_5}(\varpi)^{k_2} (k_{2},0,0,0,k_{1}),
\end{align*}
for $k_{i} \geq 0$, the elements
$t_{(k_{1},k_{2},k_{3},k_{4});\mathrm{GSO}_{8}}(\varpi )$ and
$t_{(k_{1},k_{2});\mathrm{GSO}_{10}}(\varpi )$ are given by
\begin{align*}
&t_{(k_{1},k_{2},k_{3},k_{4});\mathrm{GSO}_{8}}(\varpi )
=
\operatorname{diag}(\varpi ^{k_{1}+k_{2}+k_{3}+k_{4}}, \varpi ^{k_{2}+k_{3}+k_{4}},
\varpi ^{k_{3}+k_{4}}, \varpi ^{k_{4}}, \varpi ^{k_{3}}, 1, \varpi ^{-k_{2}},
\varpi ^{-k_{1}-k_{2}}),
\\
&t_{(k_{1},k_{2});\mathrm{GSO}_{10}}(\varpi )
= \operatorname{diag}(
\varpi ^{k_{1}+2k_{2}}, \varpi ^{k_{1}+k_{2}},\varpi ^{k_{1}+k_{2}},
\varpi ^{k_{1}+k_{2}},\varpi ^{k_{1}+k_{2}}, \varpi ^{k_{2}},\varpi ^{k_{2}},
\varpi ^{k_{2}},\varpi ^{k_{2}},1).
\end{align*}
Next, we will state some lemmas that will be used in the computation of
unramified local integrals.
\begin{lem}
\label{lem:_G(a,s,chi)_function_VTEX1}
Consider the auxiliary function $G(a,s,\chi )$ defined on
$F^{{\times }}$ given by
\begin{align}
\label{eqn:_G(a,s,chi)_function_VTEX1}
G(a,s,\chi ) = \int _{F} \psi (y) F_{2}(w_{2}
\begin{pmatrix}
1 & y
\\
& 1
\end{pmatrix}
\begin{pmatrix}
a &
\\
& 1
\end{pmatrix}, s, \chi ;\Phi _{2}^{\circ})\,dy.
\end{align}
Assuming $\operatorname{Re}(s)\gg 0$, the function $G(a,s,\chi )$ has the
closed formula
\begin{align*}
G(a,s,\chi ) = |a|^{1-s} \chi ^{-1}(a) \sum _{r=0}^{\mathrm{{ord}(a)}} (
\chi (\varpi ) q^{1-2s} )^{r}
\end{align*}
if $|a|\leq 1$ and zero otherwise, for
$a = {\epsilon }_{a} \varpi ^{\mathrm{{ord}(a)}}$ where
${\epsilon }_{a} \in {\mathcal O}_{F}^{{\times }}$ and
$\mathrm{{ord}(a) }\in \mathrm{{\mathbb Z}}$.
\end{lem}
\begin{proof}
This identity is a direct consequence of the orthogonality of
$\psi $. For convenience of the reader, we will sketch a proof. By the
definition of $F_{2}$, we have
\begin{align*}
G(a,s,\chi ) = \int _{F} \psi (y) \int _{F^{{\times }}} |ab^{2}|^{s}
\chi (b) \Phi _{2}^{\circ}((ab,by))\,d^{{\times }}b\, dy.
\end{align*}
Then, using the definition of $\Phi _{2}^{\circ}$ and the basic identity
$\int _{|y| \leq q^{N}} \psi (y)\, dy = q^{N}$ for $N \leq 0$ and zero
otherwise, we have
\begin{align*}
G(a,s,\chi ) = |a|^{s} \sum _{k=0}^{\mathrm{{ord}(a)}} (q^{-1+2s} \chi ^{-1}(
\varpi ))^{k},
\end{align*}
and the identity follows from a simple change of variables.
\end{proof}
\begin{lem}
\label{lem:_Schur_polynomial_identities_VTEX1}
Let $n\geq 2$ and for $k,j\geq 0$, we denote
$\bar k = (k,0,\dots ,0)$ and $\overline{j} = (j,0,\dots , 0)$ in
${\mathbb Z}^{n}$. We have the following identities for the character
$\chi ^{\mathrm{GL}_{n}({\mathbb C})}$ and
$\chi ^{\mathrm{Sp}_{2n}({\mathbb C})}$ of
$\mathrm{GL}_{n}({\mathbb C})$ and $\mathrm{Sp}_{2n}({\mathbb C})$,
\begin{align*}
\chi ^{\mathrm{GL}_{n}({\mathbb C})}(\bar k)\chi ^{\mathrm{GL}_{n}({
\mathbb C})}(\overline j) =& \sum _{t=0}^{\min (k,j)} \chi ^{
\mathrm{GL}_{n}({\mathbb C})}(\max (k,j)+t,\min (k,j)-t,0,\dots ,0),
\end{align*}
and for $k,j\geq 1$,
$\chi ^{\mathrm{Sp}_{2n}({\mathbb C})}(\bar k) \chi ^{\mathrm{Sp}_{2n}({
\mathbb C})}(\bar j) $ evaluates to
\begin{align*}
\chi ^{\mathrm{Sp}_{2n}({\mathbb C})}(\overline{k-1}) \chi ^{
\mathrm{Sp}_{2n}({\mathbb C})}(\overline{j-1}) + \sum _{p=0}^{\min (k,j)}
\chi ^{\mathrm{Sp}_{2n}({\mathbb C})}(k+j-p,p,0,\dots ,0).
\end{align*}
\end{lem}
\begin{proof}
The identity on the character $\chi ^{\mathrm{GL}_{n}({\mathbb C})}$ of
$\mathrm{GL}_{n}({\mathbb C})$ is a basic result of Schur's polynomials
and whereas the proof of the second identity involving the character
$\chi ^{\mathrm{Sp}_{2n}({\mathbb C})}$ of
$\mathrm{Sp}_{2n}({\mathbb C})$ can be found in
\citep*[Proposition 3.1]{Ma}.
\end{proof}
Next, we will recall some results on the decomposition of the symmetric
algebras given in \citep*{BR,Leahy,Brion}.
\begin{lem}
\label{lem:_Cauchy_identities_VTEX1}
Let $\pi _{n}$ (resp. $\tau _{n}$ and $\sigma _{n}$) be irreducible admissible
unramified representation of $\mathrm{GL}_{n}(F)$ (resp.
$\mathrm{GSpin}_{2n+1}(F)$ and $\mathrm{GSO}_{2n}(F)$). For
$X = q^{-s}$ and $Y = q^{-w}$ where
$\operatorname{Re}(s), \operatorname{Re}(w) \gg 0$, we have the following
identities.
\begin{enumerate}[(a)]
\item[(a)] The local standard tensor $L$-function
$L(s,\tau _{n} {\times }\pi _{2})$ for $n\geq 2$ can be written as the
formal series
\begin{align*}
L(s,\tau _{n} {\times }\pi _{2})
&= \sum _{m_{1},m_{2},m_{3}\geq 0} X^{m_{1}+2m_{2}+2m_{3}}
(\omega _{\pi _{2}}\omega _{\tau _{n}})(\varpi )^{m_{3}}
\\
&\cdot \chi ^{\mathrm{GL}_{2}({\mathbb C})}(m_{1}+m_{2},m_{2})\chi ^{
\mathrm{GSp}_{2n}({\mathbb C})}(m_{1}+m_{2},m_{2},0,\dots ,0;m_{1}+2m_{2}).
\end{align*}
\item[(b)] The local standard tensor $L$-function
$L(s, \tau _{2} {\times }\pi _{m})$ for $m\geq 4$ can be written as the
formal series
\begin{align*}
L(s, \tau _{2} &{\times }\pi _{m}) = \sum _{n_{1},\dots , n_{6} \geq 0}
X^{n_{1}+2n_{2}+3n_{3}+2n_{4}+4n_{5}+4n_{6}}
\\
\cdot &\chi ^{\mathrm{GSp}_{4}({\mathbb C})}(n_{1}+n_{2}+n_{3}+n_{5},n_{2}+n_{5};
n_{1}+2n_{2}+3n_{3}+2n_{4}+4n_{5}+4n_{6})
\\
\cdot &\chi ^{\mathrm{GL}_{m}({\mathbb C})}(n_{1}+n_{2}+n_{3}+n_{4}+2n_{5}+n_{6},
n_{2}+n_{3}+n_{4}+n_{5}+n_{6},
n_{3}+n_{5}+n_{6}, n_{6},0,\dots ,0).
\end{align*}
\item[(c)] The local standard tensor $L$-function
$L(s, \tau _{n} {\times }\pi _{3})$ for $n\geq 2$ can be written as the
formal series
\begin{align*}
L(s, \tau _{n} &{\times }\pi _{3})= \sum X^{n_{1}+2n_{2}+2n_{3}+3n_{4}+4n_{5}+3n_{6}}
\omega _{\pi _{3}}(\varpi )^{n_{4}+n_{5}+n_{6}} \omega _{\tau _{n}}(
\varpi )^{n_{3}+n_{4}+n_{5}}
\\
\cdot &\chi ^{\mathrm{GL}_{3}({\mathbb C})}(n_{1}+n_{2}+n_{3}+n_{5}, n_{2}
+ n_{3}, 0)
\\
\cdot &\chi ^{\mathrm{GSp}_{2n}({\mathbb C})}(n_{1}+n_{2}+n_{4}+n_{5}+n_{6},
n_{2}+n_{5}+n_{6}, n_{6},0,\dots ,0; n_{1} + 2n_{2} + n_{4} + 2n_{5} +
3n_{6}),
\end{align*}
where the series above is defined over
$\{n_{1},\dots , n_{6} \geq 0\}$ for $n \geq 3$ and
$\{n_{1},\dots , n_{5} \geq 0, n_{6}=0\}$ for $n=2$.
\item[(d)] Let $L(w, \sigma _{4},\mathrm{std})$ (resp.
$L(s,\sigma _{4},\mathrm{Spin})$) be the local standard $L$-function (resp.
local spin $L$-function corresponding to the fourth fundamental representation)
of $\mathrm{Spin}_{8}({\mathbb C})$. The product
$L(s, \sigma _{4},\mathrm{Spin})L(w,\sigma _{4},\mathrm{std})$ can be written
as the formal series
\begin{align*}
L(s, \sigma _{4}, \mathrm{Spin}) L(w, \sigma _{4}, \mathrm{std}) =
\zeta _{F}(2s) \zeta _{F}(2w) \sum _{n_{1},n_{2},n_{3}=0}^{\infty }(n_{1},0,n_{3},n_{2})
X^{n_{2}+n_{3}} Y^{n_{1}+n_{3}},
\end{align*}
for $\zeta _{F}(2s) = (1-q^{-2s})^{-1}$ and
$\zeta _{F}(2w) = (1-q^{-2w})^{-1}$.
\item[(e)] The local odd $\mathrm{Spin}$ L-function
$L(s,\sigma _{5}, \mathrm{Spin})$ of $\sigma _{5}$ as defined in
\citep*[\S 3.2]{Gin95a} can be written as the formal series
\begin{align*}
L(s,\sigma _{5}, \mathrm{Spin}) = \sum _{n_{1},n_{2}\geq 0} (n_{2},0,0,0,n_{1})
X^{n_{1}+2n_{2}} \omega _{\sigma _{5}}(\varpi )^{n_{2}}.
\end{align*}
\end{enumerate}
\end{lem}
\begin{proof}
The identities $(a)$-$(c)$ can be derived from classification of multiplicity-free
representations \citep*{Leahy,Brion}. This can also be derived from
\citep*[Equation 8.12]{ACS}. Moreover, the identity in $(d)$ follows from
\citep*[Lemma 5.8]{Leahy} and also \citep*[Lemma 3.3]{GH}, while the identity
in $(e)$ is due to \citet*{Brion}.
\end{proof}
Finally, we will state some vanishing results of the unramified Whittaker
functions.
\begin{lem}
\label{lem:_Vanishing_Whittaker_VTEX1}
Let $k\geq 1, j\geq 3$ and $n\geq 3$, also set
$j_{-} = \lfloor (j-1)/2 \rfloor $ and $j_{+} = j -j_{-}$. Let
$t \in T_{\mathrm{GL}_{k}}(F)$,
$(t_{1},t_{2}) \in T_{G(\mathrm{SL}_{2}{\times }\mathrm{SL}_{2})}(F)$ and
$t_{3} \in T_{\mathrm{GSpin}_{6}}(F)$ be given and let $\pi _{k+j}$,
$\tau _{n}$, $\sigma _{4}$ and $\sigma _{5}$ be irreducible admissible
unramified generic representations of $\mathrm{GL}_{k+j}(F)$,
$\mathrm{GSpin}_{2n+1}(F)$, $\mathrm{GSO}_{8}(F)$ and
$\mathrm{GSO}_{10}(F)$ respectively. Let
$W_{\pi _{k+j}}^{\circ}, W_{\tau _{n}}^{\circ}$,
$W_{\sigma _{4}}^{\circ}$ and $W_{\sigma _{5}}^{\circ}$ be their unramified
Whittaker functions. We have the following.
\begin{enumerate}[(a)]
\item[(a)] The function
\begin{align*}
\mathrm{Mat}_{j_{-}, k}(F) \to {\mathbb C}; && Y \mapsto W_{\pi _{k+j}}^{
\circ}(
\begin{pmatrix}
t &
\\
& I_{j}
\end{pmatrix}
\begin{pmatrix}
I_{k} &&
\\
Y & I_{j_{-}}&
\\
&&I_{j_{+}}
\end{pmatrix})
\end{align*}
has support in $\mathrm{Mat}_{j_{-},k}({\mathcal O}_{F})$.

\item[(b)] For $n\geq 3$, the function
\begin{align*}
\mathrm{Mat}_{n-2,2}(F) \to {\mathbb C};&& B \mapsto W_{\tau _{n}}^{
\circ}( j_{n}(t_{1},t_{2})\left (
\begin{smallmatrix}
I_{2} &&&&
\\
B&I_{n-2}&&&
\\
&&1&&
\\
&&&I_{n-2}&
\\
&&&B'&I_{2}
\end{smallmatrix}\right )),
\end{align*}
has support in $\mathrm{Mat}_{n-2,2}({\mathcal O}_{F})$.

\item[(c)] For $n\geq 4$, the function
\begin{align*}
\mathrm{Mat}_{n-3,3}(F) \to {\mathbb C}; && B \mapsto W_{\tau _{n}}^{
\circ}(j_{n}(t_{3})\left (
\begin{smallmatrix}
I_{3} &&&&
\\
B&I_{n-3}&&&
\\
&&1&&
\\
&&&I_{n-3}&
\\
&&&B'&I_{3}
\end{smallmatrix}\right ))
\end{align*}
has support in $\mathrm{Mat}_{n-3,3}({\mathcal O}_{F})$.
\item[(d)] Let $x_{-\alpha _{4}}(r) = I_{8} + r(E_{53} - E_{64})$ and
$t = \operatorname{diag}(a_{1},1,1,a_{2}, a_{3}, a_{2} a_{3}, a_{2}a_{3},
a_{1}^{-1}a_{2}a_{3}) \in \mathrm{GSO}_{8}(F)$, the function
\begin{align*}
F \to {\mathbb C}; && r \mapsto W_{\sigma _{4}}^{\circ}(t x_{-\alpha _{4}}(r))
\end{align*}
has support in ${\mathcal O}_{F}$.

\item[(e)] Let
$\mathbf z(z_{1},z_{2},z_{3},z_{4}) = I_{10} + z_{1}(E_{2,1} - E_{10,9})
+z_{2}(E_{6,1} - E_{10,5}) + z_{3}(E_{6,3} - E_{8,5}) + z_{4}(E_{6,4} -
E_{7,5})$ in $\mathrm{GSO}_{10}(F)$ and recalling the map
$J_{D_{5}}: \mathrm{GL}_{2} \to \mathrm{GSO}_{10}$ given in \reftext{\eqref{eqn:_J-D5_VTEX1}}, the function
\begin{align*}
F{\times }F {\times }F {\times }F \to {\mathbb C}; && (z_{1},z_{2},z_{3},z_{4})
\mapsto W_{\sigma _{5}}^{\circ}(J_{D_{5}}(\operatorname{diag}(a_{1}a_{2},a_{2}))
\mathbf z(z_{1},z_{2},z_{3},z_{4}))
\end{align*}
has support in
${\mathcal O}_{F} {\times }{\mathcal O}_{F} {\times }{\mathcal O}_{F} {
\times }{\mathcal O}_{F}$.
\end{enumerate}
\end{lem}
\begin{proof}
A proof of $(a)$ can be found in \citep*[\S 12]{S2} and a proof
of $(b)$ and $(c)$ is analogous to that of \citet*[\S 4]{G}.
Statement $(d)$ can be derived from the proof of
\citep*[Proposition 3.1]{GH} while $(e)$ is precisely
\citep*[Equation 3.8]{Gin95a}. For convenience of the reader, we will sketch
a proof for identity $(d)$. Let
$x_{\alpha _{2} + \alpha _{4}}(z) = I_{8} + z(E_{25} - E_{47})$ and
$x_{\alpha _{2}}(y) = I_{8} + y(E_{23} - E_{67})$ in
$\mathrm{GSO}_{8}(F)$. Then, for $|z|\leq 1$ we have
\begin{align*}
W_{\sigma _{4}}^{\circ}(t x_{-\alpha _{4}}(r)) =& W_{\sigma _{4}}^{
\circ}(t x_{-\alpha _{4}}(r) x_{\alpha _{2} + \alpha _{4}}(z))
\\
=&\psi (-rz) W_{\sigma _{4}}^{\circ}(t x_{-\alpha _{4}}(r)).
\end{align*}
Thus, by the assumption $\psi ({\mathcal O}_{F}) =1$, we must have $|r|\leq 1$ for $r$ in the support.
\end{proof}

\subsection{Unramified computation of local unramified integrals}
\label{sec6.2}

In this subsection, we will evaluate unramified local integrals given in
\reftext{Corollary~\ref{cor:_Euler_factorisation_VTEX1}} for unramified integration data
$W^{\circ},\Phi ^{\circ},\chi ,\mu $ defined in Section~\ref{sec:_Preliminaries_for_unramified_computation_VTEX1}. For these unramified
local integrals, we will show that they evaluate to the $L$-functions representing
the representations given in \reftext{Table~\ref{Table:_Multiplicity-free_repn_VTEX1}} as expected by \citet*{BZSV,MWZ}. More
precisely, we have the following theorem.
\begin{thm}%
\label{thm:_unramified_comp_VTEX1}
Following the notations in \reftext{Lemma~\ref{lem:_Cauchy_identities_VTEX1}}, for
$\operatorname{Re}(s), \operatorname{Re}(w)\gg 0$ (and $\operatorname{Re}(w) \gg \operatorname{Re}(s)\gg 0$), the unramified local integrals listed below are evaluated as follows.
\begin{enumerate}[(a)]
\item[(a)] The unramified local integrals
$\mathcal{Z}_{a,b}^{\mathrm{GSpin}{\times }\mathrm{GL}}(W^{\circ},s;
\Phi ^{\circ})$ for $(a,b) \in \{(m,2), (2,n), (m,3)\}$ evaluate to
\begin{align*}
\mathcal{Z}_{a,b}^{\mathrm{GSpin}{\times }\mathrm{GL}}(W^{\circ},s;
\Phi ^{\circ}) =& L(s, \tau _{a} {\times }\pi _{b}).
\end{align*}
\item[(b)] The unramified local integrals
$\mathcal{Z}_{n}^{G}(W^{\circ},s,w,\chi ,\mu ;\Phi ^{\circ})$ for
$G \in \{\mathrm{GL},\mathrm{GSpin}\}$ evaluate to
\begin{align*}
\mathcal{Z}_{n}^{\mathrm{GL}}(W^{\circ},s,w,\chi ,\mu ;\Phi ^{\circ})=&
L(s,\pi _{n} {\times }\mu ) L(w, \pi _{n} {\times }\chi ),
\\
\mathcal{Z}_{n}^{\mathrm{GSpin}}(W^{\circ},s,w,\chi ,\mu ;\Phi ^{
\circ})=& L(s, \tau _{n} {\times }\mu ) L(w, \tau _{n} {\times }\chi ).
\end{align*}
\item[(c)] The unramified local integrals
$\mathcal{Z}_{D_{4}}(W^{\circ},s,w;\Phi ^{\circ})$ and
$\mathcal{Z}_{D_{5}}(W^{\circ},s;\Phi ^{\circ})$ evaluate to
\begin{align*}
\mathcal{Z}_{D_{4}}(W^{\circ},s,w;\Phi ^{\circ}) =& L(s, \sigma _{4},
\mathrm{Spin}) L(w, \sigma _{4}, \mathrm{std}),
\\
\mathcal{Z}_{D_{5}}(W^{\circ},s;\Phi ^{\circ}) =& L(s, \sigma _{5},
\mathrm{Spin}).
\end{align*}
\item[(d)] The unramified local integrals
$\mathcal{Z}_{n}^{\mathbf 1, G}(W^{\circ},s,w;\Phi ^{\circ})$ and
$\mathcal{Z}_{m,n}^{G_{1},G_{2}}(W^{\circ},s,w;\Phi ^{\circ})$ for
$G ,G_{1},G_{2} \in \{\mathrm{GL},\mathrm{GSpin}\}$ evaluate to
\begin{align*}
\mathcal{Z}_{n}^{\mathbf1,\mathrm{GL}}(W^{\circ},s,w;\Phi ^{\circ}) =&
L(w, \pi _{2}) L(s, \pi _{2} {\times }\pi _{n}'),
\\
\mathcal{Z}_{n}^{\mathbf1,\mathrm{GSpin}}(W^{\circ},s,w;\Phi ^{\circ})
=& L(w,\pi _{2}) L(s, \pi _{2} {\times }\tau _{n}),
\\
\mathcal{Z}_{m,n}^{\mathrm{GL},\mathrm{GL}}(W^{\circ},s,w;\Phi ^{
\circ}) =& L(w, \pi _{m} {\times }\pi _{2}') L(s, \pi _{2}' {\times }
\pi _{n}''),
\\
\mathcal{Z}_{m,n}^{\mathrm{GL},\mathrm{GSpin}}(W^{\circ},s,w;\Phi ^{
\circ})=& L(w, \pi _{m} {\times }\pi _{2}') L(s,\pi _{2}' {\times }
\tau _{n}),
\\
\mathcal{Z}_{m,n}^{\mathrm{GSpin},\mathrm{GSpin}}(W^{\circ},s,w;\Phi ^{
\circ})=& L(w, \tau _{m} {\times }\pi _{2}') L(s, \pi _{2}' {\times }
\tau _{n}').
\end{align*}
\end{enumerate}
\end{thm}
\begin{proof}
We will omit the details for the computation of the unramified local integral
 $\mathcal{Z}_{m,2}^{\mathrm{GSpin}{\times }\mathrm{GL}}(W^{\circ}, s;
\Phi ^{\circ})$ as it is a direct consequence of \reftext{Lemmas~\ref{lem:_Cauchy_identities_VTEX1}, \ref{lem:_Vanishing_Whittaker_VTEX1}} and the Casselman-Shalika
\cite{CS} and Shintani \cite{Sh} formula for
$W_{\tau _{m}}^{\circ}$ and $W_{\pi _{2}}^{\circ}$. Next, we consider the
unramified local integral
$\mathcal{Z}_{2,n}^{\mathrm{GSpin}{\times }\mathrm{GL}}(W^{\circ},s;
\Phi ^{\circ})$. By Iwasawa decomposition, we have
\begin{align*}
S'(\mathrm{GSp}_{4} {\times }\mathrm{GL}_{4})(F) = N_{S'(\mathrm{GSp}_{4}
{\times }\mathrm{GL}_{4})}(F) \cdot T_{S'(\mathrm{GSp}_{4} {\times }
\mathrm{GL}_{4})}(F)\cdot S'(\mathrm{GSp}_{4} {\times }\mathrm{GL}_{4})({
\mathcal O}_{F}),
\end{align*}
where we parametrise
$U'_{\mathrm{GSp}_{4} {\times }\mathrm{GL}_{4}}(F){\setminus }N_{S'(
\mathrm{GSp}_{4} {\times }\mathrm{GL}_{4})}(F)$ and
$T_{S'(\mathrm{GSp}_{4} {\times }\mathrm{GL}_{4})}(F)$ as
\begin{align*}
&Y(y_{1},y_{2},y_{3}) =(I_{4}, \left (
\begin{smallmatrix}
1 & y_{1} & y_{2} &
\\
&1&y_{3}&
\\
&&1&
\\
&&&1
\end{smallmatrix}\right )),
\\
&\begin{aligned}
t(a) =& ( \operatorname{diag}(a_{1}a_{3}^{-1}a_{4}^{-2}a_{5}^{-3}a_{6}^{-2},
a_{2}a_{3}^{-1}a_{4}^{-2}a_{5}^{-3}a_{6}^{-2}, a_{2}^{-1}a_{6}^{-2}, a_{1}^{-1}a_{6}^{-2}),
\operatorname{diag}(a_{3}a_{4}a_{5}a_{6}, a_{4}a_{5}a_{6}, a_{5}a_{6},
a_{6}))
\end{aligned}
\end{align*}
respectively. From the formula of the Weil representation \reftext{\eqref{eqn:_formula_Weil_-_Levi_Siegel_VTEX1}} and \reftext{Lemma~\ref{lem:_Vanishing_Whittaker_VTEX1}} the unramified local integral becomes
\begin{align*}
\int |a_{1}|^{-2} &|a_{2}|^{-1} |a_{3}|^{s+\frac{4-n}{2}} |a_{4}|^{2s+5-n}
|a_{5}|^{3s+\frac{18-3n}{2}} |a_{6}|^{4s+8-2n}
\\
\cdot &W_{\tau _{2}}^{\circ}(\operatorname{diag}(a_{1}a_{6}^{2},a_{2}a_{6}^{2},a_{2}^{-1}a_{3}a_{4}^{2}a_{5}^{3}a_{6}^{2},a_{1}^{-1}a_{3}a_{4}^{2}a_{5}^{3}a_{6}^{2})
)
W_{\pi _{n}}^{\circ}(j_{n,4}(\left (
\begin{smallmatrix}
a_{3}a_{4}a_{5}a_{6}&&&
\\
&a_{4}a_{5}a_{6}&&
\\
&&a_{5}a_{6}&
\\
&&&a_{6}
\end{smallmatrix}\right ))\,d^{{\times }}a_{i},
\end{align*}

where the integral above is defined over
\begin{align*}
\left \{ a_{i} \in F^{{\times }} :
\begin{matrix}
|a_{6}|\leq 1, |a_{1}^{-1}a_{3}a_{4}^{2}a_{5}^{2}| \leq 1, |a_{2}^{-1}a_{3}a_{4}a_{5}^{2}|
\leq 1,|a_{1}a_{3}^{-1}a_{4}^{-1}a_{5}^{-2}|\leq 1,
\\
|a_{2}a_{3}^{-1}a_{4}^{-1}a_{5}^{-1}|\leq 1, |a_{2}a_{4}^{-1}a_{5}^{-2}|
\leq 1, |a_{1}a_{4}^{-1}a_{5}^{-2}|\leq 1
\end{matrix} \right \}.
\end{align*}
Then, by the Casselman-Shalika \cite{CS} and Shintani
\cite{Sh} formula for $W_{\tau _{2}}^{\circ}$ and
$W_{\pi _{n}}^{\circ}$ and \reftext{Lemma~\ref{lem:_Cauchy_identities_VTEX1}}, we obtain
the desired identity for
$\mathcal{Z}_{2,n}^{\mathrm{GSpin}{\times }\mathrm{GL}}(W^{\circ},s;
\Phi ^{\circ})$. Moving on, we consider the unramified local integral
$\mathcal{Z}_{m,3}^{\mathrm{GSpin}{\times }\mathrm{GL}}(W^{\circ},s;
\Phi ^{\circ})$ where here we are assuming
$\omega _{\tau _{m}} \omega _{\pi _{3}} = \mathbf 1$. We will only provide
details for $m\geq 3$, the case of $m=2$ follows the same argument. By
Iwasawa decomposition, we have
\begin{align*}
(\mathrm{GSpin}_{6} {\times }\mathrm{GL}_{3})(F) = N_{\mathrm{GSpin}_{6}
{\times }\mathrm{GL}_{3}}(F) \cdot T_{\mathrm{GSpin}_{6} {\times }
\mathrm{GL}_{3}}(F) \cdot (\mathrm{GSpin}_{6} {\times }\mathrm{GL}_{3})({
\mathcal O}_{F}),
\end{align*}
where we parametrise
$U'_{\mathrm{GL}_{4} {\times }\mathrm{GL}_{3}}(F) {\setminus }N_{
\mathrm{GSpin}_{6} {\times }\mathrm{GL}_{3}}(F)$ and
$Z'(F){\setminus }T_{\mathrm{GSpin}_{6} {\times }\mathrm{GL}_{3}}(F)$ as
\begin{align*}
&X(x_{1},x_{2},x_{3}) = (e, \left (
\begin{smallmatrix}
1 & x_{1} & x_{2}
\\
&1&x_{3}
\\
&&1
\end{smallmatrix}\right )),
\\
& t(a) =(e_{0}^{\ast}(a_{6}) e_{1}^{\ast}(a_{3}a_{4}a_{5})
e_{2}^{\ast}(a_{4}a_{5}) e_{3}^{\ast}(a_{5}),\operatorname{diag}(1,a_{2}^{-1},a_{1}^{-1}a_{2}^{-1})),
\end{align*}
respectively. Here, we used $e$ to denote the identity element of
$\mathrm{GSpin}_{6}(F)$. Note that
\begin{align*}
\operatorname{pr}(e_{0}^{\ast}(a_{6}) e_{1}^{\ast}(a_{3}a_{4}a_{5}) e_{2}^{
\ast}(a_{4}a_{5}) e_{3}^{\ast}(a_{5})) = \operatorname{diag}(a_{3}a_{4}^{2}a_{5}^{3}a_{6},
a_{3}a_{4}a_{5}a_{6}, a_{4}a_{5}a_{6}, a_{5}a_{6})\in \mathrm{GL}_{4}'(F).
\end{align*}
Then, by the similar argument as above, the unramified local integral evaluates
to
\begin{align*}
\int _{a_{i} \in {\Omega }_{m}}
& |a_{1}|^{-1-2s} |a_{2}|^{-1-4s} |a_{3}|^{
\frac{1}{2}-m+3s} |a_{4}|^{2-2m+6s} |a_{5}|^{\frac{9}{2}-3m+9s} |a_{6}|^{6s}
\\
\cdot &W_{\tau _{m}}^{\circ}(j_{m}(e_{0}^{\ast}(a_{6}) e_{1}^{\ast}(a_{3}a_{4}a_{5})
e_{2}^{\ast}(a_{4}a_{5}) e_{3}^{\ast}(a_{5}))) W_{\pi _{3}}^{\circ}(
\left (
\begin{smallmatrix}
a_{1}a_{2}&&
\\
&a_{2}&
\\
&&1
\end{smallmatrix}\right ))\,d^{{\times }}a_{i}
\end{align*}
where
\begin{align*}
{\Omega }_{m} =\left \{a_{i} \in F^{{\times }} :
\begin{matrix}
|a_{1}^{-1}a_{2}^{-1} a_{3}a_{4}a_{5}a_{6}|\leq 1, |a_{2}^{-1}a_{4}a_{5}a_{6}|
\leq 1, |a_{5}a_{6}|\leq 1,
\\
|a_{1}a_{2}a_{4}^{-1}a_{5}^{-1}a_{6}^{-1}|\leq 1, |a_{2}a_{5}^{-1}a_{6}^{-1}|
\leq 1, |a_{1}a_{2}a_{5}^{-1}a_{6}^{-1}|\leq 1
\end{matrix} \right \}.
\end{align*}
Then, by the Casselman-Shalika \cite{CS} and Shintani
\cite{Sh} formula for $W_{\tau _{m}}^{\circ}$ and
$W_{\pi _{3}}^{\circ}$ and \reftext{Lemma~\ref{lem:_Cauchy_identities_VTEX1}}, we obtain
the desired identity for
$\mathcal{Z}_{m,3}^{\mathrm{GSpin}{\times }\mathrm{GL}}(W^{\circ},s;
\Phi ^{\circ})$. Next, for the unramified local integral
$\mathcal{Z}_{n}^{G}(W^{\circ},s,w,\chi ,\mu ;\Phi ^{\circ})$ for
$G \in \{\mathrm{GL},\mathrm{GSpin}\}$. Similarly, we will only provide
details for
$\mathcal{Z}_{n}^{\mathrm{GL}}(W^{\circ},s,w,\chi ,\mu ;\Phi ^{\circ})$
as the other case follows the same argument. Again, by Iwasawa decomposition
for $\mathrm{GL}_{2}(F)$ we have the parametrisation
\begin{align*}
N_{\mathrm{GL}_{2}}(F)=\{n(x)\mid x \in F\},&& T_{\mathrm{GL}_{2}}(F)=
\left \{t(a_{1},a_{2})=\operatorname{diag}(a_{1}a_{2},a_{2})\mid a_{i}
\in F^{{\times }} \right \},
\end{align*}
such that the unramified local integral
$\mathcal{Z}_{n}^{\mathrm{GL}}(W^{\circ},s,w,\chi ,\mu ;\Phi ^{\circ})$
evaluates to
\begin{align*}
\int _{\substack{a_1 \in F^\times\\ |a_2|\leq 1}} &\chi (a_{1}a_{2}) \mu (a_{2}) |a_{1}|^{
\frac{s+w-n}{2}} |a_{2}|^{s+w+2-n} W_{\pi _{n}}^{\circ}(j_{n,2}(t(a_{1},a_{2})))
G\left ( a_{1}, \frac{w-s+1}{2}, \chi \mu ^{-1}\right )\,d^{{\times }}a_{i},
\end{align*}
where $G(a_{1},(w-s+1)/2,\chi \mu ^{-1})$ is the function defined in \reftext{\eqref{eqn:_G(a,s,chi)_function_VTEX1}}. Then, by \reftext{Lemma~\ref{lem:_G(a,s,chi)_function_VTEX1}}, the Casselman-Shalika-Shintani
\citep*{CS,Sh} formula for $W_{\pi _{n}}^{\circ}$ as well as \reftext{Lemma~\ref{lem:_Schur_polynomial_identities_VTEX1}}, we obtain the desired identity
for
$\mathcal{Z}_{n}^{\mathrm{GL}}(W^{\circ},s,w,\chi ,\mu ;\Phi ^{\circ})$.
Next, we will evaluate the unramified local integrals
$\mathcal{Z}_{D_{5}}(W^{\circ},s;\Phi ^{\circ})$ and
$\mathcal{Z}_{D_{4}}(W^{\circ},s,w;\Phi ^{\circ})$. By the same argument
as above, we can write the unramified local integral
$\mathcal{Z}_{D_{5}}(W^{\circ},s;\Phi ^{\circ})$ as
\begin{align*}
\mathcal{Z}_{D_{5}}(W^{\circ},s;\Phi ^{\circ})=\int _{F^{{\times }2}} |a_{1}|^{s-5}
|a_{2}|^{2s-4} W_{\sigma _{5}}^{\circ}(J_{D_{5}}(t(a_{1},a_{2})))\,d^{{
\times }}a_{i},
\end{align*}
and given that $J_{D_{5}}(t(a_{1},a_{2}))$ is exactly that of
$j(\operatorname{diag}(a_{1}a_{2},a_{2}))$ in
\citep*[Proposition 3.2]{Gin95a}, the unramified computation of
$\mathcal{Z}_{D_{5}}(W^{\circ}, s;\Phi ^{\circ})$ follows the same argument
in \xch{\cite{Gin95a}}{\textit{loc. cit}}. Similarly, for the integral
$\mathcal{Z}_{D_{4}}(W^{\circ},s,w;\Phi ^{\circ})$ we proceed by Iwasawa
decomposition to obtain
\begin{align*}
S(\mathrm{GL}_{2}{\times }\mathrm{GSO}_{4})(F) = N_{S(\mathrm{GL}_{2}{
\times }\mathrm{GSO}_{4})}(F) \cdot T_{S(\mathrm{GL}_{2}{\times }
\mathrm{GSO}_{4})}(F) \cdot {S(\mathrm{GL}_{2}{\times }\mathrm{GSO}_{4})}({
\mathcal O}_{F}),
\end{align*}
with the parametrisation
\begin{align*}
T_{S(\mathrm{GL}_{2}{\times }\mathrm{GSO}_{4})}(F)= \left \{ (
\operatorname{diag}(a_{0}a_{1},a_{0}) , \operatorname{diag}(a_{0}^{-1}a_{1}^{-1}a_{2},a_{0}^{-1}a_{1}^{-1}a_{3},a_{0}^{-1}a_{3}^{-1},a_{0}^{-1}a_{2}^{-1}))
\right \}.
\end{align*}
Observe that
\begin{align*}
J_{D_{4}}((\operatorname{diag}(a_{0}a_{1},a_{0}) ,
\operatorname{diag}(a_{0}^{-1}a_{1}^{-1}a_{2},a_{0}^{-1}a_{1}^{-1}a_{3},a_{0}^{-1}a_{3}^{-1},a_{0}^{-1}a_{2}^{-1}))
) = (a_{0}^{-1}I_{8}) t(a_{1},a_{2},a_{3}),
\end{align*}
where $t(a_{1},a_{2},a_{3})$ is the element in $\mathrm{GSO}_{8}(F)$ given
by
\begin{align*}
t(a_{1},a_{2},a_{3}) = \operatorname{diag}( a_{1}^{-1}a_{2}, a_{1}^{-1}a_{3},
1, a_{1}^{-1}, 1, a_{1}^{-1}, a_{3}^{-1}, a_{2}^{-1}).
\end{align*}
Furthermore, we have
\begin{align*}
w[42] x_{\alpha _{2} + \alpha _{4}}(r) w[42]^{-1} =& x_{-\alpha _{4}}(r),
\\
w[42]t(a_{1},a_{2},a_{3})w[42]^{-1} =& \operatorname{diag}(a_{1}^{-1}a_{2},1,1,a_{3}^{-1},a_{1}^{-1}a_{3},
a_{1}^{-1},a_{1}^{-1},a_{2}^{-1}).
\end{align*}
Thus, using \reftext{Lemma~\ref{lem:_Vanishing_Whittaker_VTEX1}} and the formula \reftext{\eqref{eqn:_formula_Weil_-_Levi_Siegel_VTEX1}} for the Weil representation as
well as the fact the $W_{\sigma _{4}}^{\circ}$ is
$Z(\mathrm{GSO}_{8})(F)$-invariant for $Z(\mathrm{GSO}_{8})(F)$ is the
center of $\mathrm{GSO}_{8}(F)$, the above integral becomes
\begin{align*}
\zeta _{F}(2s)\zeta _{F}(2w) \int _{F^{{\times }3}} |a_{1}|^{s-w} |a_{2}|^{w-3}
|a_{3}|^{w} W_{\sigma _{4}}^{\circ}(\operatorname{diag}(a_{2}, a_{1}, a_{1},
a_{1}a_{3}^{-1}, a_{3}, 1, 1, a_{1}a_{2}^{-1}))\,d^{{\times }}a_{i}.
\end{align*}
By the Casselman-Shalika \cite{CS} formula for
$W_{\sigma _{4}}^{\circ}$ and \reftext{Lemma~\ref{lem:_Cauchy_identities_VTEX1}}, we have
the desired identity for
$\mathcal{Z}_{D_{4}}(W^{\circ},s,w;\Phi ^{\circ})$. Finally, for the glued
unramified local integrals
$\mathcal{Z}_{n}^{\mathbf1,G}(W^{\circ},s,w;\Phi ^{\circ})$,
$\mathcal{Z}_{m,n}^{G_{1},G_{2}}(W^{\circ},s,w;\Phi ^{\circ})$ where
$G_{1},G_{2} \in \{\mathrm{GL}, \mathrm{GSpin}\}$, we will only provide
details for the evaluation of
$\mathcal{Z}_{m,n}^{\mathrm{GL},\mathrm{GL}}(W^{\circ},s,w;\Phi ^{
\circ})$ as the other integrals follow the same argument. By Iwasawa decomposition,
we have
\begin{align*}
S(\mathrm{GL}_{2}^{{\times }3})(F) = N_{S(\mathrm{GL}_{2}^{{\times }3})}(F)
\cdot T_{S(\mathrm{GL}_{2}^{{\times }3})}(F) \cdot S(\mathrm{GL}_{2}^{{
\times }3})({\mathcal O}_{F}),
\end{align*}
where
$U_{1,2}^{\Delta}(F) U_{3}(F) {\setminus }N_{S(\mathrm{GL}_{2}^{{
\times }3})}(F)$ and $T_{S(\mathrm{GL}_{2}^{{\times }3})}(F)$ have the
parametrisation
\begin{align*}
&u(x) = (n(x), I_{2}, I_{2}),
\\
& t(a) = (\operatorname{diag}(a_{1}a_{2},a_{2}),
\operatorname{diag}(a_{3},a_{1}^{-1}a_{2}^{-2}a_{3}^{-1}a_{4}^{-1}a_{5}^{-2}),
\operatorname{diag}(a_{4}a_{5},a_{5})),
\end{align*}
respectively. By the same argument as above, the unramified local integral
$\mathcal{Z}_{m,n}^{\mathrm{GL},\mathrm{GL}}(W^{\circ},\allowbreak s,w;\Phi ^{
\circ})$ evaluates to
\begin{align*}
&\int |a_{1}|^{w-m/2} |a_{2}|^{2w-m+1} |a_{3}|^{-1} |a_{4}|^{s-n/2} |a_{5}|^{2s-n+1}
\\
&\cdot  W_{\pi _{m}}^{\circ}(j_{m,2}(
\begin{pmatrix}
a_{1}a_{2}&
\\
&a_{2}
\end{pmatrix})) W_{\pi _{2}'}^{\circ}(
\begin{pmatrix}
a_{1}a_{2}^{2}a_{3}a_{4}a_{5}^{2}&
\\
& a_{3}^{-1}
\end{pmatrix}) W_{\pi _{n}''}^{\circ}(j_{n,2}(
\begin{pmatrix}
a_{4}a_{5}&
\\
& a_{5}
\end{pmatrix}))\,d^{{\times }}a_{i}
\end{align*}
where the integral is defined over
$\{a_{i} \in F^{{\times }}: |a_{2}|, |a_{5}|, |a_{2}a_{3}a_{5}|^{-1}, |a_{2}a_{3}a_{4}a_{5}|,
|a_{1}a_{2}a_{3}a_{5}|\leq 1\}$. By the Casselman-Shalika-Shintani
\citep*{CS,Sh} formula for
$W_{\pi _{m}}^{\circ},W_{\pi _{2}'}^{\circ}$ and
$W_{\pi _{n}''}^{\circ}$ as well as \reftext{Lemma~\ref{lem:_Schur_polynomial_identities_VTEX1}}, we obtain the desired identity
for
$\mathcal{Z}_{m,n}^{\mathrm{GL},\mathrm{GL}}(W^{\circ},s,w;\Phi ^{
\circ})$.
\end{proof}
Thus from the above as well as \reftext{Corollary~\ref{cor:_Euler_factorisation_VTEX1}}, we have part $(b)$ of \reftext{Theorem~\ref{thm:_main_thm_VTEX1}}.

\section{Relation with previously studied Rankin-Selberg integrals}
\label{sec:_Relation_with_previously_studied_Rankin-Selberg_integrals_VTEX1}

In this section, we will discuss the relations of some of the period integrals
constructed in Section~\ref{sec:_The_Period_integrals_and_their_unfolding_process_VTEX1} with some previously
studied Rankin-Selberg integrals. These include the works of
\citet*{Bump1} and \citet*{ACS} on Rankin-Selberg integrals defined on
$\mathrm{GSpin}{\times }\mathrm{GL}$, and also \citet*{Gin95a}'s Bessel-type
Rankin-Selberg integral defined on $\mathrm{GSO}_{10}$.

\subsection{Relation with Bump's and Asgari-Cogdell-Shahidi's Rankin-Selberg integrals}
\label{sec7.1}

In this subsection, we will show the
$\mathrm{GSpin}_{2m+1} {\times }\mathrm{GL}_{3} (m \geq 2)$-period integral
${\mathcal P}_{\mathcal{D},m,3}^{\mathrm{GSpin}{\times }\mathrm{GL}}={
\mathcal P}_{\mathcal{D},m,3}^{\mathrm{GSpin}{\times }\mathrm{GL}}(
\varphi , s;\Phi )$ defined in Section~\ref{sec:_Period_Integrals_on_General_Linear_Groups_and_Spin_Similitude_Groups_VTEX1}
is equivalent to that proposed by \citet*{ACS} and \citet*{Bump1} for the
$\mathrm{GSp}_{4} {\times }\mathrm{GL}_{3}$ integral. We shall provide
details for $m\geq 3$ as the same argument holds for when $m=2$. Assuming
$m\geq 3$, \citet*{ACS} constructed a global
$\mathrm{GSpin}_{2m+1} {\times }\mathrm{GL}_{3}$-period integral defined by
\begin{align}
\label{eqn:_ACS_integral_VTEX1}
\int _{Z({\mathbb A}) \mathrm{GSpin}_{6}(k) {\setminus }
\mathrm{GSpin}_{6}({\mathbb A})} {\mathcal B}_{[2m-5,1^{6}]}(\varphi _{
\tau _{m}})(j_{m}(g)) E(g, f_{\pi _{3},s})\,dg
\end{align}
where ${\mathcal B}_{[2m-5,1^{6}]}(\varphi _{\tau _{m}})$ is the Bessel
coefficient of $\varphi _{\tau _{m}}$ given in Section~\ref{sec:_Bessel_coefficient_GSpin_VTEX1} and $E(s, f_{\pi _{3}})$ is the cuspidal
Eisenstein series of $\mathrm{GSpin}_{6}({\mathbb A})$ associated to the induced space
\begin{align}
\label{eqn:_ACS_induced_space_VTEX1}
\rho _{\pi _{3},s} = \mathrm{Ind}_{P({\mathbb A})}^{\mathrm{GSpin}_{6}({
\mathbb A})}(\pi _{3}|\det |^{s} \otimes \omega _{\tau _{m}}^{-1}),
\end{align}
where $P$ is the Siegel parabolic of $\mathrm{GSpin}_{6}$ defined by deleting
the simple root $e_{2} + e_{3}$. With these, we will show that the
$\mathrm{GSpin}_{2m+1}{\times }\mathrm{GL}_{3}$-period integral
${\mathcal P}_{\mathcal{D},m,3}^{\mathrm{GSpin}{\times }\mathrm{GL}}$ admits
an inner integral which can be unfolded to the Eisenstein series
$E(\cdot , f_{\pi _{3},s})$ in \reftext{\eqref{eqn:_ACS_integral_VTEX1}}.

Recall we assumed that
$\omega _{\tau _{m}} \omega _{\pi _{3}} = \mathbf 1$, and observe that
there exists an inner period integral in
${\mathcal P}_{\mathcal{D},m,3}^{\mathrm{GSpin}{\times }\mathrm{GL}}$ given by
\begin{align}
\label{eqn:_Inner_integral_ACS_VTEX1}
&\mathfrak{E}(\varphi _{\pi _{3}}, \Phi _{12},s; \omega _{\tau _{m}})(g_{1})
= \int _{Z_{3}({\mathbb A}) \mathrm{GL}_{3}(k) {\setminus }
\mathrm{GL}_{3}({\mathbb A})} \varphi _{\pi _{3}}(g_{2}^{\ast}) E_{12}
\left ( \operatorname{pr}(g_{1}) \otimes g_{2}, \frac{2s+1}{4},
\omega _{\tau _{m}}; \Phi _{12} \right )\,dg_{2},
\end{align}
such that
\begin{align*}
{\mathcal P}_{\mathcal{D},m,3}^{\mathrm{GSpin}{\times }\mathrm{GL}} =
\int _{Z({\mathbb A}) \mathrm{GSpin}_{6}(k) {\setminus }
\mathrm{GSpin}_{6}({\mathbb A})} \mathcal{B}_{[2m-5,1^{6}]}(\varphi _{
\tau _{m}})(j_{m}(g_{1})) \mathfrak{E}(\varphi _{\pi _{3}}, \Phi _{12},s;
\omega _{\tau _{m}})(g_{1})\,dg_{1}.
\end{align*}
This inner integral was first introduced by \citet*{GS} and was further studied by \citet*{Ha}. We will write
$E_{12}(\operatorname{pr}(g_{1}) \otimes g_{2})$ to denote
$E_{12}\left ( \operatorname{pr}(g_{1}) \otimes g_{2}, \right.\left.\frac{2s+1}{4},
\omega _{\tau _{m}}; \Phi _{12} \right )$, and likewise for
$F_{12}(\operatorname{pr}(g_{1}) \otimes g_{2})$. Then assuming
$\operatorname{Re}(s) \gg 0$, from the double coset decomposition of
$P_{12}(k) {\setminus }\mathrm{GL}_{12}(k) /H_{4,3}'(k)$ given in \reftext{Lemma~\ref{lem:_Double_coset_decomp_VTEX1}} we can write the above inner integral \reftext{\eqref{eqn:_Inner_integral_ACS_VTEX1}} as
\begin{align*}
\mathfrak{E}(\varphi _{\pi _{3}}, \Phi _{12},s; \omega _{\tau _{m}})(g_{1})
= \sum _{t=1}^{3} \int \varphi _{\pi _{3}}(g_{2}^{\ast}) \sum _{
\gamma \in H'_{\gamma _{t}}(k) {\setminus }H'_{4,3}(k)} F_{12}(
\gamma _{t} \gamma \operatorname{pr}(g_{1}) \otimes g_{2})\,dg_{2}.
\end{align*}
Moreover, following the same argument in
\citep*[Proposition 3.13]{Ha} we see that the contributions from the summands
represented by $\gamma _{1}$ and $\gamma _{2}$ evaluate to zero such that
\begin{align}
\label{eqn:_inner_integral_2_VTEX1}
\mathfrak{E}(\varphi _{\pi _{3}}, \Phi _{12},s; \omega _{\tau _{m}})(g_{1})
= \int \varphi _{\pi _{3}}(g_{2}^{\ast}) \sum _{\gamma \in H'_{
\gamma _{3}}(k) {\setminus }H'_{4,3}(k)} F_{12}(\gamma _{3} \gamma
\operatorname{pr}(g_{1}) \otimes g_{2})\,dg_{2}.
\end{align}
From the parametrisation of $H'_{\gamma _{3}}(k)$ we can factorise the
summation given above as
\begin{align*}
&\sum _{\gamma} F_{12}(\gamma _{3} \gamma \operatorname{pr}(g_{1})
\otimes g_{2})
=\sum _{\gamma \in (P_{1,3}'(k) \otimes \mathrm{GL}_{3}(k))
{\setminus }H'_{4,3}(k)} \sum _{\gamma ' \in H'_{\gamma _{3}}(k) {
\setminus }(P_{1,3}'(k) \otimes \mathrm{GL}_{3}(k))} F_{12}(\gamma _{3}\gamma ' \gamma \operatorname{pr}(g_{1}) \otimes g_{2}),
\end{align*}
where $P_{1,3}'(k) = P_{1,3}(k) \cap \mathrm{GL}_{4}'(k)$. Furthermore,
observe that representatives $\gamma $ of
$(P_{1,3}'(k) \otimes \mathrm{GL}_{3}(k)){\setminus }H'_{4,3}(k)$ can be
written as $\gamma = \eta _{1} \otimes I_{3}$ for
$\eta _{1} \in P_{1,3}'(k) {\setminus }\mathrm{GL}_{4}'(k)$ and representatives
$\gamma '$ of
$H'_{\gamma _{3}}(k) {\setminus }(P_{1,3}'(k) \otimes \mathrm{GL}_{3}(k))$
can be written as
$\gamma ' =
\begin{pmatrix}
\det ^{-1} (M) &
\\
& M
\end{pmatrix}
\otimes I_{3}$ for $M \in Z_{3}(k) {\setminus }\mathrm{GL}_{3}(k)$. Therefore,
\reftext{\eqref{eqn:_inner_integral_2_VTEX1}} becomes
\begin{align*}
\int \varphi _{\pi _{3}}(g_{2}^{\ast}) \sum _{\eta _{1} \in P_{1,3}'(k)
{\setminus }\mathrm{GL}_{4}'(k)} \sum _{M \in Z_{3}(k) {\setminus }
\mathrm{GL}_{3}(k)} F_{12}(\gamma _{3}(
\begin{pmatrix}
\det ^{-1} (M) &
\\
& M
\end{pmatrix}\eta _{1}\operatorname{pr}(g_{1}) \otimes g_{2}))\,dg_{2}.
\end{align*}
By the quasi-invariance of $F_{12}$, we see that for
$M \in \mathrm{GL}_{3}(k)$ we have
\begin{align*}
F_{12}(\gamma _{3}(
\begin{pmatrix}
\det ^{-1} (M) &
\\
& M
\end{pmatrix}\eta _{1}\operatorname{pr}(g_{1}) \otimes g_{2})) = F_{12}(
\gamma _{3}(\eta _{1}\operatorname{pr}(g_{1}) \otimes (M^{\ast})^{-1} g_{2})),
\end{align*}
such that switching the order of the $dg_{2}$-integration with the
$\eta _{1}$-summation, followed by collapsing the $M$-summation with the
integral we obtain
\begin{align*}
\mathfrak{E}(\varphi _{\pi _{3}}, \Phi _{12},s; \omega _{\tau _{m}})(g_{1})
= \sum _{\eta _{1} \in P_{1,3}'(k) {\setminus }\mathrm{GL}_{4}'(k)}
\xi (\varphi _{\pi _{3}}, \Phi _{12},s; \omega _{\tau _{m}})(\widetilde \eta _{1}
g_1),
\end{align*}
where $\widetilde \eta_1 \in \operatorname{GSpin}_6(k)$ is a lift satisfying $\operatorname{pr}(\widetilde \eta_1) = \eta_1$. Thus,
\begin{align*}
\xi (\varphi _{\pi _{3}}, \Phi _{12},s; \omega _{\tau _{m}})(g_{1}) =
\int _{Z_{3}({\mathbb A}) {\setminus }\mathrm{GL}_{3}({\mathbb A})}
\varphi _{\pi _{3}}(g_{2}^{\ast}) F_{12}(\gamma _{3}(
\operatorname{pr}(g_{1}) \otimes g_{2}))\,dg_{2}.
\end{align*}
Following the matrix group isomorphism for $\mathrm{GSpin}_{6}$ given in
Section~\ref{sec:_Accidental_Isomorphism_of_low_rank_GSpin_groups_VTEX1}, we
can regard
$\xi (\varphi _{\pi _{3}}, \Phi _{12},s; \omega _{\tau _{m}})$ as an element
of $\rho _{\pi _{3},s}$ defined in \reftext{\eqref{eqn:_ACS_induced_space_VTEX1}}, up to
suitable scaling of the complex parameter $s$.

\subsection{Relation with Ginzburg's Rankin-Selberg integral}
\label{sec7.2}

In this subsection, we will show that the $\mathrm{GSO}_{10}$-period integral
${\mathcal P}_{\mathcal{D},D_{5}}(\varphi ,s;\Phi )$ given in Section~\ref{sec:_Period_Integrals_on_Special_Orthogonal_Similitude_Groups_VTEX1} is
equivalent to the Bessel-type $\mathrm{GSO}_{10}$-period integral proposed
by \citet*{Gin95a}. More specifically, we will show that these two period
integrals can be regarded as (adelic version) elements of the same
$\operatorname{Hom}$ space. The argument below was communicated by Wee
Teck Gan. Any mistakes are the author's own.

For convenience, we will consider the isometry group $\mathrm{SO}$ instead
of the similitude group $\mathrm{GSO}$. Let $V$ be a split $10$-dimensional
quadratic space over $F$ and let
$\iota : \mathrm{SL}_{2} \to \mathrm{SO}(V)$ be the homomorphism corresponding
to the nilpotent orbit $[4^{2},1^{2}]$. Write $(e,h,f)$ for the associated $\mathfrak{sl}_2$-triple. We have the decomposition
as an $\mathrm{SL}_{2}$-module: $V = W \oplus W' \oplus V_{0}$ where
$W \cong W'$ is an irreducible symplectic 4-dimensional representation
and $V_{0}$ affords two copies of the trivial representation. We remark
that as an $\mathrm{SO}(V_{0})$-module, $V_{0}$ is reducible. Then, we
consider the ordered basis
\begin{align*}
\{e_{3}, e_{3}', e_{1}, e_{1}', v_{0}, v_{0}', e_{-1}, e_{-1}', e_{-3},
e_{-3}'\}
\end{align*}
for $V$ consisting of eigenvectors of
$h
\in \mathrm{Lie}(\mathrm{SO}(V))$ where the subscripts indicate their eigenvalues,
and $V_{0} =\langle v_{0}, v_{0}' \rangle $. Furthermore, for
$i \in \{\pm 1, \pm 3\}$ we denote $X_{i}$ to be the two-dimensional space
generated by $e_{i}, e_{i}'$. By the adjoint action of $h$, we have a grading
on $\mathrm{Lie}(\mathrm{SO}(V))$ which affords a parabolic subalgebra
$\mathfrak p = \mathfrak m \oplus \mathfrak u$ where
$\mathfrak u = \bigoplus _{i\geq 1} \mathfrak u_{i}$ is its nilpotent radical
containing the subalgebra
$\mathfrak u^{+} = \bigoplus _{i\geq 2} \mathfrak u_{i}$. Furthermore,
using the ordered basis above, the Levi subalgebra is
$\mathfrak m = \mathfrak{gl}(X_{3}) {\times }\mathfrak{gl}(X_{1}) {
\times }\mathfrak{so}(V_{0})$ and the subspace $\mathfrak{u}_{1}$ is given
by
$\mathfrak{u}_{1} \cong \operatorname{Hom}(V_{0}, X_{1}) \cong V_{0}
\otimes X_{1}$. Thus, if we consider the corresponding groups
$P = M U \supset U \supset U^{+}$, we have $P$ being the parabolic subgroup
of $\mathrm{SO}(V)$ stabilising the flag
$X_{3} \subset X_{3} \oplus X_{1} \subset X_{3} \oplus X_{1} \oplus
\langle v_{0} \rangle $ since $V_{0}$ is reducible as
$\mathrm{SO}(V_{0})$-module, such that $P$ is contained in the Siegel parabolic
of $\mathrm{SO}(V)$ stabilising the maximal isotropic subspace
$X_{3} \oplus X_{1} \oplus \langle v_{0} \rangle $. Given that
$\mathfrak{u}_{1} = V_{0} \otimes X_{1}$ admits a symplectic structure
afforded by
$f
\in \mathrm{Lie}(\mathrm{SO}(V))$ and $V_{0}$ being quadratic, the space
$X_{1}$ admits a symplectic structure. Now, given non-trivial character
$\psi : F \to {\mathbb C}^{{\times }}$, consider the character
$\psi _{f} : U^{+} \to {\mathbb C}^{{\times }}$ arising from $f$. Setting $N = \{\operatorname{exp}(Y): Y \in \mathfrak u^+, \operatorname{tr}(fY) =0\} \subset \ker \psi_f$, We have
$U/N$ being the Heisenberg group
\begin{align*}
U/N = {\mathcal H}(V_{0} \otimes X_{1}) = (V_{0}
\otimes X_{1}) \oplus (U^{+}/N),
\end{align*}
where $\psi _{f}$ is a non-trivial central character of
${\mathcal H}(V_{0} \otimes X_{1})$ and its stabiliser in
$M = \mathrm{GL}(X_{3}) {\times }\mathrm{GL}(X_{1}) {\times }
\mathrm{SO}(V_{0})$ is
$\mathrm{Sp}(X_{1}) {\times }\mathrm{SO}(V_{0})$. Hence, with these notations and the hyperspherical data (see Section~\ref{sec:_hyperspherical_varieties_VTEX1}):
\begin{align*}
&\iota : H {\times }\mathrm{SL}_{2} \to \mathrm{SO}(V), &&(\rho _{H}, S)
=0,
\end{align*}
its quantization is the smooth $\mathrm{SO}(V)$-module
$\Sigma = \mathrm{Ind}_{H(F) U(F)}^{\mathrm{SO}(V)(F)} \omega _{\psi _{f}}$,
where $\omega _{\psi _{f}}$ is the Weil-representation of
$\widetilde{\mathrm{Sp}}(V_{0} \otimes X_{1}) {\mathcal H}(V_{0}
\otimes X_{1})$ pulled back to $H(F) U(F)$. With these, given smooth irreducible
representation $\pi $ of $\mathrm{SO}(V)$ we have
\begin{align}
\label{eqn:_Isom_GSO10_no.1_VTEX1}
\operatorname{Hom}_{\mathrm{SO}(V)}(\pi , \Sigma ) \cong &
\operatorname{Hom}_{H(F) U(F)}(\pi \otimes \omega _{\psi _{f}}^{\vee},
{\mathbb C}),
\end{align}
where $\omega _{\psi _{f}}^{\vee}$ is its contragredient representation.
We regard the period integral ${\mathcal P}_{\mathcal{D},D_{5}}$ as an
(adelic version) element of the last Hom-space. On the other hand, to relate
to the Bessel-type $\mathrm{GSO}_{10}$ integral in \citep*{Gin95a} we first
observe that the action of
$H = \mathrm{SO}(V_{0}) {\times }\mathrm{Sp}(X_{1})$ on
$V_{0} \otimes X_{1}$ is polarisable, i.e. it stabilises the Lagrangian
decomposition of $V_{0} \otimes X_{1}$:
\begin{align*}
V_{0} \otimes X_{1} = (\langle v_{0} \rangle \otimes X_{1}) \oplus (
\langle v_{0}' \rangle \otimes X_{1}).
\end{align*}
Furthermore, the image of
$H = \mathrm{SO}(V_{0}) {\times }\mathrm{Sp}(X_{1})$ in
$\mathrm{Sp}(V_{0} \otimes X_{1})$ is contained in the Levi subgroup of
the Siegel parabolic stabilising
$\langle v_{0} \rangle \otimes X_{1}$. For $h = (\operatorname{diag}(a,a^{-1}), g) \in H$, we define the character $\chi_H(h) = |a|$, then
\begin{align*}
\omega _{\psi _{f}}\vert _{H(F)U(F)} \cong \mathrm{Ind}_{H(F) U'(F)}^{H(F)
U(F)} (\chi_H\psi _{f}),
\end{align*}
as representations of $H(F) U(F)$, for $U' \supset U^{+}$ is such that
$U'/N$ is a maximal abelian group of the Heisenberg
group $U/N$. Moreover, up to a certain choice of basis
we have $U' = ZS$ where $Z$ and $S$ are unipotent subgroups defined in
\citep*[\S 2.2]{Gin95a}. From these, we see the isomorphism in \reftext{\eqref{eqn:_Isom_GSO10_no.1_VTEX1}} becomes
\begin{align*}
\operatorname{Hom}_{H(F) U(F)}(\pi \otimes \omega _{\psi _{f}}^{\vee},
{\mathbb C}) \cong \operatorname{Hom}_{\mathrm{SO}(V)}(\pi , \Sigma )
\cong \operatorname{Hom}_{H(F) U'(F)}(\pi \otimes \psi _{f}^{-1} \otimes \chi_H^{-1}, {
\mathbb C}),
\end{align*}
where the Bessel-type $\mathrm{GSO}_{10}$ integral in \citep*{Gin95a} is
regarded as an (adelic version) element of the last Hom-space. Consequently,
this chain of isomorphisms shows that both $\mathrm{GSO}_{10}$-period integrals
proposed by \citet*{MWZ} and \citet*{Gin95a} are equivalent.

\appendix
\section{Multiplicity free symplectic representations}
\label{sec:_Appendix_A_multiplicity_free_symplectic_representations_VTEX1}

In this section, we will present all multiplicity free symplectic representations
$\rho $ that admit a polarisation, classified by \citet*{Knop} and
\citet*{Losev} independently. We will denote
$T(\tau ) = \tau \oplus \tau ^{\vee}$ for any representation $\tau $.

The following two tables below summarise all multiplicity free polarised
symplectic representations presented in
\citep*[Table 2, 22 and S]{Knop} as well as
\citep*[Tables 5-8,17-20 and 27]{MWZ} where their generic stabilisers are
connected. Here, we follow the same labelling convention presented in
\citep*{Knop} and \citep*{MWZ}.
Specifically, \reftext{Table~\ref{Table:_Table_2_VTEX1}} contains all multiplicity free polarised symplectic
representations of the form $\tau \oplus \tau ^{\vee}$ whereas in \reftext{Table~\ref{Table:_Table_22_+_S_VTEX1}} it contains such representations of the form
$(\tau _{1} \oplus \tau _{1}^{\vee}) \oplus (\tau _{2} \oplus \tau _{2}^{
\vee})$. In the third column of each table, we have included the strongly
tempered hyperspherical quadruple proposed by \citet*{MWZ} for each of
these representations. In particular, the hyperspherical quadruples
$\mathcal{D}_{A+B}$ in \reftext{Table~\ref{Table:_Table_22_+_S_VTEX1}} for
$A,B \in \{S10, S11, S14\}$, are obtained by the gluing of certain hyperspherical
varieties with $A_{1}$-components (see \citep*[\S 9]{MWZ} for
more details). They are
\begin{align*}
\mathcal{D}_{S10 +S11} =&
\everymath{\displaystyle}\begin{cases}
(\underline{\mathrm{GL}}_{2} {\times }\mathrm{GL}_{2}, \mathrm{GL}_{2}
{\times }\mathrm{GL}_{2}, T(\mathrm{std}_{\mathrm{GL}_{2}}) \oplus T(
\mathrm{std}_{\mathrm{GL}_{2}} \otimes \mathrm{std}_{\mathrm{GL}_{2}}),
1) & \text{if }n=2%
\\
\noalign{\vspace{3pt}}
(\underline{\mathrm{GL}}_{2} {\times }\mathrm{GL}_{n}, \mathrm{GL}_{2}
{\times }\mathrm{GL}_{2}, T(\mathrm{std}_{\mathrm{GL}_{2}} \otimes
\mathrm{std}_{\mathrm{GL}_{2}}), (1,[n-2,1^{2}]))
&\text{if }n\geq 3,
\end{cases}
\\
\mathcal{D}_{S10 + S14}=& (\underline{\mathrm{GL}}_{2} {\times }
\mathrm{GSpin}_{2n+1}, G(\mathrm{SL}_{2}{\times }\mathrm{SL}_{2}) {
\times }\mathrm{GL}_{2}, T(\mathrm{std}_{\mathrm{GL}_{2}} \otimes
\mathrm{std}_{\mathrm{GL}_{2}}) \oplus T(\mathrm{std}_{\mathrm{GL}_{2}}), (1,[2n-3,1^{4}])),
\\
\mathcal{D}_{S11 + S11}=&
\everymath{\displaystyle}\begin{cases}
(\mathrm{GL}_{2} {\times }\underline{\mathrm{GL}}_{2} {\times }
\mathrm{GL}_{2}, S(\mathrm{GL}_{2}^{3}),
T(\mathrm{std}_{\mathrm{GL}_{2}})
\oplus T(\mathrm{std}_{\mathrm{GL}_{2}}) \oplus \mathrm{std}_{
\mathrm{GL}_{2}}^{\otimes 3}, 1),
&\text{if }m=n=2,%
\\
(\mathrm{GL}_{2} {\times }\underline{\mathrm{GL}}_{2} {\times }
\mathrm{GL}_{n}, S(\mathrm{GL}_{2}^{3}),
T(\mathrm{std}_{\mathrm{GL}_{2}})
\oplus \mathrm{std}_{\mathrm{GL}_{2}}^{\otimes 3}, (1,1,[n-2,1^{2}]))
&\text{if } m=2,n\geq 3,%
\\
(\mathrm{GL}_{m} {\times }\underline{\mathrm{GL}}_{2} {\times }
\mathrm{GL}_{n}, S(\mathrm{GL}_{2}^{3}), \mathrm{std}_{\mathrm{GL}_{2}}^{
\otimes 3},
 ([m-2,1^{2}],1,[n-2,1^{2}]))
& \text{if }m,n\geq 3,
\end{cases}
\\
\mathcal{D}_{S11 + S14}=&
\begin{cases}
(\mathrm{GL}_{2} {\times }\underline{\mathrm{GL}}_{2} {\times }
\mathrm{GSpin}_{2n+1}, S''(\mathrm{GL}_{2}^{4}),
 T(\mathrm{std}_{
\mathrm{GL}_{2}}) \oplus T(\mathrm{std}_{\mathrm{GL}_{2}}) \oplus
\mathrm{std}_{\mathrm{GL}_{2}}^{\otimes 3}, (1,1,[2n-3,1^{4}]))
&\text{if } m=2%
\\
(\mathrm{GL}_{m} {\times }\underline{\mathrm{GL}}_{2} {\times }
\mathrm{GSpin}_{2n+1}, S''(\mathrm{GL}_{2}^{4}),
T(\mathrm{std}_{
\mathrm{GL}_{2}}) \oplus \mathrm{std}_{\mathrm{GL}_{2}}^{\otimes 3}, ([m-2,1^{2}],1,[2n-3,1^{4}]))
&\text{if }m\geq 3,
\end{cases}
\end{align*}
and $\mathcal{D}_{S14 + S14}$ is given by
\begin{align*}
&(\mathrm{GSpin}_{2m+1} {\times }\underline{\mathrm{GL}}_{2} {\times }
\mathrm{GSpin}_{2n+1}, S^{\ast}(\mathrm{GL}_{2}^{5}), T(\mathrm{std}_{
\mathrm{GL}_{2}}) \oplus T(\mathrm{std}_{\mathrm{GL}_{2}}) \oplus
\mathrm{std}_{\mathrm{GL}_{2}}^{\otimes 3},
 ([2m-3,1^{4}],1,[2n-3,1^{4}])).
\end{align*}
In the last column of each table, we included the references for the existing
work on such integral representations that are equivalent to what
\citet*{BZSV} have proposed.

Finally, we remark that there are three families of multiplicity free polarised
symplectic representations given below, where their generic stabilisers
are not connected (hence they do not belong in the current framework of \citet*{BZSV}).
\begin{center}
\begin{tabular}{lll}%
\hline
Label & $(G^{\vee}, \tau )$ & Remark
\\ \hline
(2.3)&$(\mathrm{GL}_{n}, \operatorname{Sym}^{2})$&\citep*{BG1,P-PS,T}\\
(2.7)&$(\mathrm{SO}_{2k+1}, \mathrm{std}_{\mathrm{SO}_{2k+1}})$ &\citep*{GRS,Y1}\\
(2.9)&$(G_{2}, \mathrm{std}_{G_{2}})$&\citep*{G93}\\
\hline
\end{tabular}
\end{center}
Nevertheless, as pointed out by \citet*{MWZ}, the corresponding dual integrals
have been studied in the references listed in the Remark column.

\clearpage
\newgeometry{margin=0.4in}
\pdfpagewidth=680.315pt
\pdfpageheight=467.717pt
\paperwidth=680.315pt
\paperheight=467.717pt
\textwidth=8.65in
\textheight=5.7in
\oddsidemargin=-0.6in
\evensidemargin=-0.6in
\topmargin=-0.6in

\begin{table}[p]
\footnotesize
\tablewidth=8.65in
\caption{Multiplicity free polarised representations in \citet*{Knop}'s Table 2 and hyperspherical data.}
\label{Table:_Table_2_VTEX1}
\tabcolsep=0pt
\resizebox{\tablewidth}{!}{%
\begin{tabular}{@{\hspace{6pt}}llll@{\hspace{6pt}}}
\hline
\textbf{Label} & \textbf{$(G^{\vee}, \tau )$} & \textbf{$\mathcal{D}=(G, H,\rho _{H}, \iota )$} & \textbf{Remark}
\\
\hline
{(2.1): $\begin{cases}
    m=n\geq 2\\
    m>n\geq 2
\end{cases}$}  & $(\mathrm{GL}_{m} {\times }\mathrm{GL}_{n}, \mathrm{std}_{\mathrm{GL}_{m}} \otimes \mathrm{std}_{\mathrm{GL}_{n}})$ & $\begin{cases}
    (\mathrm{GL}_{n} {\times }\mathrm{GL}_{n}, \mathrm{GL}_{n}, T(\mathrm{std}_{\mathrm{GL}_{n}}), 1)\\
    (\mathrm{GL}_{m} {\times }\mathrm{GL}_{n}, \mathrm{GL}_{n}, 0, ([m-n,1^{n}],1))
\end{cases}$ & \citep*{JPSS}\tabnoteref[1]{1}
\\
\noalign{\vspace{3pt}}
{(2.2): $\begin{cases}
    n=2m\geq 4\\
    n=2m+1>4
\end{cases}$} & $(\mathrm{GL}_{n}, \wedge ^{2})$ &$\begin{cases}
    (\mathrm{GL}_{2m}, \mathrm{GL}_{m}, T(\mathrm{std}_{\mathrm{GL}_{m}}), [2^{m}])\\
    (\mathrm{GL}_{2m+1}, \mathrm{GL}_{m}, 0, [2^{m},1])
\end{cases}$ & \citep*{JS}\\
\noalign{\vspace{3pt}}
(2.4): $n\geq 3$ & $(\mathrm{GL}_{n}, \mathrm{std}_{\mathrm{GL}_{n}})$ & $(\mathrm{GL}_{n}, \mathrm{GL}_{1},0,[n-1,1])$ & \citep*{GGP}\\
\noalign{\vspace{3pt}}
(2.5): $m\geq 2$ & $(\mathrm{Sp}_{2m}, \mathrm{std}_{\mathrm{Sp}_{2m}})$ & $(\mathrm{SO}_{2m+1}, \mathrm{SO}_{2}, 0, [2m-1,1^{2}])$ & \citep*{N95}
\\
\noalign{\vspace{3pt}}
{(2.6): $\begin{cases}
    m\geq 2,n=2\\
    m=2,n=4\\
    m=2,n\geq 5\\
    m=2, n=3\\
    m\geq 3,n=3
\end{cases}$} & $(\mathrm{GSp}_{2m} {\times }\mathrm{GL}_{n}, \mathrm{std}_{\mathrm{GSp}_{2m}} \otimes \mathrm{std}_{\mathrm{GL}_{n}})$ & $\begin{cases}
    (\mathrm{GSpin}_{2m+1} {\times }\mathrm{GL}_{2}, \mathrm{GSpin}_{4}, T(\mathrm{std}_{\mathrm{GL}_{2}}), [2m-3,1^{4}])\\
    \hline
    (\mathrm{GSp}_{4} {\times }\mathrm{GL}_{4}, S'(\mathrm{GSp}_{4} {\times }\mathrm{GL}_{4}), \mathrm{std}_{\mathrm{GSp}_{4}} \otimes \wedge ^{2} \mathrm{std}_{\mathrm{GL}_{4}} \oplus T(\mathrm{std}_{\mathrm{GL}_{4}}),1)\\
    (\mathrm{GSp}_{4} {\times }\mathrm{GL}_{n},
    S'(\mathrm{GSp}_{4} {\times }\mathrm{GL}_{4}),
    \mathrm{std}_{\mathrm{GSp}_{4}} \otimes \wedge ^{2} \mathrm{std}_{\mathrm{GL}_{4}},
    (1,[n-4,1^{4}]))\\
    \hline
    (\mathrm{GSp}_{4} {\times }\mathrm{GL}_{3}, G, T(\mathrm{std}_{\mathrm{GSp}_{4}} \otimes \mathrm{std}_{\mathrm{GL}_{3}}),1)\\
    (\mathrm{GSpin}_{2m+1} {\times }\mathrm{GL}_{3}, \mathrm{GSpin}_{6} {\times }\mathrm{GL}_{3}, T(\mathrm{HSpin}_{6} \otimes \mathrm{std}_{\mathrm{GL}_{3}}), ([2m-5,1^{6}],1))
\end{cases}$ & \begin{tabular}{@{}c@{}}
    \citep*{N,S}\tabnoteref[2]{2}\\
    {-}\\
\citep*{ACS}
\end{tabular}\\
\noalign{\vspace{3pt}}
{(2.7): $2n\geq 4$} & $(\mathrm{GSO}_{2n}, \mathrm{std}_{\mathrm{GSO}_{2n}})$ &$(\mathrm{GSpin}_{2n}, \mathrm{GSpin}_{3}, T(\mathrm{std}_{\mathrm{GL}_{2}}), [2n-3,1^{3}])$& \citep*{E,GGP}
\\
\noalign{\vspace{3pt}}
{(2.8): $n \in \{7,9,10\}$} & $(\mathrm{GSpin}_{n}, \mathrm{Spin}_{n})$ &
$\begin{cases}
    (\mathrm{GSp}_{6}, \mathrm{GL}_{2}, T(\mathrm{std}_{\mathrm{GL}_{2}}), [3^{2}])\\
    (\mathrm{GSp}_{8}, G(\mathrm{SL}_{2} {\times }\mathrm{SL}_{2}), T(\mathrm{std}_{\mathrm{GL}_{2},2}), [3^{2},1^{2}])\\
    (\mathrm{GSO}_{10}, \mathrm{GL}_{2}, 0, [4^{2},1^{2}])
\end{cases}$
& \begin{tabular}{@{}c@{}}
    \citep*{BG}\\
    \citep*{BG}\\
    \citep*{Gin95a}\tabnoteref[3]{3}
\end{tabular}
\\
\noalign{\vspace{3pt}}
(2.10) & $(E_{6}, \mathrm{std}_{E_{6}})$ & $(E_{6}, \mathrm{GL}_{3}, T(\mathrm{std}_{\mathrm{GL}_{3}}), D_{4})$ & \citep*{Gin95}\\
\hline
\end{tabular}
}
\tabnotetext[1]{1}{As pointed out by \citet*[Example 4.3.12]{BZSV}, the
integral of \citet*{JPSS} is not derived from the hyperspherical
presentation, but corresponding cotangent bundles are equivalent.}
\tabnotetext[2]{2}{The integral representation for $\mathrm{GSp}_{4} {\times
}\mathrm{GL}_{2}$ was first considered by \citet*{N}, and the local theory of
this Rankin-Selberg integral is studied by \citet*{S}.}
\tabnotetext[3]{3}{\citet*{Gin95a} has constructed a Bessel-type
Rankin-Selberg integral which is equivalent to that derived by hyperspherical
presentation.}
\end{table}

\clearpage

\begin{table}[p]
\footnotesize
\tablewidth=8.65in
\caption{Multiplicity free polarised representations in \citet*{Knop}'s Table 22 \& S and hyperspherical data.}
\label{Table:_Table_22_+_S_VTEX1}
\tabcolsep=0pt
\resizebox{\tablewidth}{!}{%
\begin{tabular}{@{\hspace{6pt}}llll@{\hspace{6pt}}}
\hline
Label & $(G^{\vee}, \tau _{1} \oplus \tau _{2})$ & $\mathcal{D}=(G, H,\rho _{H}, \iota )$ & Remark
\\
\hline
(22.1) & $(\mathrm{Spin}_{8}, \mathrm{std}_{\mathrm{Spin}_{8}} \oplus \mathrm{Spin}_{8})$ &
$(\mathrm{GSO}_{8}, S(\mathrm{GL}_{2} {\times }\mathrm{GSO}_{4}), T(\mathrm{std}_{2}) \oplus T(\mathrm{std}_{\mathrm{GL}_{2}}), [2^{2},1^{4}])$
&-\\
\noalign{\vspace{3pt}}
(22.2) &$(\mathrm{GL}_{n}, \wedge ^{2} \mathrm{std}_{\mathrm{GL}_{n}} \oplus \mathrm{std}_{\mathrm{GL}_{n}})$& $(\mathrm{GL}_{n}, \mathrm{GL}_{\lfloor \frac{n}{2} \rfloor} {\times }\mathrm{GL}_{\lceil \frac{n}{2} \rceil}, T(\mathrm{std}_{\mathrm{GL}_{\lceil \frac{n}{2} \rceil}}),1)$  & \citep*{BF}\\
\noalign{\vspace{3pt}}
(22.3): $\begin{cases}
    m=n\geq 3\\
    m \geq n+1 \geq 4\\
    n-1=m \geq 2\\
    2 \leq m \leq n-2
\end{cases}$
&$(\mathrm{GL}_{m} {\times }\mathrm{GL}_{n}, \mathrm{std}_{\mathrm{GL}_{m}} \otimes \mathrm{std}_{\mathrm{GL}_{n}} \oplus \mathrm{std}_{\mathrm{GL}_{n}})$&
$\begin{cases}
    (\mathrm{GL}_{n}{\times }\mathrm{GL}_{n}, G, T(\mathrm{std}_{\mathrm{GL}_{n}} \otimes \mathrm{std}_{\mathrm{GL}_{n}}) \oplus T(\mathrm{std}_{\mathrm{GL}_{n}}),1)\\
\hline
    (\mathrm{GL}_{m}{\times }\mathrm{GL}_{n}, \mathrm{GL}_{n} {\times }\mathrm{GL}_{n}, T(\mathrm{std}_{\mathrm{GL}_{n}} \otimes \mathrm{std}_{\mathrm{GL}_{n}}), ([m-n, 1^{n}],1))\\
\hline
    (\mathrm{GL}_{n-1}{\times }\mathrm{GL}_{n}, G, T(\mathrm{std}_{\mathrm{GL}_{n-1}} \otimes \mathrm{std}_{\mathrm{GL}_{n}}) \oplus T(\mathrm{std}_{\mathrm{GL}_{n}}),1)\\
\hline
(\mathrm{GL}_{m} {\times }\mathrm{GL}_{n}, \mathrm{GL}_{m} {\times }\mathrm{GL}_{m+1},
T(\mathrm{std}_{\mathrm{GL}_{m}} \otimes \mathrm{std}_{\mathrm{GL}_{m+1}}),
\\
 (1,[n-m-1,1^{m+1}]))
\end{cases}$
&\citep*{Sak,GS,Ha}\tabnoteref[4]{4}\\
\noalign{\vspace{3pt}}
(22.4): $\begin{cases}
    n=2\\
    n>2
\end{cases}$ &$(\mathrm{GL}_{n}, \mathrm{std}_{\mathrm{GL}_{n}} \oplus \mathrm{std}_{\mathrm{GL}_{n}})$&   $\begin{cases}
    (\mathrm{GL}_{2}, \mathrm{GL}_{2}, T(\mathrm{std}_{\mathrm{GL}_{2}}) \oplus T(\mathrm{std}_{\mathrm{GL}_{2}}),1)\\
    (\mathrm{GL}_{n}, \mathrm{GL}_{2},  T(\mathrm{std}_{\mathrm{GL}_{2}}), [n-2,1^{2}])
\end{cases}$&- \\
\noalign{\vspace{3pt}}
(22.5)&$(\mathrm{GSp}_{2n}, \mathrm{std}_{\mathrm{GSp}_{2n}} \oplus \mathrm{std}_{\mathrm{GSp}_{2n}})$&   $(\mathrm{GSpin}_{2n+1}, \mathrm{GSpin}_{4}, T(\mathrm{std}_{\mathrm{GL}_{2}}) \oplus T(\mathrm{std}_{\mathrm{GL}_{2}}), [2n-3,1^{4}])$&- \\
\noalign{\vspace{3pt}}
(S.10)+(S.11)  &$(\underline{\mathrm{GL}}_{2} {\times }\mathrm{GL}_{n},\underline{\mathrm{std}}_{\mathrm{GL}_{2}}\oplus \underline{\mathrm{std}}_{\mathrm{GL}_{2}} \otimes \mathrm{std}_{\mathrm{GL}_{n}})$&$\mathcal D_{S10+S11}$&-\\
\noalign{\vspace{3pt}}
(S.10)+(S.14)  &$(\underline{\mathrm{GL}}_{2} {\times }\mathrm{GSp}_{2n},\underline{\mathrm{std}}_{\mathrm{GL}_{2}}\oplus \underline{\mathrm{std}}_{\mathrm{GL}_{2}} \otimes \mathrm{std}_{\mathrm{GSp}_{2n}})$&$\mathcal D_{S10+S14}$&-\\
\noalign{\vspace{3pt}}
(S.11)+(S.11)  &$\left (\mathrm{GL}_{m} {\times }\underline{\mathrm{GL}}_{2} {\times }\mathrm{GL}_{n},
\begin{matrix}
    \mathrm{std}_{\mathrm{GL}_{m}} \otimes \underline{\mathrm{std}}_{\mathrm{GL}_{2}}\\
    \oplus \\
    \underline{\mathrm{std}}_{\mathrm{GL}_{2}} \otimes \mathrm{std}_{\mathrm{GL}_{n}}
\end{matrix}
\right )$&$\mathcal D_{S11+S11}$&-\\
\noalign{\vspace{3pt}}
(S.11)+(S.14) &$\left (\mathrm{GL}_{m} {\times }\underline{\mathrm{GL}}_{2} {\times }\mathrm{GSp}_{2n},
\begin{matrix}
    \mathrm{std}_{\mathrm{GL}_{m}} \otimes \underline{\mathrm{std}}_{\mathrm{GL}_{2}}\\
    \oplus \\
    \underline{\mathrm{std}}_{\mathrm{GL}_{2}} \otimes \mathrm{std}_{\mathrm{GSp}_{2n}}
\end{matrix}
\right )$&$\mathcal D_{S11+S14}$&-\\
\noalign{\vspace{3pt}}
(S.14)+(S.14) &$\left (\mathrm{GSp}_{2m} {\times }\underline{\mathrm{GL}}_{2} {\times }\mathrm{GSp}_{2n}, \begin{matrix}
    \mathrm{std}_{\mathrm{GSp}_{2m}}\otimes \underline{\mathrm{std}}_{\mathrm{GL}_{2}}\\
    \oplus \\
     \underline{\mathrm{std}}_{\mathrm{GL}_{2}} \otimes \mathrm{std}_{\mathrm{GSp}_{2n}}
\end{matrix}\right )$&$\mathcal D_{S14+S14}$&-\\
\hline
\end{tabular}
}
\tabnotetext[4]{4}{\citet*{GS} have constructed the integral representation
for the case $(n,n)$ and \citet*{Sak} has constructed the integral
representation for the case $(m,n) \in \{(n,n), (n-1,n)\}$.}
\end{table}

\clearpage
\pdfpagewidth=612pt
\pdfpageheight=792pt
\paperwidth=612pt
\paperheight=792pt
\restoregeometry

\end{document}